\documentclass[11pt,a4paper,reqno]{amsart}
\usepackage[applemac]{inputenc}
\usepackage[T1]{fontenc}
\usepackage{amsmath}
\usepackage{amsthm}
\usepackage{amsfonts}
\usepackage{amssymb}
\usepackage{graphicx}
\usepackage{color}
\usepackage{amsbsy}
\usepackage{mathrsfs}
\usepackage{bbm}
\usepackage{amsbsy}

\newtheorem{theorem}{Theorem}[section]
\newtheorem{lemma}[theorem]{Lemma}

\newtheorem{proposition}[theorem]{Proposition}

\newtheorem{definition}[theorem]{Definition}
\theoremstyle{remark}
\newtheorem{remark}[theorem]{\it \bf{Remark}\/}

\numberwithin{equation}{section}
\catcode`@=11
\def\section{\@startsection{section}{1}%
  \z@{1.5\linespacing\@plus\linespacing}{.5\linespacing}%
  {\normalfont\bfseries\large\centering}}
\catcode`@=12
\newcommand{\be}{\begin{equation}}
\newcommand{\ee}{\end{equation}}
\newcommand{\bea}{\begin{eqnarray}}
\newcommand{\eea}{\end{eqnarray}}
\newcommand{\bee}{\begin{eqnarray*}}
\newcommand{\eee}{\end{eqnarray*}}
\newcommand{\bear}{\begin{array}{l}}
\newcommand{\ear}{\end{array}}

\def\pa{\partial}

\def\non{\nonumber}

\def\log{\ln}

\def\RR{\mathbb{R}}

\def\eps{\vare}
\def\wt{\tilde{w}}

\def\eps{\epsilon}

\catcode`@=11
\def\supess{\mathop{\operator@font Sup\,ess}}
\catcode`@=12

\def\RR{\mathbb{R}}

\def\H{{\mathcal H}}

\def\e{\varepsilon}

\def\R2+{\RR ^2_+}

\def\A{\mathcal A}

\def\pa{\partial}

\def\lim{\mathop{\rm lim}}

\def\sup{\mathop{\rm sup}}

\def\l{\lambda}

\def\LR{\Leftrightarrow}
\def\log{{\rm log}}

\def\ut{\tilde{u}}

\def\noi{\noindent}

\def\pa{\partial}

\def\Psit{\tilde{\Psi}}

\def\Lamdba{\Lambda}

\def\pa{\partial}
\def\la{\langle}
\def\matchal{\mathcal}
\def\ra{\rangle}

\def\noi{\noindent}
\def\und{\underline}

\def\nut{\tilde{\nu}}

\def\bear{\begin{array}{l}}
\def\ear{\end{array}}
\def\psit{\tilde{\psi}}

\def\phit{\tilde{\phi}}

\def\C{\mathcal C}
\def\R{\mathcal R}
\def\ve{{\bf e}}

\def\wt{\tilde{w}}

\def\A{\mathcal A}

\begin{document}

\title[]{On smooth inviscid vortices with fat tails}

\author[P. Rapha\"el]{Pierre Rapha\"el}
\address{Department of Pure Mathematics and Mathematical Statistics, Cambridge, UK}

\author[A. R. Sa\"id]{Ayman R. Sa\"id}
\address{CNRS, Laboratoire de Math\'ematiques, Reims, France}

\begin{abstract}  We derive a family of degenerate vortices with non trivial swirls and fat tails which can be used to bifurcate both stationary and self similar solutions to the three dimensional axi-symmetric incompressible Euler equations. \end{abstract}

\maketitle

\section{Introduction}

We consider the three dimensional incompressible Euler equations
\begin{equation}
\label{eulerincomp}
\left|\begin{array}{l}
\partial_tu+u\cdot\nabla u+\nabla p=0\\
\nabla \cdot u=0\\
(t,x)\in \Bbb R\times \Bbb R^3, \ \ u(t,x)\in \Bbb R^3
\end{array}\right.
\end{equation}
and address the classical problem of constructing stationary or self similar solutions.


\subsection{The axisymmetric problem}  


Under the additional assumption of cylindrical symmetry, the velocity field can be expressed in the cylindrical basis $(\ve_r,\ve_\phi,\ve_z)$ in terms of the swirl and stream functions
\be
\label{aysitniflow}
u=\frac{1}{r}\nabla^\perp \psi+u_\phi\ve_\phi, \ \ \nabla^\perp\psi=-\pa_z\psi\ve_r+\pa_r\psi\ve_z.
\ee

\noi\und{Vortex problem}. For the stationary equation, the Bernoulli function and swirl are transported by the flow $$\left|\begin{array}{l}
ru_\phi=\mathcal C(\psi)\\
\Pi=p+\frac{|u|^2}{2}=\mathcal H(\psi)
\end{array}\right.
$$  
for some free non linearities $\mathcal H,\mathcal C$, and the stationary equation reduces to the non linear elliptic equation
\be
\label{elltipicequations}
-\frac 1r\pa_r\left(\frac{1}{r}\pa_r\psi\right)-\frac{\pa_z^2\psi}{r^2}=\frac{\mathcal C(\psi)\matchal C'(\psi)}{r^2}-\mathcal H'(\psi).
\ee

\noi\und{Self similar equation}. Euler has a two parameters scaling symmetry, hence self similar renormalization yields the stationary self similar equation
\be
\label{selfsim}
\left|\bear
\left[\gamma+ (u+x)\cdot \nabla \right]u+\nabla p=0\\
\nabla \cdot u=0
\ear\right.
\ee
for some free parameter $\gamma>-1$. 

\subsection{Previous constructions and related problems} Constructing stationary (or travelling) vortices solutions is a classical problem, self similar solutions is much more delicate.\\

\noindent{\em Known vortices}. Some explicit examples are known \cite{Mo69}. A large class of axisymmetric finite energy $$\int_{\Bbb R^3}|u(x)|^2dx<\infty$$ vortex rings has been constructed in the seminal pioneering work \cite{Tur89}. An intrinsic feature of the  variational method is to produce profiles with compactly supported swirls, and this generates solutions which lose regularity on the corresponding free boundary. The existence proof also usually provides a poor understanding of the structure of the profile of the solution and the associated linearized operator, and hence for example most stability problems in the field are open. Remarkably, smooth vortices with compactly supported velocities are constructed in \cite{Gav19} using the hodograph transform.\\

\noindent{\em Self similar solutions}. The existence of solutions to \eqref{selfsim} with decay at infinity is a major open problem in the field. In the seminal work \cite{El19}, Elghindi constructed a cylindrically symmetric self similar solution with zero swirl  and limited regularity for $\gamma=-1+\e$, see also \cite{El23,Sk}, and we refer to \cite{CV} and references therein for an illuminating discussion on this subject. A fair conclusion is that there are very few examples of solutions to \eqref{selfsim}, and in fact very little understanding of how to construct them.\\
 
 \noindent{\em The tail problem}. We would like to stress an analogy with a simpler problem, the non linear heat equation 
 \be
 \label{vneoineonveonevo}
 \pa_tu=\Delta u +u^p.
 \ee 
 In the energy super critical range $p>\frac{d+2}{d-2}$, $d\ge 3$, there exists a unique radially symmmetric  ground state solution 
 \be
 \label{vnoivneionoenoive}
 \Delta Q+Q^p=0, \ \ Q>0.
 \ee
 An essential feature of this non linear profile is its $\mathcal C^\infty$ smoothness, and its {\em fat self similar tail} at infinity:
 $$
 Q(r)\sim \left[1+o_{r\to+\infty}(1)\right]\Phi_*(r)
 $$
 where 
 \be
 \label{enoenvoinoevin}
 \Phi_*(r)=\frac{c_*(d,p)}{r^{\frac{2}{p-1}}}
 \ee is the homogeneous tail. Note that a peculiar property of $\Phi_*$ is that it is {\em both} a stationary and a self similar solution that is 
 \be
 \label{vieobioeboibve}
 \left|\bear
 \Delta \Phi_*+\Phi_*^p=0\\
 \left(\frac{2}{p-1}+r\pa_r\right)\Phi_*=0.
 \ear\right.
 \ee
 The set of smooth solutions to \eqref{vnoivneionoenoive} in the energy super critical range is poorly understood, mostly because these are infinite energy solutions which are not amenable to variational methods. Understanding which asymptotics tails are generated by smooth solutions is a classical open problem. The simplest way to construct solutions to \eqref{vnoivneionoenoive} is to assume radial symmetry and study the corresponding ode through its Emden transform, a strategy which will inspire our approach to derive solutions to \eqref{elltipicequations}.\\
  
 \noindent{\em Blow up problem}.  The fat self similar tail \eqref{enoenvoinoevin} of {\em stationary solutions} is deeply connected to the existence of {\em smooth self similar solutions} for \eqref{vneoineonveonevo} which can be constructed through a subtle bifurcation process, \cite{MRS20}. In other words, in certain regimes of parameters $(d,p)$, the singular tail solution \eqref{vieobioeboibve} can be used to bifurcate smooth both stationary and self similar solutions. More generally, for other non linear fluid or dispersive problems like in particular the nonlinear Schr\"odinger equation or compressible fluid mechanics, and in connection to singularity formation problems, the works \cite{El19,HV94,MRR15,MRRS19-1,MRRS19-2,MRRS19-3} have demonstrated the importance of {\em infinite energy smooth with fat self similar tail} solutions to the either stationary or self similar equation. These energy super critical non linear bubbles are the key stone of the underlying mechanism of energy concentration and singularity formation.
  

\subsection{Statement of the result}


 Inspired by the above problem, we ask for the existence of suitable profiles at the threshold between stationary versus self similar solutions for \eqref{eulerincomp}. We propose a new  bifurcation scheme which starting point is the {\em anisotropic front renormalization} initiated in \cite{CMR19, MRS20}.\\
  
 The first step is the existence of  {\em degenerate homogeneous stationary vortices} which solve the limiting degenerate front equation.
  
\begin{proposition}[The degenerate homogeneous vortex]
\label{thmdegen}
Let 
\be
\label{gmamintro}
\gamma=2+\frac{1}{m}, \ \ m\in \Bbb N^*.
\ee 
There exists $A\equiv A(\gamma)>0$ such that for the scale invariant homogeneous non linearities 
\be
\label{veonoeinoineonevo}
\left|\begin{array}{l}
\mathcal C_0(\psi)=\frac{\gamma-2}{\gamma -1}\sqrt{A}\psi^{\frac{\gamma-1}{\gamma-2}}\\
\mathcal H_0'(\psi)=\psi^{\frac{\gamma+2}{\gamma-2}},
\end{array}\right.
\ee
the degenerate ode 
\be
\label{degenerateode}
-\frac 1r\pa_r\left(\frac{1}{r}\pa_r\psi\right)=\frac{\mathcal C_0(\psi)\matchal C'_0(\psi)}{r^2}-\mathcal H_0'(\psi)
\ee 
admits a solution $\psi_*(r)$ which is a bell shaped $\mathcal C^\infty$ non negative profile with self similar decay at infinity
$$\left|\begin{array}{l}
\psi_*(r)\sim C(\gamma)\left[1+o_{r\to +\infty}(1)\right]r^{2-\gamma}\\
C(\gamma)>0.
\end{array}\right.
$$
\end{proposition} 

We now claim that, at least in a suitable range of $\gamma$, the degenerate vortex profile can be used to bifurcate {\em both} a solution to the full stationary vortex equation \eqref{elltipicequations}, and an exact  self similar solution to \eqref{selfsim} smooth away from the origin.

\begin{theorem}[Existence of vortices or self similar profiles with fat tail]
\label{thmmain}
Let $m\ge m_*$ large enough, $\gamma$ given by \eqref{gmamintro} and $\psi_*$ be the smooth degenerate homogeneous vortex of Proposition \ref{thmdegen}. Then for all $0<a<a_*(\gamma)$ small enough, the following holds. Let the front renormalization $z=aZ$.\\

\noindent\underline{\em Vortex bifurcation.} There exists a smooth scaling factor $\nu(Z)>0$ and a corrected stream function $\psi_{\rm corr}$ such that $$\psit(r,Z)=\frac{1}{\left[\nu(Z)\right]^{\gamma-2}}\psi_*\left(\frac{r}{\nu(Z)}\right)\left[1+a^2\psi_{\rm cor}(r,Z)\right]
$$ satisfies the following:\\

\noindent {\em (i) Decay in $Z$}: $\nu$ and $\psit$ are even in $Z$ and there holds  $$\la Z\ra^{\eta}\lesssim \nu(Z)\lesssim \la Z\ra^{\eta}\ \ \mbox{for some}\ \ 0<\eta\ll1 .$$ 
{\em (ii) Uniform smallness of the correction}: $$a^2\|\la z\ra^\delta \psi_{\rm cor}\|_{L^\infty}\le \frac 12\ \ \mbox{for some}\ \ \delta>0.$$
{\em (iii) Regularity and equation:} for a suitable deformation $\matchal C$ of the homogeneous non linearity \eqref{veonoeinoineonevo}, the axisymmetric vector field 
$$\left|\bear
u(r,z)\equiv \ut(r,Z)=\frac{\nabla^\perp\psit}{r}+\tilde{u}_\phi\ve_\phi\\
\tilde{u}_\phi=\mathcal C(\psit)
\ear\right.
$$
  is a $\matchal C^\infty$ smooth stationary solution to \eqref{eulerincomp} with fat self similar tail at $+\infty$.\\
  
  \noindent\underline{\em Self similar bifurcation.} There exists a bifurcated profile $\psi_a(R)=\psi_*(R)\left[1+O_R(a)\right]$ such that the cylindrical vector field generated by  the stream function 
  \be
  \label{vneionveoinioenvnev}
  \psi(r,z)=\frac{1}{|aZ|^{\gamma-2}}\psi_a\left(\frac{r}{|Z|}\right)
  \ee 
  with swirl given by the homogeneous non linearity $\matchal C_0$  \eqref{veonoeinoineonevo} is a $\mathcal C^\infty$ smooth solution to the self similar equation \eqref{selfsim} in $\Bbb R^3\backslash \{0\}$ with asymptotic self similar decay.

\end{theorem}

\medskip
\noindent{\em Comments on the result.}\\

\noindent{\em 1. Range of $\gamma$}. The scaling constraint $\gamma>2$ in Theorem \ref{degenerateode} is {\em fundamental} and follows from the nature of the non linearities $\mathcal C_0, \mathcal H_0$ which shapes change dramatically for $\gamma\le 2$. We will see that the profile $\psi_*$ concentrates in a very non linear way as $\gamma \downarrow 2$,  Section \ref{degeneratevortex}, and in particular the associated swirl has a non trivial short scale connection from zero to the self similar tail which is spectacular.\\

\noindent{\em 2. Front renormalization}. The bifurcation process appears through an internal parameter $0<a\ll1 $ in the front renormalization $Z=az$. The case $a=0$ is the non linear ode \eqref{degenerateode} and can be investigated using non linear shooting methods. This strategy of reduction to a simpler model with study of the linearized flow (in the singular limit $a\downarrow 0$) is reminiscent from the works on compressible fluids \cite{MRRS19-1,MRRS19-2,MRRS19-3}. Moreover, as in these works, the $a$ term cannot be neglected globally since it contains the highest number of derivatives in the equation, it needs to push for us. This is particularly important for the construction of the self similar profile where the stabilization mechanism at $+\infty$ and the $\mathcal C^\infty$ regularity through $z=0$, which is not obvious in light of \eqref{vneionveoinioenvnev}, are generated by the $a$ terms.\\

\noindent{\em 3. Local classification}. For the sake of simplicity, we will give the proof of the vortex construction for a special case of deformation of \eqref{veonoeinoineonevo}, and we will explicitly enforce
$$\left|\begin{array}{l}
 \mathcal C\matchal C'(\psi)=A(\gamma)\psi^{\frac{\gamma}{\gamma-2}}\left[1+\frac{a^2}{A(\gamma)}\psi^{m_1}\right]\\
\mathcal H'(\psi)=\psi^{\frac{\gamma+2}{\gamma-2}}\left[1+a^2\psi^{m_2}\right]
  \end{array}\right.
  $$
  for some well chosen $m_1,m_2>0$. Given these deformations, we will construct  $(\nu(z),\psi_{\rm cor})$ using a non linear fixed point argument. Obviously the proof constructs more generally a map $(\matchal C,\mathcal H)\mapsto \psi\equiv(\nu,\psi_{\rm cor})$ locally defined around $(\mathcal C_0,\matchal H_0)$ in a suitable function space\footnote{We refer to \cite{CS12} for a related study of the classification of stationary solutions of 2d Euler in a domain near a given profile.}. As in \cite{CMR19,MRS20}, the understanding of the $\nu$ function which measures decay in $z$ is the heart of the analysis. It is computed as the unique smooth decaying solution to a non linear ode parameterized by ($\mathcal H,\mathcal C)$, see \eqref{equaitonoineog}, which corresponds to the projection of the flow onto the scaling instability of the linearized operator close to $\psi_*$.\\  
    
 \noindent{\em 4. Vortex bifurcation}. The classical difficulty in the study of \eqref{eulerincomp} is that even in axisymmetry, the stationary problem is a PDE, and the solutions of Theorem \ref{thmmain} arise with infinite energy and thus are not amenable to variational methods. A canonical advantage of the bifurcation approach with respect to abstract variational methods is to produce profiles whose shape is explicit, and the construction comes with a complete understanding of the associated linearized operator.\\
 
 \noindent{\em 5. Self similar and stationary solution}. A spectacular feature of the profile \eqref{vneionveoinioenvnev} is that, exactly like the profile $\Phi_*$ given by \eqref{vneionveoinioenvnev} for the non linear heat equation, it is {\em both} a stationary and a self similar solution, in the sense that the associated velocity field kills scaling 
 \be
 \label{veionoienoinoevnoive}
 \left(\gamma+x\cdot\nabla\right)u=0.
 \ee
 
  \noindent{\em 5. Self similar bifurcation}. The self similar velocity field associated to \eqref{vneionveoinioenvnev} has an essential singularity at $r=z=0$ where it diverges, and this is necessary in cylindrical symmetry in the range $\gamma>1$ as pointed out in \cite{CV}. But the reduction of the full self similar equation to an explicitly analyzable ode is remarkable. We hope that this approach, as has been the case in other problems, can yield some insight into fluid singularities, at least in some region in space.\\

 The rest of this paper is organized as follows. In section \ref{sectionhom}, we prove Proposition \ref{thmdegen} by transforming the problem into a second order system of autonomous ODEs and give a detailed study of the phase portrait and its asymptotics. In Section \ref{sectionresolventone} and Section \ref{sectionresolventtwo}, we study the linearized operator of the {\em full} problem \eqref{homvortedeuqation} close to the homogeneous vortex. In Section \ref{sectionnonlinone} and Section \ref{sectionnonlintwo}, we solve the non linear bifurcation problem for suitable  perturbations and derive in particular the ode for the $\nu$ function. In Section \ref{sectionselfsim}, we construct the self similar profile and study its decay at $+\infty$ and regularity away from the origin, hence concluding the proof of Theorem \ref{thmmain}.
 
 \subsection*{Notations} We let $(r,z)$ be the cylindrical coordinates and define in the cylindrical basis $(\ve_r,\ve_\phi,\ve_z)$ $$\nabla ^\perp\psi=\left|\begin{array}{lll} -\pa_z\psi\\0\\\pa_r\psi.\end{array}\right.$$   We let $$\la y\ra=\sqrt{1+y^2}.$$
 
 \subsection*{Acknowledgment} P.R and A. R. S. are supported by the ERC/UKRI advanced grant SWAT. A.R.S is supported by ANR SMASH.


\section{The degenerate homogeneous vortex profile}
\label{sectionhom}


In this section, we use a front renormalization  to reduce the problem to an asymptotic degenerate problem, and prove the existence of smooth degenerate profiles for the special choice of homogeneous nonlinearities.


\subsection{Front renormalization}


We proceed to the front renormalization of the axisymmetric flow which makes the limiting degenerate problem appear.

\begin{lemma}[Renormalization and elliptic formulation ]
\label{vneoivnovnkvnnve}
Pick $a>0$. and two $\mathcal C^1$ functionals $\mathcal H, \mathcal C$ and let $\psi(r,z)$ solve the elliptic equation
\be
\label{vnbeioneinveneoivbis}
-\frac 1r\pa_r\left(\frac{1}{r}\pa_r\psi\right)-\frac{a^2}{r^2}\pa_z^2\psi=\frac{\mathcal C(\psi)\matchal C'(\psi)}{r^2}-\mathcal H'(\psi).
\ee 
Let 
\be
\label{defevtit}
\left|\begin{array}{l}
Z=\frac{z}{a}\\
\psi(r,z)=\psit(r,Z)\\
\ut_\phi(r,Z)=\frac{\mathcal C(\psit)}{r}\\
\ut(r,Z)=\frac{\nabla^\perp\psit}{r}+\ut_\theta\ve_\theta\\
\tilde{p}(r,z)=\mathcal H(\psit)-\frac{|\ut|^2}{2}
\end{array}\right.
\ee 
then $(\ut,\tilde{p})$ solve the stationary flow problem:
\be
\label{remainingequationsbis}
\left|\begin{array}{l}
\left[\ut_r\pa_r+\ut_Z\pa_Z\right]\ut_r-\frac{\ut_\theta^2}{r}+\pa_r\tilde{p}=0\\
\left[\ut_r\pa_r+\ut_Z\pa_Z\right]\ut_\theta+\frac{\ut_\theta \ut_r}{r}=0\\
\left[\ut_r\pa_r+\ut_Z\pa_Z\right]\ut_Z+\pa_Z\tilde{p}=0\\
\frac{1}r\pa_r(r\ut_r)+\pa_Z\ut_Z=0.
\end{array}\right.
\ee
\end{lemma}

\begin{proof}[Proof of Lemma \ref{vneoivnovnkvnnve}] This is a classical computation. Let $$\tilde{\Pi}=\tilde{p}+\frac{|\ut|^2}{2},$$ we recall the classical formula
$$\ut\wedge \tilde{w}=\nabla \left(\frac{|\ut|^2}{2}\right)-\ut\cdot\nabla \ut$$ so that the stationary equation \eqref{remainingequationsbis} is
$$\ut\cdot\nabla \ut+\nabla \tilde{p}=0\Leftrightarrow \nabla \tilde{\Pi}-\ut\wedge \tilde{w}=0.$$The  swirl is transported
$$\ut\cdot\nabla \ut_\phi+\frac{\ut_r\ut_\phi}r=0\Leftrightarrow \ut\cdot\nabla (r\ut_\phi)=0\Rightarrow r\ut_\phi=\mathcal C(\psi).$$
as well as the Bernoulli function
$$\ut\cdot\nabla \tilde{\Pi}=0\Rightarrow \tilde{\Pi}=\mathcal H(\psit).$$We now compute
 $$\ut\wedge \wt=\left|\begin{array}{l}
  -\frac 1r\pa_Z\psit\\
  \ut_\phi\\
  \frac{1}{r}\pa_r\psit
  \end{array}\right.\wedge
 \left|\begin{array}{l}-\pa_Z\ut_\phi\\ \wt_\phi\\ \frac 1r\pa_r(r\ut_\phi)\end{array}\right.= \left|\begin{array}{l}
  -\frac 1r\pa_Z\psit\\
  \frac{\mathcal C}{r}\\
  \frac{1}{r}\pa_r\psit
  \end{array}\right.\wedge
  \left|\begin{array}{l}
  -\mathcal C'\frac{\pa_Z\psit}{r}\\
  \wt_\phi\\
  \mathcal C'\frac{\pa_r\psit}{r}
  \end{array}\right.=\left|\begin{array}{l}
  \pa_r\psit\left[\frac{\mathcal C\matchal C'}{r^2}-\frac{\wt_\phi}{r}\right]\\
  0\\
   \pa_Z\psit\left[\frac{\mathcal C\matchal C'}{r^2}-\frac{\wt_\phi}{r}\right]
    \end{array}\right.
  $$
  and hence \eqref{remainingequationsbis} is equivalent to
  \be
  \label{neoneneov}
  \frac{\mathcal C\matchal C'}{r^2}-\frac{\wt_\phi}{R}=\mathcal H'.
  \ee
  Moreover:
  $$\wt_\phi=\pa_Z\ut_r-\pa_r \ut_Z=\pa_z\left(-\frac{1}{r}\pa_Z\psit\right)-\pa_r\left(\frac{1}{r}\pa_r\psit\right)=-\pa_r\left(\frac{1}{r}\pa_r\psit\right)-\frac{1}{r}\pa_Z^2\psit$$ and hence \eqref{neoneneov} is equivalent to 
  \bee
 &&-\pa_r\left(\frac{1}{r}\pa_r\psit\right)-\frac{1}{r}\pa_Z^2\psit=\frac{\mathcal C(\psi)\matchal C'(\psit)}{r^2}-\mathcal H'(\psit)\\
  &\Leftrightarrow& -\frac 1r\pa_r\left(\frac{1}{r}\pa_r\psi\right)-\frac{a^2}{r^2}\pa_z^2\psi=\frac{\mathcal C(\psi)\matchal C'(\psi)}{r^2}-\mathcal H'(\psi).
  \eee

\end{proof}


\subsection{Homogeneous vortex}


The limiting $a=0$ case in \eqref{vnbeioneinveneoivbis} is the degenerate vortex equation
\be
\label{homvortedeuqation}
-\frac 1r\pa_r\left(\frac{1}{r}\pa_r\psi\right)=\frac{\mathcal C\matchal C'}{r^2}-\mathcal H'.
\ee

\noindent{\em Homogeneous non linearities}. 
We pick  $\gamma>2$ and enforce the explicit choice 
\begin{equation}\label{eq: choice of C and H}
\left|\begin{array}{l}
\mathcal C_0(\psi)=\frac{\gamma-2}{\gamma -1}\sqrt{A}\psi^{\frac{\gamma-1}{\gamma-2}}\Rightarrow \matchal C_0\matchal C_0'(\psi) =A\psi^{\frac{\gamma}{\gamma-2}}\\
\mathcal H_0'(\psi)=\psi^{\frac{\gamma+2}{\gamma-2}},
\end{array}\right.
\end{equation}
and the problem then becomes
\be
\label{vnnrononr}
-\frac{1}{r}\pa_r\left(\frac{\pa_r\psi}{r}\right)=-\psi^{\frac{\gamma+2}{\gamma-2}}+\frac{A}{r^2}\psi^{\frac{\gamma}{\gamma-2}}.
\ee
\noindent{\em Change of variables.}  Let $\psi=r^2\Psi$, then 
\bea
\nonumber &&-\frac 1r\pa_r\left(\frac{1}{r}\pa_r\psi\right)=\frac{\mathcal C\matchal C'}{r^2}-\mathcal H'\Leftrightarrow
-\frac{1}{r}\pa_r\left(\frac{\pa_r\psi}{r}\right)=-\psi^{\frac{\gamma+2}{\gamma-2}}+\frac{A}{r^2}\psi^{\frac{\gamma}{\gamma-2}}\\
 &\Leftrightarrow&
\label{vneiovnineoneneovnvnoe}
\left|\begin{array}{l}\pa_r^2\Psi+\frac{3}{r}\pa_r\Psi- r^{\frac{2(\gamma+2)}{\gamma-2}}\Psi^{\frac{\gamma+2}{\gamma-2}}+Ar^{\frac{4}{\gamma-2}}\Psi^{\frac{\gamma}{\gamma-2}}=0\\
\Psi(0)=1\\
\Psi'(0)=0
\end{array}\right.
\eea
up to the scaling \be
\label{vniovneivnoneoiscalign}
\frac{1}{\l^\gamma}\Psi\left(\frac{r}{\l}\right).
\ee
induced by the homogeneous choice \eqref{eq: choice of C and H}.\\

\noindent{\em Reformulation of Proposition \ref{thmdegen}}. The rest of this section is devoted to the proof of the following Proposition \ref{thmdegen} which implies Theorem \ref{thmdegen}. We refer to Section \ref{degeneratevortex} for another proof restricted to $0<\gamma -2\ll1$ and which describes the profile as $\gamma\downarrow 2$.

\begin{proposition}[Existence of quantized smooth homogeneous vortex rings]
\label{propvrotexprofile}
Let 
\be
\label{quenionatgamma}
\gamma=2+\frac{1}{m}, \ \ m\in \Bbb N^*.
\ee Then there exists $A(\gamma)>0$ such that the unique solution to \eqref{vneiovnineoneneovnvnoe} is smooth, non negative and
satisfies $$\lim_{r\to +\infty}\Psi_*(r)=0.$$ Moreover, let $\psi_*=r^2\Psi_*$, then the following properties hold.\\

\noindent{\em (i) Non degeneracy.} $ \psi_*'$ vanishes exactly once for $r>0$ and
\be
\label{nondgeenera}
\zeta_*(r)=(\gamma-2)\psi_*+r\partial_r\psi_*>0\ \ \mbox{for}\ \ r>0.
\ee

\noindent{\em (ii) Tail.} let
\be
\label{numeoropoepoj}
A_\infty(\gamma)=\left(\frac{A(\gamma)+ \sqrt{A^2(\gamma)+4\gamma(\gamma-2)}}{2}\right)^{\frac{\gamma-2}{2}},\\
\ee
then there holds the asymptotic expansions near $r\to +\infty$: 
\be
\label{asymptoticexpansionvortex}
\left|\begin{array}{l}
\psi^*(r)=\frac{A_\infty(\gamma)(1+o(1))}{r^{\gamma-2}}\\
\zeta_*(r)=\frac{1+o(1)}{r^{\gamma-2+|\l_-|}}
\end{array}\right.,
\ee 
and similarly for higher derivatives.\\

\noindent{\em (iii) Regularity}: For any even function $\nu\in \mathcal C^\infty([0,+\infty),\Bbb R_*^+)$ with $\liminf_{z\in \Bbb R}\nu>0$, let 
\be
\label{vneiovnevneinenove}
\left|\begin{array}{l}
\psi_\nu(r,z)=\frac{1}{\left[\nu(z)\right]^{\gamma-2}}\psi_*\left(\frac{r}{\nu(z)}\right)\\
u_\phi=\pm\left[\frac{\gamma-2}{\gamma-1}A(\gamma)\right]^{\frac 12}\frac{\psi_\nu^{\frac{\gamma-1}{\gamma-2}}}{r}\\
u=\frac{\nabla^\perp\psi_\nu}{r}+u_\phi\ve_\phi
\end{array}\right.,
\ee then $u\in \mathcal C^\infty(\Bbb R^3,\Bbb R^3)$.\\

\noindent{\em (iv) Numerology as $\gamma \downarrow 2$}. There holds $$\inf_{\gamma>2}A(\gamma)>0.$$ Moreover,  let 
\be
\label{defnumberK}
\left|\begin{array}{l}
K(\gamma)=\frac{A_\infty(\gamma)^{\frac{2}{\gamma-2}}}{\gamma-2}\left[(\gamma+2)A_\infty(\gamma)^{\frac{2}{\gamma-2}}-\gamma A(\gamma)\right]\\
\l_-(\gamma)=\gamma-1-\sqrt{1+K(\gamma)} 
\end{array}\right.,
\ee
then
\be
\label{vneoinoenoveenolminus}
\l_-(\gamma)<0,
\ee
and
\be
\label{fneionfenvoinoe}
\lim_{\gamma \downarrow 2} K(\gamma)=+\infty.
\ee

\end{proposition}


\subsection{Emden transform}


We use the Emden transform to reduce \eqref{vneiovnineoneneovnvnoe} to the study of an explicit phase portrait. 
\begin{lemma}[Emden variables]
\label{lemamphaseportrait}
Let the Emden transform $$\left|\bear
s=\log r\\
\phi=r^\gamma \Psi\\
u=\frac{d\phi}{ds}\equiv \phi',
\ear\right.
$$
then \eqref{vneiovnineoneneovnvnoe} is 
\be
\label{vneioenenoven}
 \frac{1}{2(\gamma-1)}u\frac{du}{d\phi}=u-F(\phi)
 \ee 
with
\be
\label{vneoneineone}
F(\phi)=\frac{\gamma(\gamma-2)\phi-\phi^{\frac{\gamma+2}{\gamma-2}}+A\phi^{\frac{\gamma}{\gamma-2}}}{2(\gamma-1)}.
\ee
\end{lemma}

\begin{proof}[Proof of Lemma \ref{lemamphaseportrait}]  Let 
\be
\label{neionivnone}
s=\log r, \ \ \phi=r^\gamma \Psi,
\ee 
then
$$\Lambda^2_r=r\pa_r(r\pa_r)=r^2\pa_r^2+r\pa_r\Rightarrow r^2\pa_r^2=\Lambda^2_r-\Lambda_r,$$ from which 
$$r^2\left(\pa_r^2\Psi+\frac{3}{r}\pa_r\Psi\right)=(\pa_s^2+2\pa_s)(e^{-\gamma s}\phi),$$ and $$\left|\begin{array}{l}
\pa_s(e^{-\gamma s}\phi)=(\phi'-\gamma\phi)e^{-\gamma s}\\
\pa^2_s(e^{\gamma s}\phi)=(\phi''-\gamma\phi'-\gamma(\phi'-\gamma\phi))e^{-\gamma s}=(\phi''-2\gamma\phi'+\gamma^2\phi)e^{-\gamma s}
\end{array}\right.,
$$
giving
\bee
e^{\gamma s}(\pa_s^2+2\pa_s)(e^{-\gamma s}\phi)&=&\phi''-2\gamma\phi'+\gamma^2\phi+2(\phi'-\gamma\phi)\\
& = & \phi''-2(\gamma-1)\phi'+\gamma(\gamma-2)\phi.
\eee
We compute
$$\left|\begin{array}{l}
r^{\frac{2(\gamma+2)}{\gamma-2}}\Psi^{\frac{\gamma+2}{\gamma-2}}=r^{\frac{2(\gamma+2)}{\gamma-2}-\frac{\gamma(\gamma+2)}{\gamma-2}}\phi^{\frac{\gamma+2}{\gamma-2}}=r^{-(\gamma+2)}\phi^{\frac{\gamma+2}{\gamma-2}}\\
r^{\frac{4}{\gamma-2}}\Psi^{\frac{\gamma}{\gamma-2}}=r^{\frac{4}{\gamma-2}-\frac{\gamma^2}{\gamma-2}}\phi^{\frac{\gamma}{\gamma-2}}=r^{-(\gamma+2)}\phi^{\frac{\gamma}{\gamma-2}}
\end{array}\right.,
$$
and hence the Emden transform formulation of \eqref{vneiovnineoneneovnvnoe} becomes
\be
\label{emdnenfonro}
\phi''-2(\gamma-1)\phi'+\gamma(\gamma-2)\phi-\phi^{\frac{\gamma+2}{\gamma-2}}+A\phi^{\frac{\gamma}{\gamma-2}}=0,
\ee
which we rewrite as a first order system of ODEs
\bea
\nonumber &&\left|\begin{array}{l}
\phi'=u\\
u'=2(\gamma-1)u-\gamma(\gamma-2)\phi+\phi^{\frac{\gamma+2}{\gamma-2}}-A\phi^{\frac{\gamma}{\gamma-2}}
\end{array}\right.\\
&\Leftrightarrow&\label{eq: ode system hom vortex}
    \left|\begin{array}{l}
\phi'=u\\
u'=2(\gamma-1)\left[u-F(\phi)\right]\end{array}\right.\\
\nonumber &\Rightarrow&
 \frac{1}{2(\gamma-1)}\frac{du}{d\phi}=\frac{u-F(\phi)}{u}.
 \eea
 \end{proof}

 \subsection{The radial solution at the origin}
 \label{subsec:initial cond}
 
 We reformulate the regularity of the $\Psi$ solution at $r=0$ in the language of the Emden variables. The linearized system at the origin $r=0$ is as $s\to -\infty$ and $(u,\phi)=(0,0)$
\bee
&&\left|\begin{array}{l}
\phi'=u\\
u'=2(\gamma-1)u-\gamma(\gamma-2)\phi
\end{array}\right.\Leftrightarrow \left|\begin{array}{l}u=\phi'\\ \phi''-2(\gamma-1)\phi'+\gamma(\gamma-2)\phi=0\end{array}\right.\\
&\Leftrightarrow & \left|\begin{array}{l}
\phi=a_1e^{\gamma s}+a_2e^{(\gamma-2)s}\\ u=a_1\gamma e^{\gamma s}+a_2(\gamma-2)e^{(\gamma-2)s},
\end{array}\right.
\eee
and hence the two principal slopes at $(u,\phi)=(0,0)$ are given by: $$\left|\begin{array}{l}
\frac{u}{\phi}=\gamma-2, \ \ {\rm stable}\\ 
\frac{u}{\phi}=\gamma, \ \ {\rm unstable}.
\end{array}\right.
$$
Observe that $$\gamma-2-F'(0)=\gamma-2-\frac{\gamma(\gamma-2)}{2(\gamma-1)}=\frac{(\gamma-2)^2}{2(\gamma-1)}>0.$$
Recall that in $r$ we had the initial condition $$\Psi(0)=1, \ \ \Psi'(0)=0,$$ which from \eqref{vneiovnineoneneovnvnoe} becomes the solution to the fixed point equation: 
\be
\label{vneonneinvnqyrtueytuy}
\Psi'(r)=\frac{1}{r^3}\int_0^r\left[\tau^{\frac{2(\gamma+2)}{\gamma-2}}\Psi^{\frac{\gamma+2}{\gamma-2}}-A\tau^{\frac{4}{\gamma-2}}\Psi^{\frac{\gamma}{\gamma-2}}\right]\tau^3d\tau,
\ee 
and hence $$\left|\begin{array}{l}
\phi=r^{\gamma}\Psi=e^{\gamma s}(1+o_{s\to -\infty}(1))\\
u=\frac{d\phi}{ds}=r\pa_r\phi=r\pa_r(r^\gamma\Psi)=\gamma r^\gamma \Psi+r^{\gamma+1}\pa_r\Psi=\gamma e^{\gamma s}(1+o_{s\to -\infty(1)}),\end{array}\right.$$ which yields$$u=\gamma \phi\left[1+o_{s\to -\infty}(1)\right]$$ for the radial solution. Hence the radial solution is the separatrix emerging from the origin with exceptional (unstable) $\gamma$ slope.
 
 
\subsection{Drawing the phase portrait} 


For this, we need only understand the function $F$.\\

\noindent{\em Properties of $F$}. We compute $$
2(\gamma-1)F'(\phi)=\gamma(\gamma-2)-\frac{\gamma+2}{\gamma-2}\phi^{\frac{4}{\gamma-2}}+\frac{A\gamma}{\gamma-2}\phi^{\frac{2}{\gamma-2}},$$
and
\bee
2(\gamma-1)F''(\phi)&= &-\frac{4(\gamma+2)}{(\gamma-2)^2}\phi^{\frac{4}{\gamma-2}-1}+\frac{2A\gamma}{(\gamma-2)^2}\phi^{\frac{2}{\gamma-2}-1}\\
&= &\frac{1}{(\gamma-2)^2}\phi^{\frac{2}{\gamma-2}-1}\left[-4(\gamma+2)\phi^{\frac{2}{\gamma-2}}+2A\gamma\right].
\eee
Let $\phi^*$ be defined by $$(\phi^*)^{\frac{2}{\gamma-2}}=\frac{A\gamma}{2(\gamma+2)} \Leftrightarrow F''(\phi^*)=0,$$
then 
\bee
2(\gamma-1)F'(\phi^*)&= &\gamma(\gamma-2)-\frac{\gamma+2}{\gamma-2}\phi_*^{\frac{4}{\gamma-2}}+\frac{A\gamma}{\gamma-2}\phi_*^{\frac{2}{\gamma-2}}\\
& = & \gamma(\gamma-2)-\frac{\gamma+2}{\gamma-2}\left[\frac{A\gamma}{2(\gamma+2)}\right]^2+\frac{A\gamma}{\gamma-2}\left[\frac{A\gamma}{2(\gamma+2)}\right]\\
& = & \gamma(\gamma-2)-\frac{A^2\gamma^2}{4(\gamma-2)(\gamma+2)}+\frac{A^2\gamma^2}{2(\gamma-2)(\gamma+2)}\\
& = & \gamma(\gamma-2)+\frac{A^2\gamma^2}{4(\gamma-2)(\gamma+2)}>0.
\eee
We conclude that:$$F''\left|\begin{array}{l} >0\ \ \mbox{for}\ \ 0\le \phi< \phi^*\\ <0 \  \ \mbox{for}\ \ 0\le \phi< \phi^*\end{array}\right.,$$ 
and hence since $F'(\phi^*)>0$, $F'(0)>0$ and $\lim_{\phi\to +\infty} F'(\phi)=-\infty$, there exist $0<\phi^*<\phi^*_1$ such that: $$F'\left|\begin{array}{l} >0\ \ \mbox{for}\ \ 0\le \phi< \phi_1^*\\ <0\ \  \mbox{for}\ \ \phi^*_1<\phi\end{array}\right..
$$
The sign of $F$ may now be exactly computed by noticing that
\be
\label{vneiovnenvenonenv}
\left|\begin{array}{l}
2(\gamma-1)F(\phi)=\phi\left[\gamma(\gamma-2)-\phi^{\frac{4}{\gamma-2}}+A\phi^{\frac{2}{\gamma-2}}\right]=\phi\left[-x^2+Ax+\gamma(\gamma-2)\right]\\
 x=\phi^{\frac{2}{\gamma-2}}
 \end{array}\right.,
 \ee and hence the discriminant $\Delta=A^2+4\gamma(\gamma-2)>0,$ which yields the only positive root
\be
\label{venonoenoeovnvjeojejioe}
(\phi^*_+)^{\frac{2}{\gamma-2}}=\frac{A+ \sqrt{A^2+4\gamma(\gamma-2)}}{2}.
\ee

\noindent{\em Phase portrait}. The formulation \eqref{vneioenenoven} now allows us to draw the phase portrait in the $(\phi,u)$ plane, see Figure \ref{fig:vect field plot 1} below.\\

 \begin{figure}[h]
     \centering
    \includegraphics[width=0.4\linewidth]{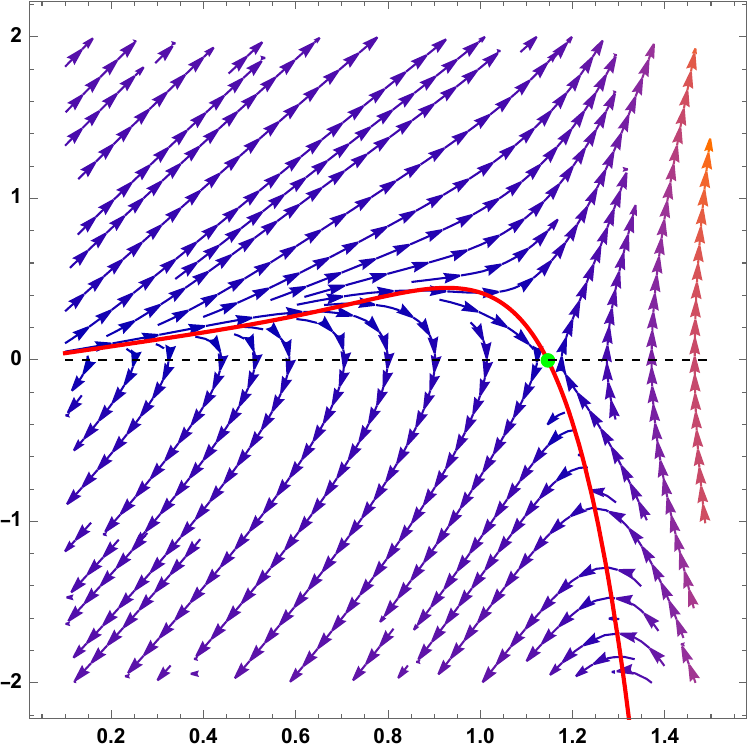}
     \caption{Plot of the vector field for $\gamma=2.5$ and $A=1$ in $(\phi,u)$ coordinates. In red we have the graph $u=F(\phi)$ and in green $\phi_+^*$.}
     \label{fig:vect field plot 1}
 \end{figure}
 

\subsection{The stable manifold entering $(\phi_+^*,0)$}


  We now study the flow near the stationary solution $(\phi^*_+,0)$.\\

\noindent{\em Linearisation around $(\phi_+^*,0)$}. We linearize \eqref{eq: ode system hom vortex}
at $\phi^*_+$ and obtain the linearized system
$$\left|\begin{array}{l}
\phi'=u\\
u'=2(\gamma-1)(u-F'(\phi^*_+)\phi)
\end{array}\right.\Leftrightarrow \left|\begin{array}{l} u=\phi'\\\phi''-2(\gamma-1)\phi'+2(\gamma-1)F'(\phi^*_+)\phi=0.
\end{array}\right.
$$
The equation
\be
\label{equationsolminsu}
z^2-2(\gamma-1)z+2(\gamma-1)F'(\phi^*_+)=0
\ee
 has positive discriminant from the fact that $F'(\phi^*_+)<0$, and hence the roots are given by 
 \bea
 \label{vneioneinveoinenoier}
 \l_{\pm}&= &\frac{2(\gamma-1)\pm\sqrt{4(\gamma-1)^2-8(\gamma-1)F'(\phi^*_+)}}{2}\nonumber \\
 &= &\gamma-1\pm\sqrt{(\gamma-1)^2-2(\gamma-1)F'(\phi^*_+)},
  \eea
 with $\l_+>0$ and $\l_-<0.$ Hence the only separatrix that enters $(\phi^*_+,0)$ as $s\to +\infty$ has slope $\l_-$.\\
 
 \noindent{\em Proof of \eqref{vneoinoenoveenolminus}}. Let $K$ be given by \eqref{defnumberK}. Recall \eqref{vneoneineone}, \eqref{vneioneinveoinenoier} and compute with
 $x=(\phi_+^*)^{\frac{2}{\lambda-2}}=A_\infty^{\frac{2}{\lambda-2}}$:
\bee
    \l_-&= &\gamma-1-\sqrt{(\gamma-1)^2-2(\gamma-1)F'(\phi_+^*)}\\
    &= &\gamma-1-\sqrt{(\gamma-1)^2-\left(\gamma^2-2\gamma-\frac{\gamma+2}{\gamma-2}x^2+\frac{A\gamma}{\gamma-2}x\right)}\\
& =  &\gamma-1-\sqrt{1+\frac{\gamma+2}{\gamma-2}x^2-\frac{A\gamma x}{\gamma-2}}=\gamma-1-\sqrt{1+\frac{x}{\gamma-2}[(\gamma+2)x-\gamma, A]}\eee 
thus
\[
\l_-=\gamma-1-\sqrt{1+K},
\]
which is \eqref{defnumberK}, and \eqref{vneoinoenoveenolminus} is proved.\\
 
\noindent{\em Relative position of the stable manifold}. First compute 
\[\l_-=\gamma-1-\sqrt{(\gamma-1)^2-2(\gamma-1)F'(\phi^*_+)}=\gamma-1-\sqrt{(\gamma-1)^2+2(\gamma-1)|F'(\phi^*_+)|},
\]
and thus 
\bee
&&|\l_-|<|F'(\phi^*_+)|\Leftrightarrow \sqrt{(\gamma-1)^2+2(\gamma-1)|F'(\phi^*_+)|}<|F'(\phi^*_+)|+(\gamma-1)\\
&\Leftrightarrow& (\gamma-1)^2+2(\gamma-1)|F'(\phi^*_+)|<(\gamma-1)^2+2(\gamma-1)|F'(\phi^*_+)|+|F'(\phi^*_+)|^2,
\eee
which holds true and thus the stable manifold entering $(0,\phi_+^*)$ with slope $\lambda_-$ is below $u=F(\phi)$  near $(\phi_+^*,0)$.\\
We now claim that  this stable manifold is globally below its tangent at $(0,\phi_+^*)$. Let $$h(\phi)=u-\l_-(\phi-\phi^*),$$ then from  \eqref{equationsolminsu}:
\bee
&&\frac{dh}{d\phi}=2(\gamma-1)\frac{u-F(\phi)}{u}-\l_-=\frac{2(\gamma-1)(u-F(\phi))-\l_-u}{u}\\
& \Rightarrow  &uh'= 2(\gamma-1)(h+\l_-(\phi-\phi^*)-F(\phi))-\l_-[h+\l_-(\phi-\phi^*)]\\
&\Rightarrow &uh'+(\l_--2(\gamma-1))h=2(\gamma-1)\left[-F(\phi)+(\phi-\phi^*)\left(\l_--\frac{\l_-^2}{2(\gamma-1)}\right)\right]\\
&\Rightarrow &uh'+(\l_--2(\gamma-1))h=-2(\gamma-1)\left[F(\phi)-F'(\phi^*)(\phi-\phi^*)\right]>0,
\eee
for $0<\phi<\phi^*_+$. Since $u\sim |\l_-|(\phi^*-\phi)$ for the tangent this yields
$$h'-\frac{|\l_-|+2(\gamma-1)}{|\l_-|(\phi^*-\phi)}h=H(\phi)>0,$$ which implies with $\sigma=\phi^*-\phi$:
$$-\frac{dh}{d\sigma}-\frac{|\l_-|+2(\gamma-1)}{|\l_-|\sigma}h=H(\phi)\Leftrightarrow \frac{1}{\sigma^\alpha}\frac{d}{d\sigma}(\sigma^\alpha h)<0, \ \ \alpha=1+2\frac{\gamma-1}{|\l_-|}>0,$$ and hence $h<0$ for $\phi<\phi^*$, and the curve is below its tangent initially. At a previous point of contact with $h=0$, the derivative is positive which is a contradiction, and hence the curve is globally below its tangent.\\

\noindent{\em Connection problem}. The heart of the proof of Proposition \ref{propvrotexprofile} is now  to show that there exists $A(\gamma)>0$ such that the solution emanating from the unstable separatrix with $\gamma$ slope at $(0,0)$ enters the critical point $(\phi_+^*,0)$ when $s\to +\infty$. We will argue using a classical bifurcation argument on $A(\gamma)$.


\subsection{The case $A$ large}\label{sectionc3large}


\begin{lemma}[Limiting problem]
\label{limtiengningrrp}
 Consider the ODE
\be
\label{newinvoeihroie}
\left|\begin{array}{l}
\phi'=u\\
u'=2(\gamma-1)u-\gamma(\gamma-2)\phi-\phi^{\frac{\gamma}{\gamma-2}}
\end{array}\right.,
\ee
then the  curve with $\gamma$ slope at $(u,\phi)=(0,0)$ touches $u=0$ in finite time $s>0.$
\end{lemma}

\begin{proof}[Proof of Lemma \ref{limtiengningrrp}] Equivalently 
\be
\label{vneioenenovenbis}
 \frac{1}{2(\gamma-1)}\frac{du}{d\phi}=\frac{u-F(\phi)}{u}, \text{ with }F(\phi)=\frac{\gamma(\gamma-2)\phi+\phi^{\frac{\gamma}{\gamma-2}}}{2(\gamma-1)}.\ee
Thus the slopes of the stable and unstable manifolds at the origin are the same as previously i.e $\gamma-2$ and $\gamma$ respectively. Plotting the vector field we get Figure \ref{fig:vect field plot 2}.\\
 \begin{figure}[h]
     \centering
     \includegraphics[width=0.4\linewidth]{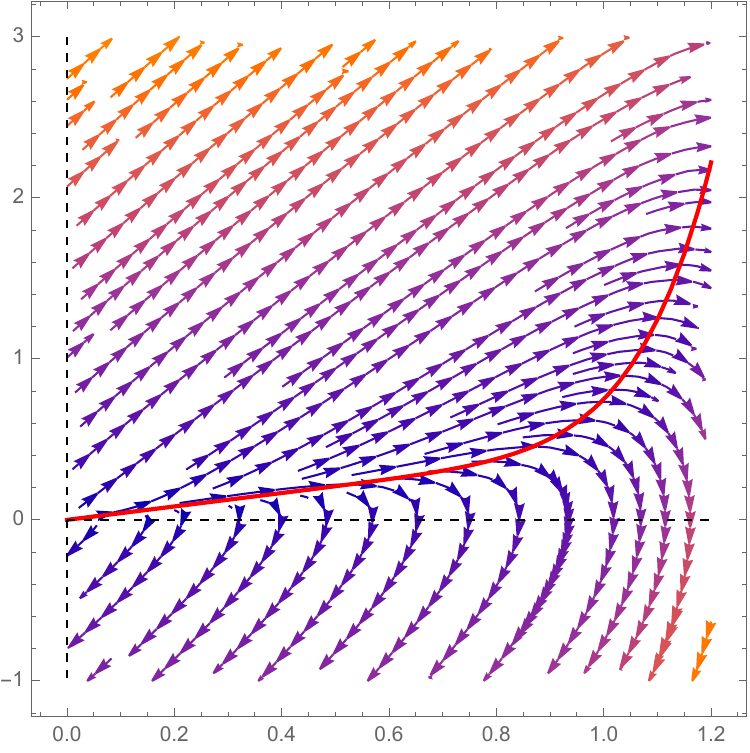}
     \caption{Plot of the vector field for $\gamma=2.5$ in $(\phi,u)$ coordinates. In red we have the graph $u=F(\phi)$.}
     \label{fig:vect field plot 2}
 \end{figure}

\noindent{\bf step 1} Touching $u=F(\phi)$ in finite time. We compute for $\delta\ge 2$:
\bee
&&F(\phi)-(\gamma-2+\delta)\phi=0\Leftrightarrow \frac{\gamma(\gamma-2)\phi+\phi^{\frac{\gamma}{\gamma-2}}}{2(\gamma-1)}=(\gamma-2+\delta)\phi\\
&\Leftrightarrow&-\phi^{\frac{\gamma}{\gamma-2}}+\left[2(\gamma-1)(\gamma-2+\delta)-\gamma(\gamma-2)\right]\phi=0\\
& \Leftrightarrow& \phi\left[-\phi^{\frac{2}{\gamma-2}}+(\gamma-2)^2+2\delta(\gamma-1)\right]=0,
\eee
and hence $$\phi_\delta^{\frac{2}{\gamma-2}}=(\gamma-2)^2+2\delta(\gamma-1).$$
We now consider the integral curve emanating from $(\phi_\delta,F(\phi_\delta))$ such that $$[u-(\gamma-2+\delta)\phi](\phi_\delta)=0.$$ Assume that there is another point of contact $0<\phi_1<\phi_\delta$ with $$[u-(\gamma-2+\delta)\phi](\phi_1)=0,$$ then at this point:
\bee
&&\frac{d}{d\phi}[u-(\gamma-2+\delta)\phi]=\frac{du}{d\phi}-(\gamma-2+\delta)\\
&=& 2(\gamma-1)\left[1-\frac{F(\phi)}{u}\right]-(\gamma-2+\delta)=\gamma-\delta-2(\gamma-1)\frac{F(\phi)}{(\gamma-2+\delta) \phi}\\
& = & \frac{2(\gamma-1)}{(\gamma-2+\delta)\phi}\left[\frac{(\gamma-\delta)(\gamma-2+\delta)}{2(\gamma-1)}\phi-F(\phi)\right]\\
&=& \frac{2(\gamma-1)}{(\gamma-2+\delta)\phi}\left[\left(F'(0)-\frac{\delta(\delta-2)}{2(\gamma-1)}\right)\phi-F(\phi)\right]\\
&\leq &\frac{2(\gamma-1)}{(\gamma-2+\delta)\phi}\left[F'(0)\phi-F(\phi)\right]=-\frac{2(\gamma-1)}{(\gamma-2+\delta)\phi}\frac{\phi^{\frac{\gamma}{\gamma-2}}}{2(\gamma-1)}<0,
\eee
and hence since $u>0$ for $\phi_1<\phi<\phi_\delta$ from the phase portrait, there cannot be a point of contact before. We conclude that the integral curve must remain above the line $u=(\gamma-2+\delta)\phi$ and hence cannot reach $(0,0)$ with $\gamma-2$ slope. It must therefore either enter $(0,0)$ with the $\gamma$ slope or from the phase portrait touch $\phi=0$ at $u>0$. In both cases since integral curves cannot cross the curve with $\gamma$ slope at $(0,0)$ must lie below this trajectory and hence touch $u=F(\phi)$.\\

\noindent{\bf step 2} Touching $u=0$ in finite time. We now claim that the curve with $\gamma$ slope at $(0,0)$ touches the axis  $u=0$ in finite $s$.  Indeed, after touching $u=F(\phi)$ which from the phase portrait must be crossed strictly, we have $\frac{du}{d\phi}<0$ from \eqref{vneioenenovenbis}, and hence $u$ decreases. Hence \eqref{vneioenenovenbis} forces  $$\frac{1}{4(\gamma-1)}\frac{du^2}{d\phi}\le u^*- F(\phi)\le u^*-\frac{\gamma(\gamma-2)}{2(\gamma-1)} \phi,$$ which forces $u^2$ to vanish in finite $\phi> 0$ and such point being non critical this also occurs for finite $s$.

\end{proof}

\begin{lemma}[$A$ large]
\label{lemmanveneonoe}
 Assume $\gamma>2$, then there exists $A(\gamma)$ such that for $A\geq A(\gamma)$, the curve with $\gamma$ slope at $(0,0)$ touches $u=0$ in finite $s>0.$
\end{lemma}

\begin{proof}[Proof of Lemma \ref{lemmanveneonoe}] The curve with $\gamma$ curve at $(0,0)$ corresponds up to scale invariance to the radial solution $\Psi$ solution to 
$$
\left|\begin{array}{l}\pa_r^2\Psi+\frac{3}{r}\pa_r\Psi- r^{\frac{2(\gamma+2)}{\gamma-2}}\Psi^{\frac{\gamma+2}{\gamma-2}}+Ar^{\frac{4}{\gamma-2}}\Psi^{\frac{\gamma}{\gamma-2}}=0\\
\Psi(0)=1\\
\Psi'(0)=0
\end{array}\right.,
$$
We renormalize
$$\Psi=\mu\Phi\left(\frac{r}{\l}\right),$$ and hence
\bee
&&\pa_r^2\Psi+\frac{3}{r}\pa_r\Psi- r^{\frac{2(\gamma+2)}{\gamma-2}}\Psi^{\frac{\gamma+2}{\gamma-2}}+Ar^{\frac{4}{\gamma-2}}\Psi^{\frac{\gamma}{\gamma-2}}=0\\
&\Leftrightarrow & \frac{\mu}{\l^2}\left(\Phi''+\frac{3}{R}\Phi'\right)-\mu^{\frac{\gamma+2}{\gamma-2}}(\l R)^{\frac{2(\gamma+2)}{\gamma-2}}\Phi^{\frac{\gamma+2}{\gamma-2}}+A\mu^{\frac{\gamma}{\gamma-2}}(\l R)^{\frac{4}{\gamma-2}}\Phi^{\frac{\gamma}{\gamma-2}}=0\\
&\Leftrightarrow& \Phi''+\frac{3}{R}\Phi'+(\mu \l^\gamma)^{\frac{4}{\gamma-2
}}R^{\frac{2(\gamma+2)}{\gamma-2}}\Phi^{\frac{\gamma+2}{\gamma-2}}-A (\mu \l^\gamma)^{\frac{2}{\gamma-2
}}R^{\frac{4}{\gamma-2}}\Phi^{\frac{\gamma}{\gamma-2}}=0,
\eee
and hence the choice $$\mu=1, \ \ A\l^{\frac{2\gamma}{\gamma-2}}=1,$$ yields
$$\left|\begin{array}{l}
 \Phi''+\frac{3}{R}\Phi'-\frac{1}{A^2}R^{\frac{2(\gamma+2)}{\gamma-2}}\Phi^{\frac{\gamma+2}{\gamma-2}}-R^{\frac{4}{\gamma-2}}\Phi^{\frac{\gamma}{\gamma-2}}=0\\
\Phi(0)=1, \ \ \Phi'(0)=0
\end{array}\right.,
$$
which locally converges  as $A\to +\infty$ to the solution to the limiting problem \eqref{newinvoeihroie} after performing an Emden transform, and hence the claim follows from Lemma \ref{limtiengningrrp} and a straightforward continuity argument.

\end{proof}

\subsection{The case $A$ small}\label{sectionc3small}

\begin{lemma}[Case $A=0$]
\label{lemannvdnononr}
The separatrix with $\gamma$ slope at $(0,0)$ remains above $u=F(\phi)$.
\end{lemma}

\begin{proof}[Proof of Lemma \ref{lemannvdnononr}] The shape of $F$ is the same, but $$2(\gamma-1)F''(\phi)=\frac{1}{(\gamma-2)^2}\phi^{\frac{2}{\gamma-2}-1}\left[-4(\gamma+2)\phi^{\frac{2}{\gamma-2}}\right]<0,$$ and hence $F$ is concave. Moreover, the separatrix with $\gamma$ slope corresponds from \eqref{vneonneinvnqyrtueytuy} to
$$\Psi'(r)=\frac{1}{r^3}\int_0^r\left[\tau^{\frac{2(\gamma+2)}{\gamma-2}}\Psi^{\frac{\gamma+2}{\gamma-2}}-A\tau^{\frac{4}{\gamma-2}}\Psi^{\frac{\gamma}{\gamma-2}}\right]\tau^3d\tau=\frac{1}{r^3}\int_0^r\left[\tau^{\frac{2(\gamma+2)}{\gamma-2}}\Psi^{\frac{\gamma+2}{\gamma-2}}\right]\tau^3d\tau>0,$$
and hence $\Psi'(r)>0$ near $r>0$. Hence $$r^\gamma\frac{d}{dr}\left(\frac{\phi}{r^\gamma}\right)>0\Leftrightarrow \phi'-\gamma\phi>0\ \ \mbox{near}\ \ s\to -\infty,$$ which means that $$u-\gamma\phi>0\ \ \mbox{near}\ \ \phi=0.$$
Now suppose a first point of contact $u-\gamma \phi=0$ occurs, then at that point we have
\bee
\frac{d}{d\phi}[u-\gamma\phi]&= &\frac{du}{d\phi}-\gamma=2(\gamma-1)\left[1-\frac{F(\phi)}{u}\right]-\gamma=\gamma-2-2(\gamma-1)\frac{F(\phi)}{\gamma\phi}\\
& = & \frac{2(\gamma-1)}{\gamma\phi}\left[\frac{\gamma(\gamma-2)}{2(\gamma-1)}\phi-F(\phi)\right]=\frac{2(\gamma-1)}{\gamma \phi}[F'(0)\phi-F(\phi)]>0,
\eee
since $F$ is concave. Thus the unstable manifold with $\gamma$ slope at $(0,0)$ remains above it's tagent which lies strictly above $F$.
\end{proof}


\subsection{Existence of a critical $A(\gamma)$}


\begin{lemma}[Case $A(\gamma)$ critical]
\label{lemannvdnononrml;mv}
There exists $A(\gamma)>0$ such that the unstable manifold with $\gamma$ slope at $(0,0)$ touches $u=F(\phi)$ at $\phi<\phi_1^*$, and exits on the other side of $u=F(\phi)$.
\end{lemma}

\begin{proof}[Proof of Lemma \ref{lemannvdnononrml;mv}]
The unstable manifold with $\gamma$ slope at $(0,0)$ is a continuous function of $A(\gamma)$ (through a standard Banach fixed point scheme with parameter $A(\gamma)$). For $A(\gamma)$ small, by continuous deformation of Lemma \ref{lemannvdnononr}, the curve cannot touch $u=F(\phi)$, and it does for $A>A(\gamma)$ from Lemma \ref{lemmanveneonoe}. Hence at the first value of contact, the point of contact must be at the maximum of $F$, and the conclusion follows by continuous deformation of this value.
\end{proof}

\subsection{Proof of Proposition \ref{propvrotexprofile}}

We are now in position to conclude the proof of Proposition \ref{propvrotexprofile}.\\

\noindent{\bf step 1} Existence of $A(\gamma)>0$. We conclude from section \ref{sectionc3large} and section \ref{sectionc3small} and a straightforward continuity argument on $A$ that for $\gamma>2$, there exists at least one value $A(\gamma)>0$ for which the curve with slope $\gamma$ as $s\to-\infty$ coincides with the separatrix which enters $(\phi^+_*,0)$ as $s\to \infty$. We let $\psi=r^2\Psi$ be such a curve which is the smooth radial solution to \eqref{vneiovnineoneneovnvnoe}.\\

\noindent{\bf step 2} Lower bound as $\gamma\downarrow 2$. We claim
\be
\label{nienneinv}
\liminf_{\gamma\downarrow 2} A(\gamma)>0.
\ee
Indeed, assume by contradiction that we can find a sequence $\gamma_n\downarrow 2$ with $A(\gamma_n)\to 0$, then from \eqref{venonoenoeovnvjeojejioe}, 
\be
\label{vneovnenenvo}
\phi^*_+(\gamma_n)\to 0.
\ee On the other hand, from \eqref{vneonneinvnqyrtueytuy}, the separatrix with $\gamma$ slope at $(0,0)$ is the solution to the fixed point equation:
\bee
\Psi(r)&=&1+\int_0^r\left(\int_0^{\tau}\left[\sigma^{\frac{2(\gamma+2)}{\gamma-2}}\Psi^{\frac{\gamma+2}{\gamma-2}}-A\sigma^{\frac{4}{\gamma-2}}\Psi^{\frac{\gamma}{\gamma-2}}\right]\sigma^3d\sigma\right) \frac{1}{\tau^3}d\tau\\
& =  & 1+\int_0^r\left(\int_0^{\tau}\left[\left(\sigma^2\Psi\right)^{\frac{\gamma+2}{\gamma-2}}-A\left(\sigma^4\Psi^{\gamma}\right)^{\frac{1}{\gamma-2}}\right]\sigma^3d\sigma\right) \frac{1}{\tau^3}d\tau.
\eee
Pick $0<\delta\ll 1$ such that $0<\gamma-2<\e(\delta),$ then we have $$\forall r\in \left[0,\frac 12\right], \ \ |\Psi-1|+|\Psi'|\le \delta,$$ and hence for $r\in [0,\frac 12]$: 
$$|\Psi'(r)|\le r^{\frac{1}{10(\gamma-2)}},$$ which implies: 
\bee
&&\left|\Psi'\right|=\left|\frac{d}{dr}\left(\frac{\phi}{r^\gamma}\right)\right|=\frac{|u-\gamma \phi|}{r^{\gamma+1}}<r^{\frac{1}{10(\gamma-2)}}\\
&\Rightarrow& u(r)\ge \gamma \phi(r) -r^{100}\ge \frac{r^{\gamma}}{2}.
\eee
Hence the $u(\phi)$ trajectory remains for $r\in \left(0,\frac 12\right]$ in the zone $(\phi>0,u>0)$ and reaches at $r=\frac 12$ a point for which $\phi(\frac 12)>  \phi^*_+(\gamma_n)$ for $n$ sufficiently large by \eqref{vneovnenenvo}, and hence from the phase portrait $u'(\phi)>0$ after this point, and hence in particular $u(\phi^*_+)>0$ which is a contradiction.\\
 
\noindent{\bf step 3} Monotonicity of the stream function. We claim that $\psi'$ vanishes exactly once on $(0,+\infty)$ and 
\be
\label{senocnoen}
(\gamma-2)\psi+r\pa_r\psi>0.
\ee
Indeed, we first compute
$$\frac{d\psi}{dr}=\frac1r\frac{d}{ds}\left(\frac{\phi}{r^{\gamma-2}}\right)=\frac{r^{\gamma-2}}{r}(u-(\gamma-2)\phi).$$
Let $$h(\phi)=u-(\gamma-2)\phi,$$ then at a point of contact $h(\phi)=0$, we have 
\bee
\frac{d}{d\phi}[u-(\gamma-2)\phi]&= & \frac{du}{d\phi}-(\gamma-2)
= 2(\gamma-1)\left[1-\frac{F(\phi)}{u}\right]-(\gamma-2)\\
&= & \gamma-2(\gamma-1)\frac{F(\phi)}{(\gamma-2) \phi}
 = \frac{2(\gamma-1)}{(\gamma-2)\phi}\left[\frac{\gamma(\gamma-2)}{2(\gamma-1)}\phi-F(\phi)\right]\\
&= & \frac{2(\gamma-1)}{(\gamma-2)\phi}\left[F'(0)\phi-F(\phi)\right]=-\frac{2(\gamma-1)}{(\gamma-2)\phi}\frac{A(\gamma)\phi^{\frac{\gamma}{\gamma-2}}-\phi^{\frac{\gamma+2}{\gamma-2}}}{2(\gamma-1)}\\
& = & -\frac{2\phi^{\frac{2}{\gamma-2}}(A(\gamma)-\phi^{\frac{2}{\gamma-2}})}{\gamma-2}.
\eee
The solution starts with $\gamma$ slope at the origin and hence $u-\gamma \phi>0$ initially which forces $h'(\phi_0)\le  0$ at the first point of contact which must exist. If there is another point of contact $\phi_1$, then the curve passes above the line and hence there must be another point of contact $\phi_2>\phi_1$. But then $$h'(\phi_1)\ge 0\Rightarrow A-\phi_1^{\frac{2}{\gamma-2}}\le 0\Rightarrow A-\phi_2^{\frac{2}{\gamma-2}}< 0\Rightarrow h'(\phi_2)>0,$$
 which is a contradiction. Finally,
 $$\frac{d}{dr}(r^{\gamma-2}\psi)=\frac{d}{dr}\phi=\frac{u}r>0,$$ from the phase portrait, and \eqref{senocnoen} is proved as well as $\psi>0$ for $r>0$.\\
 
  \noindent{\bf step 4} Tail of the homogeneous vortex. From the phase portrait
$$r^{\gamma-2}\psi=\phi=\phi_+^*\left(1+e^{\l_-s}+{\rm lot}\right)=A_\infty(\gamma)\left(1+\frac{1+o(1)}{r^{|\l_-|}}\right),$$ and \eqref{asymptoticexpansionvortex} follows. Next 
\[
\xi(r)=\frac{1}{r^{\gamma-2}}\frac{d}{dr}(r^{\gamma-2}\psi)=\frac{1}{r^{\gamma-2}}\frac{d}{dr}\phi=\frac{1}{r^{\gamma-2}}\frac{u}r=\frac{1}{r^{\gamma-2}}\left(e^{\l_-s}+{\rm lot}\right)=\frac{1}{r^{\gamma-2}}\frac{1+o(1)}{r^{|\l_-|}},
\]
and estimates for higher derivatives follows from \eqref{vnnrononr}.\\

\noindent{\bf step 5} Regularity. We now prove point (iv) of Proposition \ref{propvrotexprofile}. We compute using $\gamma=2+\frac{1}{m}$:
 \be
\label{formulasink}
\left|\begin{array}{l}
\frac{2(\gamma+2)}{\gamma-2}=2(1+4m) \\
\frac{\gamma+2}{\gamma-2}=1+4m, \\
\frac{4}{\gamma-2}=4m,\
\frac{\gamma}{\gamma-2}=1+2m,\
\frac{\gamma-1}{\gamma-2}=1+m
\end{array}\right..
\ee
Recall \eqref{vneiovnineoneneovnvnoe}
$$\pa_r^2\Psi+\frac{3}{r}\pa_r\Psi- r^{\frac{2(\gamma+2)}{\gamma-2}}\Psi^{\frac{\gamma+2}{\gamma-2}}+A(\gamma)r^{\frac{4}{\gamma-2}}\Psi^{\frac{\gamma}{\gamma-2}}=0,$$
 then \eqref{formulasink} and the normalization $\Psi(0)=1$ ensure that $\Psi$ is a smooth function of $r^2$. Moreover, $$u_\phi=\frac{\matchal C(\psi)}{r}=\frac{(r^2\Psi)^{\frac{\gamma-1}{\gamma-2}}}{r}=\frac{(r^2\Psi)^{1+m}}{r}=r(r^2\Psi)^m,$$ and the $\mathcal C^\infty$ regularity of the vector field \eqref{vneiovnevneinenove} at $r=0$ follows. Regularity for $r>0$ follows from the regularity of $\Psi_*$ which since $\Psi_*$ does not vanish follows from \eqref{vneiovnineoneneovnvnoe} and standard elliptic regularity. Finally observe that 
 \[
 K(\gamma)\underset{\gamma\to 2}{\sim}\frac{2A(\gamma)}{\gamma-2}\geq  \frac{1}{\gamma-2}\liminf_{\gamma'\downarrow 2} A(\gamma')\underset{\gamma\to 2}{\to} +\infty.
 \]
 This concludes the proof of Proposition \ref{propvrotexprofile}.
 

\subsection{Degenerate vortices in the limit $\gamma\downarrow 2$}
\label{degeneratevortex}


In this section, we revisit the proof of Proposition \ref{thmdegen} in the limit $\gamma\downarrow 2$. We give a complete description of the profile in this singular limit, which will be important for the construction of the self similar solution in section \ref{sectionselfsim}. Interestingly enough, the solution displays a violent {\em concentration ring} which is reminiscent from the study of limiting self similar profiles for the compressible Euler equations in \cite{MRRS19-1}.

\begin{proposition}[Shape of the degenerate vortex as $\gamma \downarrow 2$]
\label{propappendix}
Assume $0<\gamma-2<\e_*$ universal small enough, then up to scaling invariance, the solution to Proposition \ref{propvrotexprofile} is given by the following renormalization.\\

\noindent{\em 1. Value of $A(\gamma)$.}
$$
A(\gamma)=\frac{\sqrt{2}\gamma}{\sqrt{\gamma-2}}\left[1+o_{\gamma\to 2}(1)\right].
$$
In particular, \eqref{fneionfenvoinoe} holds.\\

\noindent{\em 2. Profile.} Let $$\left|\bear
b=\frac{\sqrt{\gamma-2}}{A(\gamma)}\\
y_{\gamma}=-\frac{\gamma-2}{2b\gamma}\log A(\gamma)
\ear\right.,
$$
then $$\left|\bear
r^\gamma\Psi_*(r)= \phi(s)\\
y=\frac{s+by_{\gamma}}{b}\\
\phi(s)=\left[h(y)\sqrt{A(\gamma)}\right]^{\gamma-2}\\
h(y)=h_0(y)\left[1+h_{\rm corr}(y)\right]\\
h_0(y)=\frac{1}{\left(1+e^{-\sqrt{2}y}\right)^{\frac 12}}
\ear\right.
$$
with $$\|h_{\rm corr}\|_{L^\infty}\lesssim b\underset{\gamma \to 2}{\lesssim}(\gamma-2).$$ 

\end{proposition}

\begin{proof}[Proof of Proposition \ref{propappendix}] We recall the Emden formulation of Lemma \ref{lemamphaseportrait} and \eqref{eq: ode system hom vortex}:
$$\left|\bear
s=\log r
, \ \phi=r^{\gamma-2}\psi\\
\phi''-2(\gamma-1)\phi'+\gamma(\gamma-2)\phi-\phi^{\frac{\gamma+2}{\gamma-2}}+A\phi^{\frac{\gamma}{\gamma-2}}=0.
\ear\right.
$$

\noindent{\bf step 1} Renormalization.
\bee
&&\phi''-2(\gamma-1)\phi'+\gamma(\gamma-2)\phi-\phi^{\frac{\gamma+2}{\gamma-2}}+A\phi^{\frac{\gamma}{\gamma-2}}=0\\
&\Leftrightarrow & \frac{\phi''}{\phi}-2(\gamma-1)\frac{\phi'}{\phi}+\gamma(\gamma-2)-\phi^{\frac{4}{\gamma-2}}+A\phi^{\frac{2}{\gamma-2}}=0\\
&\Leftrightarrow & \left(\frac{\phi'}{\phi}\right)'+\left(\frac{\phi'}{\phi}\right)^2-2(\gamma-1)\frac{\phi'}{\phi}+\gamma(\gamma-2)-\phi^{\frac{4}{\gamma-2}}+A\phi^{\frac{2}{\gamma-2}}=0.
\eee
Let $$\phi=(\kappa \hat{h})^{\gamma-2}, \ \ \kappa=\kappa(\gamma),$$ then equivalently:
\bee
&&(\gamma-2)\left(\frac{\hat{h}'}{\hat{h}}\right)'+\left[(\gamma-2)\frac{\hat{h}'}{\hat{h}}\right]^2-2(\gamma-1)(\gamma-2)\frac{\hat{h}'}{\hat{h}}+\gamma(\gamma-2)-(\kappa \hat{h})^4+A(\kappa \hat{h})^2=0\\
&\Leftrightarrow& \left(\frac{\hat{h}'}{\hat{h}}\right)'+(\gamma-2)\left(\frac{\hat{h}'}{\hat{h}}\right)^2-2(\gamma-1)\frac{\hat{h}'}{\hat{h}}+\gamma-\frac{\kappa^4}{\gamma-2} \hat{h}^4+\frac{A\kappa^2}{\gamma-2} \hat{h}^2=0.
\eee
We let $$\left|\begin{array}{l}
\kappa^2=A\\
b^2=\frac{\gamma-2}{A^2},
\end{array}\right.$$ and obtain the semi classical formulation:
\bee
&& \left(\frac{\hat{h}'}{\hat{h}}\right)'+(\gamma-2)\left(\frac{\hat{h}'}{\hat{h}}\right)^2-2(\gamma-1)\frac{\hat{h}'}{\hat{h}}+\gamma+\frac{A^2}{\gamma-2}\left[- \hat{h}^4+\hat{h}^2\right]=0\\
&\Leftrightarrow& b^2\left[ \left(\frac{\hat{h}'}{\hat{h}}\right)'+(\gamma-2)\left(\frac{\hat{h}'}{\hat{h}}\right)^2-2(\gamma-1)\frac{\hat{h}'}{\hat{h}}+\gamma\right]-\hat{h}^4+\hat{h}^2=0.
\eee
We let $$\hat{h}=e^{-\hat{H}}$$ so that 
$$b^2\left[ -\hat{H}''+(\gamma-2)(\hat{H}')^2+2(\gamma-1)\hat{H}'+\gamma\right]-e^{-4\hat{H}}+e^{-2\hat{H}}=0.$$ We then renormalize $$\left|\begin{array}{l}
\hat{H}(s)\equiv \tilde{H}(\tilde{y})\\
\tilde{y}=\frac{s}{b},
\end{array}\right.
$$
and obtain the equivalent flow 
\be
\label{vnenvenvenove}
\tilde{H}''-(\gamma-2)(\tilde{H}')^2-2b(\gamma-1)\tilde{H}'-b^2\gamma+e^{-4\tilde{H}}-e^{-2\tilde{H}}=0.
\ee

\noindent{\bf step 2} Boundary condition at the origin. We normalize $$\Psi=\frac{\phi}{r^\gamma}\to 1\ \ \mbox{as}\ \ r\to 0,$$ and compute:
$$
\log\left(\frac{\phi}{r^\gamma}\right)=-\gamma\log r+(\gamma-2)\log (\kappa \hat{h})=-b\gamma \tilde{y}+(\gamma-2)\log \kappa -(\gamma-2)\tilde{H},
$$
and hence the boundary condition 
$$
\tilde{H}=-\frac{b\gamma\tilde{y}}{\gamma-2}+\log \kappa+o_{y\to -\infty}(1).
$$
Let $$\left|\begin{array}{l} \tilde{H}(\tilde{y})=H(y)\\
y=\tilde{y}+y_{\gamma}\\
y_{\gamma}=-\frac{\gamma-2}{b\gamma}\log \kappa=-\frac{\gamma-2}{2b\gamma}\log A.
\end{array}\right.
$$
then the translation invariance of \eqref{vnenvenvenove} yields the equivalent equation:
\be
\label{tobesolved}
\left|\begin{array}{l}
H''-(\gamma-2)(H')^2-2b(\gamma-1)H'-b^2\gamma+e^{-4H}-e^{-2H}=0\\
H(y)=-\frac{b\gamma y}{\gamma-2}+o_{y\to -\infty}(1).
\end{array}\right.
\ee
We collect the changes of variables
$$y=\tilde{y}+y_{\gamma}=\frac{s}{b}+y_{\gamma}=\frac{s+by_{\gamma}}{b}=\frac{s-\frac{\gamma-2}{\gamma}\log \kappa}{b}.$$

\noindent{\bf step 3} Limiting profile. Let $H_0$ solve 
\be
\label{vneivneoneovn}
H_0''+e^{-4H_0}-e^{-2H_0}=0.
\ee
The phase portrait reveals that there is a unique solution with $$\lim_{y\to +\infty}(H_0',H_0)=(0,0),$$ and in fact explicitly
$$\frac{(H'_0)^2}{2}-\frac{e^{-4H_0}}{4}+\frac{e^{-2H_0}}{2}=\frac14\Leftrightarrow (H_0')^2=\frac{(1-e^{-2H_0})^2}{2}$$  and hence looking for $H_0>0\Longrightarrow H_0'<0$: 
$$\frac{H_0'}{1-e^{-2H_0}}=-\frac{\sqrt{2}}{2}\Leftrightarrow \frac 12\log(e^{2H_0}-1)=-\frac{\sqrt{2}}{2}y$$
 yields the limiting profile :
\begin{equation}\label{formulahzero}
\left|\begin{array}{l}
H_0(y)=\frac12\log\left(1+e^{-\sqrt{2}y}\right)\\
h_0=e^{-H_0}=\frac{1}{\left(1+e^{-\sqrt{2}y}\right)^{\frac 12}}.
\end{array}\right.
\end{equation}
In view of \eqref{tobesolved} and consistently with the fact that $H_0$ is the leading order term, we have to leading order
\be
\label{ahuiteynsnojovuors}
\left|\begin{array}{l}
\frac{\sqrt{2}}2|y|=\left[\frac{b \gamma }{\gamma-2}+{\rm lot}\right]|y|\Rightarrow b=\frac{\sqrt{2}}{2}\frac{\gamma-2}{\gamma}+{\rm lot}\\
A=\frac{\sqrt{\gamma-2}}{b}=\frac{\sqrt{2}\gamma}{\sqrt{\gamma-2}}+{\rm lot}\\
y_{\gamma}=-\frac{\gamma-2}{b \gamma }\log \kappa=-\frac{\sqrt{2}}{2}\log A=-\frac{\sqrt{2}}{2}\log \left(\frac{\sqrt{\gamma-2}}{\sqrt{2}\gamma}\right).
\end{array}\right.
\ee

\noindent{\bf step 4} Linearized operator close to $H_0$. We set up the functional setting to invert the linearized operator.\\

\noindent\und{\em E space}. Pick a small universal loss $0<\delta\ll 1$. We let 
$$E=\{u\in \mathcal C^2(\Bbb R,\Bbb R), \ \|u\|_E<+\infty\},$$ with 
\bea
\label{defnorme}
\|u\|_E&=&\|e^{(\sqrt{2}-\delta)|y|}u''\|_{L^\infty(\Bbb R_-)}+\|u'\|_{L^\infty(\Bbb R_-)}+\left\|\frac{u}{\la y\ra}\right\|_{L^\infty(\Bbb R_-)}\\
\nonumber &+& \|e^{(\sqrt{2}-\delta)|y|}u''\|_{L^\infty(\Bbb R_+)}+\|e^{(\sqrt{2}-\delta)|y|}u'\|_{L^\infty(\Bbb R_+)}+\left\|u\right\|_{L^\infty(\Bbb R_+)}.
\eea
Let $u\in E$, then $u''\in L^1$ and hence $\displaystyle u'_{-\infty}=\lim_{y\to -\infty} u'(y)$ exists and is well defined with 
\be
\label{estderviative}
\left|\begin{array}{l}
|u'_{-\infty}|\lesssim \|u\|_E\\
|u'(y)-u'_{-\infty}|=\left|\int_{-\infty}^yu''d\tau\right|\lesssim e^{-(\sqrt{2}-\delta)|y|}\|u\|_E.
\end{array}\right.
\ee
 Hence for $y<0$:
\bee
u(0)-u(y)&=&\int_y^0 u'(t)dt=\int_y^0 \left[u'(t)-u'_{-\infty}\right]dt+u'_{-\infty}y\\
&=&u'_{-\infty}y+ \int_{-\infty}^0 \left[u'(t)-u'_{-\infty}\right]dt+O\left(\|u\|_Ee^{-(\sqrt{2}-\delta)|y|}\right),
\eee
and hence there exists $u_{-\infty}', u_{-\infty}$ with 
\be
\label{esatoenoineneg}
\left|\begin{array}{l}
|u_{-\infty}'|+|u_{-\infty}|\lesssim \|u\|_E, \ \ i=0,1,\\
\left|u(y)-u'_{-\infty}y-u_{-\infty}\right|\lesssim \|u\|_E e^{-(\sqrt{2}-\delta)|y|}, \ \ y<0.
\end{array}\right.
\ee
Similarly on the right, 
\be
\label{esatoenoinenegbis}
\left|\begin{array}{l}
|u_{+\infty}|\lesssim \|u\|_E,\\
\left|u(y)-u_{+\infty}\right|\lesssim \|u\|_Ee^{-(\sqrt{2}-\delta)|y|}, \ \ y>0.
\end{array}\right.
\ee
\noindent\und{\em $E'$ space}. 
We define similarly $$E'=\{F\in \mathcal C^1(\Bbb R,\Bbb R), \ \|F\|_{E'}<+\infty\},$$ with 
\bea
\label{defeprime}
\|F\|_{E'}&=&\|e^{(\sqrt{2}-\delta)|y|}F'\|_{L^\infty(\Bbb R_-)}+\left\|e^{(\sqrt{2}-\delta)|y|}F\right\|_{L^\infty(\Bbb R_-)}\\
&+&\|e^{(\sqrt{2}-\delta)|y|}F'\|_{L^\infty(\Bbb R_+)}+\left\|F\right\|_{L^\infty(\Bbb R_+)}.
\eea
We have similarly the representation $$\left|\begin{array}{l}
F=F_{+\infty}+O\left(\|F\|_{E'}e^{-(\sqrt{2}-\delta)y}\right), \ \ y>0\\
|F_{+\infty}|\lesssim \|F\|_{E'}.
\end{array}\right.
$$

\noindent{\bf step 5} Resolvent estimate. Let the linearized operator be $$L_0=-\frac{d^2}{dy^2}+4e^{-4H_0}-2e^{-2H_0}\\
.$$ 
Let $F\in E'$, we claim that there exists a unique solution to 
\be
\label{venoivbnenoveinvo}
\left|\begin{array}{l}
L_0u=F\\
u\in E\\
u_{-\infty}=0,
\end{array}\right.
\ee
and 
\be
\label{bnoigeneoipotinrtesolvetn}
\|u\|_E\lesssim \|F\|_{E'}.
\ee
\noi {\em Proof of \eqref{bnoigeneoipotinrtesolvetn}} This will follows from the explicit representation of the resolvent.\\

\noindent\und{\em Resolvent formula}. By translation invariance of \eqref{vneivneoneovn}:
$$\left|\begin{array}{l}
L_0H'_0=0\\
H_0'=-\frac{\sqrt{2}}{2}\frac{1}{e^{\sqrt{2}y}+1}.
\end{array}\right.
$$ Since $H_0'$ does not vanish, there holds the factorization formula
$$\left|\begin{array}{l}
L_0=A^*A\\
Au=-u'+\frac{H_0''}{H'_0}u=-H'_0\left(\frac{u}{H'_0}\right)'\\
A^*u=u'+\frac{H_0''}{H'_0}u=\frac{(uH'_0)'}{H'_0}.
\end{array}\right.
$$
Hence a solution to $L_0H=F$ is given by $$\left|\begin{array}{l}
Au=-\frac{1}{H'_0}\int_{y}^{+\infty}FH'_0\\
u=c_0H_0'-H_0'\int_{0}^{y}\frac{Au}{H_0'}.
\end{array}\right.
$$
Moreover this ensures that $$\left|\begin{array}{l}
L_0u=0\\
u\in E\\
u_{-\infty}=0
\end{array}\right. \Rightarrow \left|\begin{array}{l}
u=cH_0'\\
u_{-\infty}=0
\end{array}\right. \Rightarrow u=0,
$$
which implies that there exists at most one solution to \eqref{venoivbnenoveinvo}.\\

\noindent\und{\em Choice of $c_0$ and estimates}.\\

\noindent\underline{$y<0$}. We estimate:
\bee
Au&=& -\frac{1}{H'_0}\int_{y}^{+\infty}FH'_0d\tau=-\frac{1}{-\frac{\sqrt{2}}{2}+O(e^{-\sqrt{2}|y|})}\int_y^{+\infty} FH_0'd\tau\\
& = & \sqrt{2}\int_{-\infty}^{\infty}FH_0'd\tau+O\left(\|F\|_{E'}  e^{-(\sqrt{2}-\delta)|y|}\right),
\eee
and 
\bee
u&=&c_0H_0'-H_0'\int_{0}^{y}\frac{Au}{H_0'}=-\frac{\sqrt{2}}{2}c_0\left[1+O(e^{-\sqrt{2}|y|})\right]\\
&-&\frac{\sqrt{2}}{2}\left[1+O(e^{-\sqrt{2}|y|})\right]\int_y^0\frac{(Au)_{-\infty}+Au-(Au)_{-\infty}}{-\frac{\sqrt{2}}{2}+O(e^{-\sqrt{2}|\tau|})}d\tau\\
& = & -(Au)_{-\infty}y-\frac{\sqrt{2}}{2}c_0+\int_{-\infty}^{0}(Au-(Au)_{-\infty})d\tau+O\left[(|c_0|+\|F\|_{E'}) e^{-(\sqrt{2}-\delta)|y|}\right],
\eee
and we freeze the choice $$c_0=\frac{2}{\sqrt{2}}\int_{-\infty}^{0}(Au-(Au)_{-\infty})d\tau=O(\|F\|_{E'}),$$ which ensures $u_{-\infty}=0.$ This ensures:
$$\left|\begin{array}{l}
u'=-Au+\frac{H_0''}{H_0'}u=O(\|F\|_{E'})\\
u''=F-4e^{-4H_0}u+2e^{-2H_0}u=O\left(\|F\|_{E'} e^{-(\sqrt{2}-\delta)|y|}\right).
\end{array}\right.
$$

 \noindent\underline{$y>0$}. We estimate:
\bee
Au&=& -\frac{1}{H'_0}\int_{y}^{+\infty}FH'_0d\tau\\
&=&-e^{\sqrt{2}y}\left[1+O(e^{-\sqrt{2}y})\right]\int_y^{+\infty}\left[F_{+\infty}+O(\|F\|_{E'}e^{-(\sqrt{2}-\delta)\tau})\right]e^{-\sqrt{2}\tau}\left[1+O(e^{-\sqrt{2}\tau})\right]d\tau\\
&=&-\frac{F_{+\infty}}{\sqrt{2}}+O\left(\|F\|_{E'}e^{-(\sqrt{2}-\delta)y}\right)
\eee
and 
\bee
u&=&c_0H_0'-H_0'\int_{0}^{y}\frac{Au}{H_0'}=c_0e^{-\sqrt{2}y}\left[1+O(e^{-\sqrt{2}y})\right]\\
&-&e^{-\sqrt{2}y}\left[1+O(e^{-\sqrt{2}y})\right]\int_0^{y}\left[-\frac{F_{+\infty}}{\sqrt{2}}+O\left(\|F\|_{E'}e^{-(\sqrt{2}-\delta)\tau}\right)\right]e^{\sqrt{2}\tau}\left[1+O(e^{-\sqrt{2}\tau})\right]d\tau\\
&=& \frac{F_{+\infty}}{2}+O\left(\|F\|_{E'}\la y\ra e^{-(\sqrt{2}-\delta)y}\right)
\eee
which ensures:
$$\left|\begin{array}{l}
u'=-Au+\frac{H_0''}{H_0'}u=O\left(\|F\|_{E'} e^{-(\sqrt{2}-\delta)y}\right)\\
u''=F-4e^{-4H_0}u+2e^{-2H_0}u=O\left(\|F\|_{E'}e^{-(\sqrt{2}-\delta)|y|}\right).
\end{array}\right.
$$
\noindent\und{\em Conclusion}. The collection of the above bounds yields \eqref{bnoigeneoipotinrtesolvetn} and concludes the proof of \eqref{bnoigeneoipotinrtesolvetn}.\\

\noindent{\bf step 6} Non linear fixed point. We now solve \eqref{tobesolved}\[
H_{sol}''-(\gamma-2)(H_{sol}')^2-2b(\gamma-1) H_{sol}'-b^2\gamma+e^{-4H_{sol}}-e^{-2H_{sol}}=0,\]
by a fixed point argument near $H_0$.\\

\noi\und{Boundary condition}. Observe that $H_{sol}\in E$ imposes by taking the limit $y\to -\infty$
\[
b^2\gamma+2(\gamma-1) {H_{sol}}'_{-\infty}b+(\gamma-2)({H_{sol}}'_{-\infty})^2=0,
\]
thus $b=-{H_{sol}}'_{-\infty} \text{ or } (\gamma-2)\frac{-{H_{sol}}'_{-\infty}}{\gamma},$ and we choose the latter. The equation then becomes with $\epsilon=\gamma-2$
\bee
&&H_{sol}''-\epsilon\left[(H_{sol}')^2-({H_{sol}}'_{-\infty})^2+\frac{2(\gamma-1)}{\gamma}{H_{sol}}'_{-\infty} \left(H_{sol}'-{H_{sol}}'_{-\infty}\right)\right]\\
&+&e^{-4H_{sol}}-e^{-2H_{sol}}=0.\eee

\noi\und{Fixed point formulation}. We look for a solution in the form $H_{sol}=H_0+\epsilon H$ which after division by $\epsilon$  yields the equation on $H$ 
$$\left|\bear
L_0 H=F(H,\epsilon)=F_0(H_0)+2\epsilon A H+\epsilon^2 B(H)+N(H,\epsilon)\\
F_0(H_0)=[(H_0')^2-({H_0}'_{-\infty})^2]+\frac{2(\gamma-1)}{\gamma}{H_0}'_{-\infty} \left(H_0'-{H_0}'_{-\infty}\right)\\
AH=(H_0'-{H_0}'_{-\infty})\left(H'+\frac{\gamma-1}{\gamma}H'_{-\infty}\right)+\frac{2\gamma-1}{\gamma}{H_0}'_{-\infty}(H'-H'_{-\infty})\\
B(H)=(H')^2-(H'_{-\infty})^2+\frac{2(\gamma-1)}{\gamma}H'_{-\infty} \left(H'-H'_{-\infty}\right)\\
N(H,\epsilon)=e^{-4 H_0}\frac{e^{-4\epsilon H}-1+4\epsilon H}{\epsilon}-e^{-2 H_0}\frac{e^{-2\epsilon H}-1+2\epsilon H}{\epsilon}.
\ear\right.
$$

\noindent{\bf step 7} Banach fixed point.  The proof of Proposition \ref{propappendix} now follows from a standard application of the Picard fixed point Theorem. Indeed, we claim that there exists $M$ a large enough universal constant such that for all $0<\epsilon<\epsilon^*$ small enough, the mapping $$H \mapsto L_0^{-1}F(H,\epsilon)$$ is a contraction on 
\[B_{M}=\{H\in E, \left\Vert H\right\Vert_{E}\leq M \text{ and }H_{-\infty}=0\}.\] 
We give estimates as well as difference estimates for each term in $L^{-1}F(H,\epsilon)$ separately with $H\in B_M$.\\

\noindent\und{\em Inhomogeneous terms.}  We directly observe that $F_0(H_0)\in E'$ is ensured by the choice of $b$. Hence from \eqref{bnoigeneoipotinrtesolvetn}, we may take $M$ large enough such that  
\[
2 \left\Vert L^{-1}\right\Vert_{E'\to E} \left\Vert F_0(H_0)\right\Vert_{E'}\leq M.
\]

\noindent\und{\em  Linear and quadratic terms.} Again the choice of $b$ ensures $AH$ and $B(H)\in E'$ thus
\[
\left\Vert L^{-1}A H \right\Vert_{E}\lesssim  \left\Vert L^{-1}\right\Vert_{E'\to E}\left\Vert A H \right\Vert_{E'} \leq C M,
\]
\[
\left\Vert L^{-1}B(H) \right\Vert_{E}\lesssim  \left\Vert L^{-1}\right\Vert_{E'\to E}\left\Vert B(H) \right\Vert_{E'} \leq C M^2,
\]
and for $\tilde{H}\in B_M$
\[
\left\Vert L^{-1}\left(B(H)-B(\tilde{H})\right)\right\Vert_{E}\lesssim  \left\Vert L^{-1}\right\Vert_{E'\to E}\left\Vert B(H) -B(\tilde{H})\right\Vert_{E'} \leq C M \left\Vert H-\tilde{H}\right\Vert_{E}.
\]
\noindent\und{\em  Exponential terms.} It suffices to study 
\[
L^{-1}\left(e^{-c H_0}\frac{e^{-c\epsilon H}-1+c\epsilon H}{\epsilon}\right)=\epsilon L^{-1}\left(\frac{1}{(1+e^{-\sqrt{2}y})^{\frac{c}{2}}}\int^{1}_{0}e^{-c\epsilon H s}dsH^2\right),
\]
with $c\in \{2,4\}$. And observe that by choosing the small loss $\delta$ such that 
\[
\frac{c}{2}\sqrt{2}-c  \epsilon C M>\sqrt{2}-\delta \Leftrightarrow \delta>4\epsilon C M,
\]
we have 
\[
\left\Vert L^{-1}\left(\frac{1}{(1+e^{-\sqrt{2}y})^{\frac{c}{2}}}\int^{1}_{0}e^{-c\epsilon H s}dsH^2\right) \right\Vert_{E} \leq C e^{4\epsilon M} M^2,
\]
and analogously 
\[
\left\Vert L^{-1}\left(N(H,\epsilon)-N(\tilde{H},\epsilon)\right)\right\Vert_{E}\leq C \epsilon M e^{4\epsilon M} \left\Vert H-\tilde{H}\right\Vert_{E}. 
\]
\noindent\und{\em Conclusion.} The collection of the bounds give 
\[
\left\Vert L^{-1}F(H,\epsilon) \right\Vert_{E} \leq \left(\frac{1}{2}+\epsilon C+C\epsilon^2 M+C \epsilon e^{4\epsilon M}M  \right)M,
\]
and 
\[
\left\Vert L^{-1}\left(F(H,\epsilon)-F(\tilde{H},\epsilon)\right)\right\Vert_{E} \leq C \epsilon \left[1+M (1+e^{4\epsilon M})\right]\left\Vert H-\tilde{H}\right\Vert_{E},
\]
which give the desired result for $\epsilon$ small enough.

\end{proof}


\section{Energy estimates for the resolvent}
\label{sectionresolventone}


We now start the proof of the vortex construction of Therorem \ref{thmmain}. This section is devoted to the derivation of energy estimates for the resolvent of the linearized operator close to the homogeneous vortex $\psi_\nu$ using a Lax-Milgram type argument. From now on, we fix a small enough $0<a\ll1$.

 
 \subsection{Linearized operator close to $\psi_*$}
 
Define
$$M_*\psi=-\frac{1}{R}\pa_R\left(\frac{1}{R}\pa_R\psi\right)+  \left[\mathcal H'_0-\frac{(\C_0')^2+\C_0\C_0''}{R^2}\right](\psi_*)\psi$$ acting on functions $$\psi=R^2\Psi, \ \ \Psi\in \mathcal D_{\rm rad}(\Bbb R^2),$$
where $D_{\rm rad}(\Bbb R^2)$ are smooth compactly supported radially symmetric functions vanishing at the origin. 
\begin{lemma}[Coercivity of $M$]
\label{coercm}
Assume $2<\gamma<2+\epsilon_*$ for some small enough universal constant $0<\epsilon_*\ll1$.  Let 
$$\zeta_*(R)=(\gamma-2)\psi_*+r\pa_r\psi_*,$$ Then there exists $c_{*}=c_*(\gamma)>0$ such that
$$\int_{R>0}(M_*\psi)\psi RdR\ge c_{*}\int_{R>0}\left(\frac{|\pa_R\psi|^2}{R^2}+\frac{\psi^2}{R^4}\right)RdR-\frac1{c_{*}}\left(\int_{R>0} \psi \zeta_*RdR\right)^2.$$
\end{lemma}

\begin{proof}[Proof of Lemma \ref{coercm}] This will follow from the positivity of $M_*$ up to a compact perturbation. We recall from \eqref{asymptoticexpansionvortex}
\be
\label{vneinveonenenoev}
\zeta_*(R)=\frac{1+o_{R\to +\infty}(1)}{R^{\gamma-2+|\l_-|}},
\ee where from \eqref{defnumberK}, \eqref{fneionfenvoinoe}:
\be
\label{funoenioenone}
\lim_{\gamma\downarrow 2}|\l_-|=+\infty.
\ee

\noindent{\bf step 1} Hardy inequality. We claim that for all  $$\psi =R^2\Psi, \ \ \Psi\in \mathcal D_{\rm rad}(\Bbb R^2),$$ there holds 
\be
\label{Hardyenienige}
\int_{R>0}\frac{|\pa_R\psi|^2}{R^2}RdR\ge \int_{R>0}\frac{\psi^2}{R^4}RdR.
\ee
Indeed, we integrate by parts using the vanishing at the origin
\bee
&&\int_{R>0}\frac{\psi^2}{R^4}RdR=\int_{R>0}\frac{\psi^2}{R^3}dR=-\frac 12\left[\frac{\psi^2}{R^2}\right]_0^{+\infty}+\int_{R>0}\frac{\psi\pa_R\psi}{R^3}RdR\\
&\leq& \left(\int_{R>0}\frac{\psi^2}{R^4}RdR\right)^{\frac 12}\left(\int_{R>0}\frac{|\pa_R\psi|^2}{R^2}RdR\right)^{\frac 12},
\eee
and the conclusion follows.\\

\noindent{\bf step 2} Decay of the potential term. We compute
\bee
V_*(R)&=& \mathcal H_0''(\psi_*(R))-\frac{(\C'_0)^2(\psi_*(R))+\C_0\C_0''(\psi_*(R))}{R^2}\\
&=& \frac{\gamma+2}{\gamma-2} \psi_*^{\frac{\gamma+2}{\gamma-2}-1}-\frac 1{R^2}A(\gamma)\frac{\gamma}{\gamma-2}\psi_*^{\frac{\gamma}{\gamma-2}-1}\\
& = &  \frac{\gamma+2}{\gamma-2} \psi_*^{\frac{4}{\gamma-2}}-\frac{\gamma A(\gamma)}{\gamma-2}\frac{\psi_*^{\frac{2}{\gamma-2}}}{R^2}.
\eee
We have at the origin $$|V_*(R)|\leq C R^{\frac{4}{\gamma-2}-2},$$ and near $+\infty$:
\bee
&&V_*(R)=\frac{\gamma+2}{\gamma-2}\left(\frac{A_\infty(\gamma)(1+o(1))}{R^{\gamma-2}}\right)^{\frac{4}{\gamma-2
}}-\frac{A(\gamma)\gamma}{(\gamma-2)R^2}\left(\frac{A_\infty(\gamma)(1+o(1))}{R^{\gamma-2}}\right)^{\frac{2}{\gamma-2
}}\\
& =& \frac{1}{R^4}\left\{\frac{A_\infty(\gamma)^{\frac{2}{\gamma-2}}}{\gamma-2}\left[(\gamma+2)A_\infty(\gamma)^{\frac{2}{\gamma-2}}-\gamma A(\gamma)\right]\right\}=\frac{K}{R^4},
\eee
with $K>0$ given by \eqref{defnumberK}. 
Hence we have the following lower bound for some $c,\delta>0$
$$\int_{R>0}(M_*\psi)\psi RdR\ge c\int_{R>0}\left(\frac{|\pa_R\psi|^2}{R^2}+\frac{\psi^2}{R^4}\right)RdR-\frac 1c\int_{R>0}\frac{\psi^2}{R^2(1+R^{2-\delta})}RdR.$$ \\

\noindent{\bf step 3} Limiting profile. Arguing by contradiction, we suppose there exists a sequence such that $$\left|\begin{array}{l}
\int_{R>0}\left(\frac{|\pa_R\psi|_n^2}{R}+\frac{\psi_n^2}{R^4}\right)RdR=1\\
\int_{R>0} (M_*\psi_n)\psi_nRdR\le \frac 1n\\
\int_{R>0}\psi_n\zeta_*RdR=0.
\end{array}\right.
$$
Hence 
\bee
    \frac 1n&\ge& c\int_{R>0}\left(\frac{|\pa_R\psi_n|^2}{R}+\frac{\psi_n^2}{R^4}\right)RdR-\frac 1c\int_{R>0}\frac{\psi_n^2}{R^2(1+R^{2-\delta})}RdR\\
    &=&c-\frac 1c\int_{R>0}\frac{\psi_n^2}{R^2(1+R^{2-\delta})}RdR,
    \eee
    which implies $$\int_{R>0}\frac{\psi_n^2}{R^2(1+R^{2-\delta})}RdR\gtrsim 1.$$
Passing to a weak limit and using the positivity of the potential at infinity, weak lower semi continuity, Sobolev embedding to recover compactness locally in space and the strong decay \eqref{vneinveonenenoev} and \eqref{funoenioenone} for $\gamma$ close enough to 2 yields 
\be
\label{veoneonnevvneo}
\left|\begin{array}{l}
\int_{R>0}\left(\frac{|\pa_R\psi|^2}{R^2}+\frac{\psi^2}{R^4}\right)RdR\le 1\\
\int_{r>0}\frac{\psi^2}{R^2(1+R^{2-\delta})}RdR\gtrsim 1\\
\int_{R>0}(M_*\psi)\psi R dR\leq 0,
\end{array}\right.
\ee
as well as
\be
\label{neioheinvnbivoeopi}
\int_{R>0} \psi \zeta_*RdR=\lim_{n\to +\infty}\int_{R>0} \psi_n \zeta_*RdR=0.
\ee

\noindent{\bf step 4} Conclusion. The scale invariance $\psi_\l=\frac{1}{\l^{\gamma-2}}\psi\left(\frac{r}{\l}\right)$ of \eqref{degenerateode} and \eqref{nondgeenera} ensures that 
$$\left|\begin{array}{l}
M\zeta_*=0\\
 \zeta_*=(\gamma-2)\psi_*+r\partial_r \psi_*>0.
 \end{array}\right.
 $$ 
Let $$\left|\begin{array}{l} A=-\frac1r\pa_r+W\\A^*=\frac 1r\pa_r+W\end{array}\right.,\ \ W=\frac{1}{r}\frac{\pa_r\zeta^*}{\zeta^*},$$
then
\bee
A^*A \psi& = &\left(\frac 1r\pa_r+W\right)\left(-\frac1r\pa_r\psi+W\psi\right)\\
&=&-\frac{1}{r}\pa_r\left(\frac 1r\pa_r\psi\right)+\frac{\pa_r W}r\psi+\frac{W}{r}\pa_r\psi-\frac{W}{r}\pa_r\psi+W^2\psi\\
& =& -\frac{1}{r}\pa_r\left(\frac 1r\pa_r\psi\right)+\psi\left[\frac 1r\pa_r\left(\frac{1}{r}\frac{\pa_r\zeta^*}{\zeta^*}\right)+\left(\frac{1}{r}\frac{\pa_r\zeta^*}{\zeta^*}\right)^2\right]\\
& = &  -\frac{1}{r}\pa_r\left(\frac 1r\pa_r\psi\right)+\psi\left[\frac{1}{\zeta_*}\frac 1r\pa_r\left(\frac{1}{r}\pa_r\zeta^*\right)\right]=M\psi,
\eee
and hence
\be
\label{factorizationgona}
M=A^*A.
\ee
Hence \eqref{veoneonnevvneo} forces $\psi=c\zeta_*$, $c\neq 0$, a contradiction to \eqref{neioheinvnbivoeopi}.
\end{proof}


\subsection{$\psi_\nu$ profile}


The scaling invariance \eqref{vniovneivnoneoiscalign} of \eqref{homvortedeuqation} ensures that for {\rm any} $\nu\in \mathcal C^2(\Bbb R,\Bbb R^*_+)$, 
 \be
\label{cneonveoneonvmemnvemeve}
\psi_\nu(r,z)=\frac{1}{\nu^{\gamma-2}(z)}\psi_*\left(\frac r{\nu(z)}\right)
\ee
is a solution to the degenerate flow \eqref{homvortedeuqation} for the homogeneous functionals:
\be
\label{vneoinonnrinrbno}\left|\begin{array}{l}
\mathcal C_0(\psi)=\pm \frac{\gamma-2}{\gamma-1} \sqrt{A(\gamma)}\psi^{\frac{\gamma-1}{\gamma-2}}, \\
 \mathcal H_0'(\psi)=\psi^{\frac{\gamma+2}{\gamma-2}}, \\
 \end{array}\right.
 \ee

We shall from now on assume that $\nu$ belongs to the class of admissible functions:

\begin{definition}[Admissible potentials functions] Let $(\eta,c_0,(C_i)_{0\le i\le 2})\in \Bbb (R_+^*)^5$. We say that $\nu\in \mathcal N^\eta(c_0,(C_i)_{0\le i\le 2})$ if\\
(i) $\nu\in \mathcal C^2(\Bbb R,\Bbb R_*^+)$;\\
(ii) 
\be
\label{neovnonvievvnkvnoneii}
\forall z\in \Bbb R, \ \ c_0\la z\ra^{\eta}\le \nu(z)\le C_0 \la z\ra^{\eta};
\ee
(iii) for $k=1,2$:
\be
\label{estoatmgnoi}
\frac{1}{\nu}\left|\frac{d^k\nu}{dz^k}\right|\le \frac{C_{k}}{\la z\ra^{k}}.
\ee
\end{definition}


\subsection{The linearized operator}  


We set up the analysis of the linearized operator close to $\psi_\nu$ given by \eqref{cneonveoneonvmemnvemeve}. We systematically use in the sequel the notation $$R=\frac{r}{\nu(z)}.$$

\noindent{\em Linearized operator.} We define
$$\left|\begin{array}{l}
L_a\psi= M\psi-\frac{a^2}{r^2}\pa_z^2\psi\\
M\psi=-\frac{1}{r}\pa_r\left(\frac{1}{r}\pa_r\psi\right)+  \left[\mathcal H''_0-\frac{(\C_0')^2+\C_0\C_0''}{r^2}\right](\psi_\nu)\psi
\end{array}\right.
$$

\noindent{\em Hilbert space and weight}. We let $H_a$ be the closure of $$\mathcal C_{{\rm cyl}}=\{r^2\Psi, \Psi\in \mathcal D_{{\rm cyl}} (\Bbb R^5)\},$$ (where $\mathcal D_{{\rm cyl}}$ refers to test functions with cylindrical symmetry) for the norm induced by the scalar product $$
\la u,v\ra_{H_a}=\int_{r>0}\int_{z\in \Bbb R}\frac{\pa_ru\pa_rv+a^2\pa_zu\pa_zv}{r^2}rdrdz.
$$

\noindent{\em Renormalized kernel}. Using \eqref{neovnonvievvnkvnoneii}, we pick $p$ integer large enough such that
\be
\label{defhshth}
h(r,z)=\frac{1}{\nu^p(z)}\zeta_*(R)\in H_a.
\ee

\noindent{\em The closed subset $V$}. Given $z_0\in \Bbb R$, we define the linear form
$$T_{z_0} \psi=\int_{r>0} \psi(r,z_0)h(r,z_0)rdr.$$ We claim that 
\be
\label{neoineonvoienoiev}
V=\cap_{z\in \Bbb R}{\rm Ker}T_{z} \ \ \mbox{is a closed subset of $H_a$}.
\ee 
\noindent{ Proof of \eqref{neoineonvoienoiev}}. For $\psi\in \mathcal D_{\rm cyl}(\Bbb R^5)$,
\bee
\nonumber
& &\psi(r,z_0)^2=-\int_{z_0}^{+\infty}\pa_{z}\left(\psi(r,z)^2\right)dz\lesssim  \left(\int_{z\in \Bbb R}(\pa_z\psi)^2dz\right)^{\frac 12}\left(\int_{z\in \Bbb R}\psi^2dz\right)^{\frac 12},\eee
and hence
$$|\psi(r,z_0)|\lesssim \left[\left(\int_{z\in \Bbb R}(\pa_z\psi)^2dz\right)^{\frac 12}+\left(\int_{z\in \Bbb R}\psi^2dz\right)^{\frac 12}\right].$$
 From which we deduce from the fast decay of $\xi_*$ in $R$:
\bee
|T_{z_0}(\psi)|&\leq&C(z_0)\int_{r>0} \left[\left(\int_{z\in \Bbb R}(\pa_z\psi)^2dz\right)^{\frac 12}+\left(\int_{z\in \Bbb R}\psi^2dz\right)^{\frac 12}\right]h(r,z_0)rdr\\
&\leq& C(z_0)\left(\int_{r>0}\int_{z\in\mathbb{R}}(\pa_z\psi)^2rdrdz\right)^{\frac 12}\left(\int_{r>0}h^2(r,z_0)rdr\right)^{\frac 12}\\
&+& C(z_0)\left(\int_{r>0}\int_{z\in\mathbb{R}}\frac{\psi^2}{r^4}rdrdz\right)^{\frac 12}\left(\int_{r>0}r^4h^2(r,z_0)rdr\right)^{\frac 12}\\
&\leq & C(z_0)\|\psi\|_{H_a},
\eee
and hence $T_{z_0}$ uniquely extends as a continuous linear form on $H_a$, and \eqref{neoineonvoienoiev} follows.

\subsection{Resolvent of $L_a$ on $V$}

We are now in position to invert $L_a$ on $V$ \`a la Lax Milgram.

\begin{lemma}[Resolvent of $L_a$ on $V$]
\label{vneokvneneonvie}
Let 
\be
\label{assumptitoinf}
r^2f\in L^2\left(rdrdz\right),
\ee 
then there exists a unique $\phi\in V$ such that 
\be
\label{estaimteV}
\forall v\in V, \ \ \int_{r>0}\int_{z\in \Bbb R}(L_a\phi)vrdrdz=\int_{r>0}\int_{z\in \Bbb R} fvrdrdz.
\ee
\noindent{\em Estimates}. 
\be
\label{estiaotinetpsibis}
\int_{r>0}\int_{z\in \Bbb R} \left[\frac{|\pa_r\phi|^2}{r^2} +\frac{a^2}{r^2}(\pa_z\phi)^2\right]rdrdz\lesssim \int_{r>0}\int_{z\in \Bbb R} r^4f^2rdrdz.
\ee
\noindent{\em Equation}. Let 
\be
\label{defdfz}
D_f(z)=\frac{\int_{r>0}\left[-\frac{a^2}{r^2}\phi\pa_z^2h-fh\right]rdr}{\int_{r>0}h^2(r,z)r dr},
\ee 
then in the sense of distributions:
 \be
 \label{eqagheingeio}
 L_a\phi=f+D_f(z)h.
 \ee
\end{lemma}

\begin{proof}[Proof of Lemma \ref{vneokvneneonvie}] This follows from a Lax-Milgram type argument.\\

\noindent{\bf step 1} Coercivity on $V$. 
 Let
\begin{equation}
\label{defnintin}
\nonumber V_0(r,z)=\left(\mathcal H_0''-\frac{(\C'_0)^2+\C_0\C_0''}{r^2}\right)(\psi_\nu(r,z))= \frac 1{\nu^4(z)}V_*(R),
\end{equation}
and define $$\phi(r,z)=\phit(R,z), \ \ R=\frac{r}{\nu},$$ so that
$$M\phi(r,z)=\frac{1}{\nu^4}M_*\phit(R,z).$$ Viewing $z$ as a parameter we compute
\bee
&&\int_{r>0}(M\phi)\phi rdr=\int_{R>0}\frac{\nu^2}{\nu^4}(M_*\phit(R,z))\phit(R,z)RdR\\
&\ge&\frac{1}{\nu^2}\left\{  c_* \int_{R>0}\left(\frac{|\pa_R\phit|^2}{R^2}+\frac{\phit^2}{R^4}\right)RdR-\frac1{c_* }\left(\int_{R>0} \phit \xi_*(R)RdR\right)^2\right\}\\
& = & c_* \int_{r>0}\left(\frac{|\pa_r\phi|^2}{r^2}+\frac{\phi^2}{r^4}\right)rdr-\frac1{c_* \nu^2}\left(\nu^{p-2}(z)\underbrace{\int_{r>0} \phi(r,z) h(r,z)rdr}_{=0 \text{ in }V}\right)^2.
\eee
This yields the coercivity property: $\forall \phi\in V,$
\bea\label{cnekovneonveonve}
\int_{r>0}\int_{z\in \Bbb R} (L_a\phi)\phi rdrdz
\nonumber &\ge&   c_* \int_{z\in \Bbb R}\int_{r>0}\left(\frac{|\pa_r\phi|^2}{r^2}+\frac{\phi^2}{r^4}\right)rdrdz\\
&+&a^2\int_{r>0}\int_{z\in \Bbb R}\frac{|\pa_z\phi|^2}{r^2}rdrdz.
\eea
We conclude from \eqref{cnekovneonveonve} that the scalar product $$\la \phi,\phit\ra_V=\int_{r>0}\int_{z\in \Bbb R}(L_a\phi)\phit rdrdz$$ induces on $V$ a norm equivalent to $\|\cdot\|_{H_a}$. Moreover from \eqref{assumptitoinf}, 
\bee
& & \left|\int_{r>0}\int_{z\in \Bbb R}f(r,z)\phi(r,z)rdrdz\right|\\
&\lesssim &\left(\int_{r>0}\int_{z\in \Bbb R}r^4f^2rdrdz\right)^{\frac 12}\left(\int_{z>0}\int_{r>0}\frac{\phi^2(r,z)}{r^4}rdrdz\right)^{\frac 12}\\
&\leq &  \|r^2f\|_{L^2(rdrdz)}\|\psi\|_{H_a},
\eee
where we used the Hardy inequality \eqref{Hardyenienige} in the last step, and \eqref{estaimteV} follows from Riesz's representation Theorem in the Hilbert space $V$. Applying this with $v=\phi$ yields \eqref{estiaotinetpsibis}.\\
 
 \noindent{\bf step 2} Equation satisfied by $\phi$. From the decay of $\zeta_*$ we have
 $$\int_{r>0}h^2(r,z)rdr=\frac{\nu^2(z)}{\nu^{2p}(z)}\int_{R>0}\xi_*^2(R)RdR=c(z)\in (0,+\infty).$$ Moreover $h\in H_a$ and hence given $\phi \in H_a$, we conclude that $$ \left|\begin{array}{l}
v=\phi-\ell_\phi(z)h\in V\\
 \ell_\phi(z)=\frac{\int_{r>0}\phi(r,z)h(r,z)rdr}{\int_{r>0}h^2(r,z)r dr}
 \end{array}\right.,
 $$
 Hence for a solution $\phi$ of \eqref{estaimteV} we compute
 \[\int_{r>0}\int_{z\in \Bbb R}(L_a\phi)v rdrdz=\int_{r>0}\int_{z\in \Bbb R} f v rdrdz,\]
 thus
 \[
 \int_{r>0}\int_{z\in \Bbb R}(L_a\phi)\left[\phi-\ell_\phi h\right] rdrdz=\int_{r>0}\int_{z\in \Bbb R} f \left[\phi-\ell_\phi h\right] rdrdz,\]
 which gives \[\int_{r>0}\int_{z\in \Bbb R}(L_a\phi)\phi rdrdz=\int_{r>0}\int_{z\in \Bbb R} f \phi rdrdz+\int_{r>0}\int_{z\in \Bbb R} \ell_\phi(z)\left[\phi(L_ah)-fh\right]rdrdz .
 \]
 Let $$C_f(z)=\int_{r>0}\left[\phi(L_ah)-fh\right]rdr,$$ then we rewrite
\bee
&& \int_{r>0}\int_{z\in \Bbb R} \ell_\phi(z)\left[\phi(L_ah)-fh\right]rdrdz=\int_{z\in \Bbb R} \ell_\phi(z)C_f(z)dz\\
& =& \int_{z>0}\frac{\int_{r>0}\phi(r,z)h(r,z)rdr}{\int_{\rho>0}h^2(\rho,z)\rho d\rho}C_f(z)dz=  \int_{r>0}\int_{z>0}\phi(r,z)\left[\frac{h(r,z)C_f(z)}{\int_{\rho>0}h^2(\rho,z)\rho d\rho}\right]rdrdz\\
& = & \int_{r>0}\int_{z>0}\phi(r,z)D_f(z)h(r,z)rdrdz,
 \eee
 with $$D_f(z)=\frac{\int_{r>0}\left[\phi(L_ah)-fh\right]rdr}{\int_{r>0}h^2(r,z)r dr}.$$ There remains to compute $$L_ah=L_a\left(\frac{\zeta_*}{\nu^p}\right)=-\frac{a^2}{r^2}\pa_z^2h,$$ and \eqref{eqagheingeio} is proved.
 \end{proof}
 

\section{Pointwise decay for the resolvent}
\label{sectionresolventtwo}


This section is devoted to the derivation of pointwise decay estimates for the resolvent of Lemma \ref{vneokvneneonvie} which are essential to close the full non linear problem. We will use the growth \eqref{funoenioenone} which after conjugation allows us to work with the high dimensional Laplace operator.  A function $\Psi(r,z)$ will be canonically viewed as a cylindrical function in $\Bbb R^{D+1}$ through the decomposition $$\left|\begin{array}{l}
x=x_\perp+z\ve_z\\
x_\perp\in \Bbb R^D\\
 |x_\perp|=r.
 \end{array}\right.
 $$
 

\subsection{Pointwise decay}


Pick a small loss $0<\delta_*\ll 1$. We introduce the norm
\be
\label{vneionveoiveovei}
\|f\|_{\alpha}=\sup_{r>0}\left[\la r\ra^\alpha\left(\int_{z>0}\la z\ra^{3-\delta_*} f^2(r,z)dz\right)^{\frac12}\right]+\|\la r\ra^\alpha f\|_{L^\infty}.
\ee

Let  $E_{\alpha}$ be the Banach obtained by the closure of $\matchal D_{\rm cyl}(\Bbb R^2)$ for the norm $$\|\Psi\|_{E_{\alpha}}=\|\la r\ra^\alpha\la z\ra^{\frac 32-\delta_*}\pa_r\Psi\|_{L^2(\Bbb R^5)}+a\|\la r\ra^\alpha\la z\ra^{\frac 32-\delta_*}\pa_z\Psi\|_{L^2(\Bbb R^5)}+a^{\frac 12}\left\|\la z\ra^{\frac{1}{4}}\frac{\la r\ra^{\alpha+2}}{r^2}\psi\right\|_{L^\infty}.$$  Let $$\left|\begin{array}{l}
L_a\psi=f+D_f h\\
\psi=r^2\Psi, \ \ \Psi\equiv \mathcal L_a^{-1}f
\end{array}\right.,
$$ 
be the resolvent map constructed in Lemma \ref{vneokvneneonvie} which solves \eqref{eqagheingeio}. 
\begin{proposition}[Pointwise decay]
\label{proppointizedfjifw}
Consider
$$\left(c_0,(C_i)_{0\le i\le 2}\right)\in (\Bbb R_*^+)^4.$$
Then there exists $m^*\in \mathbb{N}$, such that for $\gamma=2+\frac 1m$, $m\in \Bbb N, m\geq m^*$ there exists a universal constant $\alpha(\gamma)$ such that for $1\ll \alpha\leq \alpha(\gamma)$ there exists a universal constant $K=K(c_0,C_0,\gamma)$ such that for all $0<\eta<\eta^*(\gamma,\delta_*)$ small enough and for all $\nu\in \mathcal N^\eta(c_0,(C_i)_{0\le i\le 2})$ and all $0<a<a^*((c_0,(C_i)_{0\le i\le 2}))$ small enough we have
\be
\label{continieou}
\|\mathcal L_a^{-1}f\|_{E_{\alpha-6}}\leq K\|f\|_{\alpha}.
\ee
\end{proposition}
The rest of this section is devoted to the proof of Proposition \ref{proppointizedfjifw}.


\subsection{Energy bound}


 We fix once and for all the source $f$ with $\|f\|_{\alpha}<+\infty$. Note that using a standard density argument  we may assume $f\in \mathcal S(\Bbb R^5)$, and then standard elliptic regularity estimates ensure $\Psi\in \mathcal C^\infty(\Bbb R^5)$.
 
 \begin{lemma}[Energy estimate in $\Psi$]
 \label{energyestimatePsi}
 There holds 
 \be
\label{enguybaopor}
\left( \int_{\Bbb R^5} \left[(\pa_r\Psi)^2+a^2(\pa_z\Psi)^2\right]dx\right)^{\frac 12}\lesssim \|f\|_{\alpha}.
\ee
\end{lemma}

\begin{proof}[Proof of Lemma \ref{energyestimatePsi}] Observe that
$$\int_{r>0,z\in \Bbb R}r^2f^2r^3drdz\lesssim \|f\|_{\alpha}^2\int_{r>0}\frac{r^3dr}{\la r\ra^{2\alpha}}\lesssim \|f\|_{\alpha}^2,
$$
for $\alpha$ large enough, so we may apply Lemma  \ref{vneokvneneonvie} and observe that $\psi=r^2\Psi$ satisfies from \eqref{eqagheingeio}
\be
 \label{vneoneonenvoi}
 \left|\begin{array}{l}
L_a\psi=\mathcal L_a\Psi\equiv \left(-\pa_r^2-\frac{3}{r}\pa_r-a^2\pa_z^2+W\right)\Psi=f+D_f(z)h\\
W(r,z)=r^2V_0(r,z)=\frac1{\nu^2}W_*(R)\\
W_*(R)=R^2V_*(R)
\end{array}\right.,
\ee
 Moreover 
 \bee
 &&\int_{r>0,z\in \Bbb R}\left[-\frac{1}{r}\pa_r\left(\frac{1}{r}\pa_r\psi\right)-\frac{a^2}{r^2}\pa_z^2\psi\right]\psi rdrdz\\
 & = & \int_{r>0,z\in \Bbb R}\left[-\pa_r^2\Psi-\frac{3}{r}\pa_r\Psi-a^2\pa_z^2\Psi\right]\Psi r^3drdz= \int_{\Bbb R^5} \left[(\pa_r\Psi)^2+a^2(\pa_z\Psi)^2\right]dx,
 \eee
and hence \eqref{estiaotinetpsibis}  ensures the  a priori bound \eqref{enguybaopor}.
\end{proof}

\subsection{$L^\infty$ bound}

We claim the following rough $L^\infty$ bound. 

\begin{lemma}[Pointwise bounds]
\label{lemmalinftybound}
Under the assumptions of Proposition \ref{proppointizedfjifw}, we have $\Psi\in \mathcal C_0(\Bbb R^5)$ with
\be
\label{vneovneovneoneonvonbi}
\|\Psi\|_{L^\infty}\lesssim \frac{\|f\|_{\alpha}}{a^{\frac 12}}.
\ee 
\end{lemma}

 \begin{proof}[Proof of Lemma \ref{lemmalinftybound}] Let $$\left|\begin{array}{l}
F(r,z)=f+D_fh-W\Psi=\tilde{F}(r,Z)\\
\Psi(r,z)=\Psit(r,Z)\\
W(r,z)=\tilde{W}(r,Z)
\end{array}\right.,
$$
then  \eqref{vneoneonenvoi} yields $$-\Delta_{\Bbb R^5}\Psit=\tilde{F},$$ with the a priori bound from \eqref{estiaotinetpsibis}
$$\left(\int_{\Bbb R^5}\left[(\pa_r\Psi)^2+a^2(\pa_z\Psi)^2\right]dx\right)^{\frac12}\lesssim \|rf\|_{L^2(\Bbb R^5)}\lesssim \|f\|_{\alpha}.$$ 
We compute
 \bea
 \label{estpeoneitneil}
\nonumber  &&\|\nabla \Psit\|^2_{L^2(\Bbb R^5)}=\int_{r>0,Z\in \Bbb R}\left( -\pa_r^3\Psit-\frac{3}r\pa_r\Psit-\pa_Z^2\Psit\right)\Psit(r,Z)r^3drdZ\\
\nonumber  &=& \frac{1}{a}\int_{r>0,Z\in \Bbb R}\left( -\pa_r^3\Psi-\frac{3}r\pa_r\Psi-a^2\pa_z^2\Psi\right)(r,z)\Psi(r,z)r^3drdz\\
\nonumber  & = &  \frac{1}{a}\int_{r>0,z\in \Bbb R}\left[(\pa_r\Psi)^2+a^2(\pa_z\Psi)^2\right]r^2drdz,
 \eea
 and hence $$\|\nabla \Psit\|_{L^2(\Bbb R^5)}\lesssim  \frac{\|f\|_{\alpha}}{a^{\frac 12}}.$$
 The potential term satisfies the bound
$$|W(r,z)|=\frac{1}{\nu^2}W_*(R)\lesssim \frac{1}{\nu^2\left(1+\frac{r^2}{\nu^2}\right)}\lesssim \frac{\nu^2}{\nu^2+r^2}\lesssim \frac{1}{\la r\ra^2},$$
and hence using the 4 dimensional Hardy inequality
$$\|r\tilde{W}\Psit\|_{L^2(\Bbb R^5)}\lesssim \left\|\frac{\Psit}{r}\right\|_{L^2(\Bbb R^5)}\lesssim \|\pa_r\Psit\|_{L^2(\Bbb R^5)}\lesssim \frac{\|f\|_{\alpha}}{a^{\frac 12}},$$ and hence 
\be
\label{venoenneoieonviv}
\|\tilde{W}\Psit\|_{L^{\frac{10}{3}}(\Bbb R^5)}+\|r\tilde{W}\Psit\|_{L^2(\Bbb R^5)}\lesssim \|\nabla \Psit\|_{L^{2}(\Bbb R^5)}\lesssim  \frac{\|f\|_{\alpha}}{a^{\frac 12}}.
 \ee
 We now observe 
 \bee
 \int_{r>0,z\in \Bbb R}|f|^{\frac{10}{3}}r^3drdz&\lesssim& \int_{r>0,z\in \Bbb R}\left(\frac{\|f\|_{\alpha}}{\la r\ra^\alpha}\right)^{\frac{4}{3}}|f(r,z)|^2r^3drdz\\
    &\lesssim& \|f\|_{\alpha}^{\frac{10}{3}}\int_{r>0}\frac{r^3dr}{\la r\ra^{\frac{10\alpha}{3}}}\lesssim \|f\|_{\alpha}^{\frac{10}{3}},
  \eee
 and rescaling in $z$ yields
 $$\|\tilde{f}\|_{L^{\frac{10}{3}}(\Bbb R^5)}=\frac{\|f\|_{L^{\frac{10}{3}}(\Bbb R^5)}}{a^{\frac 3{10}}}\lesssim \frac{\|f\|_{\alpha}}{a^{\frac 12}},$$
 We now claim the bound 
 \be
 \label{fmeognneineo}
 \|r\tilde{D_fh}\|_{L^2(\Bbb R^5)}+\|\tilde{D_fh}\|_{L^{\frac{10}{3}}(\Bbb R^5)}\lesssim \frac{\|r^2f\|_{L^2(\Bbb R^2)}+\|f\|_{L^{\frac{10}{3}}(\Bbb R^5)}}{a^{\frac 12}}.
 \ee
Assume \eqref{fmeognneineo}, then the Sobolev embedding $\dot{H}^1(\Bbb R^5)\subset L^{\frac{10}{3}}(\Bbb R^5)$ implies:
 $$\|\Psit\|_{L^{\frac{10}{3}}}+\|\Delta \Psit\|_{L^{\frac{10}{3}}}\lesssim \frac{\|f\|_{\alpha}}{a^{\frac 12}}.$$ Since $2>\frac{5}{\frac{10}{3}}$, the Sobolev embedding $W^{2,\frac{10}{3}}(\Bbb R^5)\subset \matchal (C_0(\Bbb R^5),\|\cdot\|_{L^\infty})$ yields $$\|\Psi\|_{L^\infty}=\|\Psit\|_{L^\infty}\lesssim \frac{\|rf\|_{L^2(\Bbb R^5)}+\|f\|_{L^{\frac{10}{3}}(\Bbb R^5)}}{a^{\frac 12}},$$ and \eqref{vneovneovneoneonvonbi} is proved.\\

 \noindent{\em Proof of \eqref{fmeognneineo}}. We recall \eqref{defdfz} and split
 \be
 \label{veniovneoneoiv}
 \left|\begin{array}{l}
D_f(z)h=D_1-D_2\\
D_1=\frac{\int_{r>0}\left[\psi(L_ah)\right]rdr}{\int_{r>0}h^2(r,z)r dr}h\\
D_2=\frac{\int_{r>0}\left[fh\right]rdr}{\int_{r>0}h^2(r,z)r dr}h
\end{array}\right..
\ee
\noindent\underline{Bound for $D_2$}. Recall from \eqref{defhshth} that 
$h(r,z)=\frac{1}{\nu^p(z)}\zeta_*(R).$ We compute
\be
\label{veniovnenveinvoe}
\left|\begin{array}{l}
\int_{r>0}h^2(r,z)r dr=\frac{\nu^2}{\nu^{2p}}\int_{R>0}\zeta_*^2(R)RdR=\frac{c_*}{\nu^{2p-2}}\\
\int_{r>0}\frac{h^2(r,z)}{r^3} dr=\frac{1}{\nu^{2p+2}}\int_{R>0}\frac{(\zeta_*)^2}{R^3}dR=\frac{c_*}{\nu^{2p+2}}
\end{array}\right.,
\ee 
and estimate
$$\left|\int_{r>0}\left[fh\right]rdr\right|\lesssim \left(\int_{r>0}f^2r^4 rdr\right)^{\frac12}\left(\int_{r>0}\frac{h^2rdr}{r^4}\right)^{\frac 12}\lesssim \frac{1}{\nu^{p+1}}\left(\int_{r>0}r^2f^2r^3dr\right)^{\frac12},$$ which yields
\bea
\label{vdjkvbndkndlknvdlnv}
|D_2(r,z)|&\lesssim& \frac{1}{\nu^p\la R\ra^{\gamma-2+|\l_-|}}\frac{\nu^{2p-2}}{\nu^{p+1}}\left(\int_{r>0}r^2f^2r^3dr\right)^{\frac12}\nonumber \\
&\lesssim& \frac{1}{\nu^3\la R\ra^{\gamma-2+|\l_-|}}\left(\int_{r>0}f^2r^2r^3dr\right)^{\frac12},
\eea
and thus
\bee
&&\int_{r,z}|D_2(r,z)|^{\frac{10}{3}}r^3drdz\lesssim \int_{r,z}\frac{r^3drdz}{\nu^{10}\la R\ra^{10|\l_-|}}\left(\int_{u>0}f^2(u,z)u^2u^3du\right)^{\frac{5}{3}}\\
&\lesssim & \|f\|^{\frac 43}_{\alpha}\int_{z\in \Bbb R}dz\left(\int_{u>0}u^2f^2(u,z)u^3du\right)\left(\int_{u>0}\frac{u^5du}{\la u\ra^{2\alpha}}\right)^{\frac 23}\left(\int_{r\le \nu}\frac{r^3dr}{\nu^{10}\la R\ra^{10}}\right)\\
& \lesssim & \|f\|_{\alpha}^{\frac{10}{3}}\int_{u>0}\frac{u^5du}{\la u\ra^{2\alpha}}\lesssim \|f\|_{\alpha}^{\frac{10}{3}},
\eee
 and similarly
\bee
&&\int_{r,z}|D_2(r,z)|^{2}r^2r^3drdz\lesssim \int_{r,z}\frac{r^2r^3drdz}{\nu^6\la R\ra^{2|\l_-|}}\int_{u>0}u^2f^2(u,z)u^3du\\
& \lesssim & \int_{u>0,z\in \Bbb R}f^2(u,z)u^5dudz\int_{R>0}\frac{R^5dR}{\la R\ra^{2|\l_-|}}\lesssim \|f\|_{\alpha}^2\int_{u>0}\frac{du}{\la u\ra^{2\alpha-5}}\lesssim \|f\|_{\alpha}^2.
\eee 
Hence re-scaling in $z$
$$\|r\tilde{D}_2\|_{L^2(\Bbb R^5)}+\|\tilde{D_2}\|_{L^{\frac {10}{3}}(\Bbb R^5)}\lesssim \frac{\|f\|_{\alpha}}{a^{\frac 12}}+\frac{\|f\|_{\alpha}}{a^{\frac3{10}}}\lesssim \frac{\|f\|_{\alpha}}{a^{\frac 12}}.$$
\noindent\underline{Bound for $D_1$}. We estimate using \eqref{estoatmgnoi}
$$|L_ah|\lesssim \frac{a^2}{\nu^{p+4}\la R\ra^{|\l_-|}},$$
 and hence
\be
\label{vnovnnvenevbvbbvk}
|D_1|\lesssim \frac{\zeta_*}{\nu^p}\frac{\int_{r'>0}|\Psi|\frac{(r')^3dr'}{\nu^{p+4}\la R'\ra^{|\l_-|}}}{\frac{1}{\nu^{2p-2}}}\lesssim \frac{1}{\nu^6\la R\ra^{|\l_-|}}\int_{r'>0}\frac{|\Psi|}{\la R'\ra^{|\l_-|}}(r')^3dr'.
\ee
This implies
\bee
|D_1|&\lesssim &  \frac{1}{\nu^6\la R\ra^{|\l_-|}}\left(\int_{r>0}\Psi^{\frac{10}{3}}r^3dr\right)^{\frac{3}{10}}\left(\int_{r>0}\frac{r^3}{\la R\ra^{|\l_-|}}dr\right)^{\frac{7}{10}} \\
&\lesssim & \frac{1}{\la R\ra^{|\l_-|}\nu^{\frac{16}5}}\left(\int_{r>0}\Psi^{\frac{10}{3}}r^3dr\right)^{\frac{3}{10}},
 \eee
 and hence 
 \bee
 \int_{Z>0,z\in \Bbb R}|\tilde{D_1}|^{\frac{10}{3}}r^3drdZ&\lesssim& \int_{Z\in \Bbb R,r'>0}\Psit^{\frac{10}{3}}(r',Z)(r')^3dr'dZ\left(\int_{r\le \nu}\frac{r^3dr}{\nu^{\frac{32}{3}}\la R\ra^{|\l_-|}}\right)\\
 &\lesssim& \|\Psit\|_{L^{\frac{10}{3}}(\Bbb R^5)}^{\frac{10}{3}},\eee
 thus \[\|\tilde{D_1}\|_{L^{\frac{10}{3}}(\Bbb R^5)}\lesssim \|\nabla \Psit\|_{L^2}\lesssim \frac{\|f\|_{\alpha}}{a^{\frac 12}}.\]
 Similarly
 \be
 \label{estiaofnrdiune}
 r|D_1|\lesssim \frac{\nu}{\nu^6\la R\ra^{|\l_-|}}\left(\int \frac{\Psi^2}{r^2}r^3dr\right)^{\frac 12}\left(\int_{r>0}\frac{r^2r^3dr}{\la R\ra^{|\l_-|}}\right)^{\frac 12}\lesssim \frac{1}{\nu^2\la R\ra^{|\l_-|}}\left(\int \frac{\Psi^2}{r^2}r^3dr\right)^{\frac 12},
 \ee 
 and hence
\bee
\|rD_1\|_{L^2(\Bbb R^5)}^2&\lesssim& \int_{z>0}\left(\int_{r'>0} \frac{\Psi^2(r',z)}{(r')^2}(r')^3dr\right)\int_{r>0}\frac{r^3dr}{\nu^4\la R\ra^{|\l_-|}}dz\lesssim \left\|\frac{\Psi}{r}\right\|^2_{L^2(\Bbb R^5)}\\
&\lesssim & \|\pa_r\Psi\|^2_{L^2(\Bbb R^5)}\lesssim \|f\|^2_{\alpha}.
\eee 
which after rescaling in $z$ concludes the proof of \eqref{fmeognneineo}.
 \end{proof}

 \subsection{From the $L^2$ bound to pointwise decay}
We prove pointwise decay in $z$ in the {\em non compact zone} $R\lesssim 1$ assuming weighted $L^2$ decay. 
\begin{lemma}[$L^2$ decay implies pointwise decay in $z$ for $R\lesssim 1$]
\label{leammapotinsno}
Under the assumptions of Proposition \ref{proppointizedfjifw}, let $0\le \theta \le 1,$ and $$J_F(x)=\int_{\Bbb R^5}\frac{F(x')}{a|X-X'|^3}dx'.$$ Then for all $R_0\ge 1$, there exists $C_{R_0}$ such that 
\bea
\label{estiinidoioeiog}
&&|J_F(x)|{\bf 1}_{R\le R_0,|z|\ge 1}\\
\nonumber &\leq & C_{R_0}\left[\frac{\|\la z\ra^{\frac{\theta}{2}} F\|_{L^2(\Bbb R^5)}}{a^{\frac 12}\la z\ra^{\frac{\theta}{2}}}+ \int_{|x_\perp-x'_\perp|\le 1, \frac{|z|}{2}\le |z|\le\frac{3|z|}{2}}\frac{|F(x')|}{a|X-X'|^3}dx'\right].
\eea
\end{lemma}

\begin{proof} We split the convolution in suitable zones. Observe that $|aX|\ge |z|\ge 1.$\\

\noindent\underline{$|X'|\ge 2|X|$}. Then we have
\bee
|J_F^{(1)}(x)|\lesssim \frac1{\sqrt{a}}\left(\int_{\Bbb R^5}\la z'\ra^\theta F^2dx'\right)^{\frac 12}\left(\int_{|X'|\ge 2|X|}\frac{dX'}{|X'|^6\la z'\ra^{\theta}}\right)^{\frac 12}\lesssim\frac{ \|\la z\ra^{\frac{\theta}{2}}F\|_{L^2(\Bbb R^5)}}{|aX|^{\frac 12}}.
\eee

\noindent\underline{$|X-X'|\le\frac{|X|}{4}$, $|X'|\le 2|X|$}. Then we have
\bee
|J_F^{(2)}(x)|&\lesssim& \frac1{\sqrt{a}|X|^3}\left(\int_{\Bbb R^5}\la z'\ra^\theta F^2dx'\right)^{\frac 12}\left(\int_{|X'|\le 2|X|}\frac{dX'}{\la z'\ra^{\theta}}\right)^{\frac 12}\lesssim\frac{ \|\la z\ra^{\frac{\theta}{2}} F\|_{L^2(\Bbb R^5)}}{|aX|^{\frac 12}}.
\eee

\noindent\underline{$|X-X'|\le\frac{|X|}{4}$}. If $|Z|\le r$, then $|z|\le ar\le aR_0\nu(z)\le aR_0C_\gamma\la z\ra^\eta$ implies $|z|\ll1$, and hence this case is absent. Hence $|Z|\ge r$ implies  $Z\le |X|=r+Z\le 2Z$ thus $$|Z-Z'|\le\frac{|X|}{4}\le \frac{Z}{2}\Rightarrow \frac{Z}{2}\le |Z'|\le \frac{3Z}{2}\Rightarrow \frac{z}{2}\le |z'|\leq \frac{3z}2.$$ 
If $|x_\perp-x'_\perp|\ge 1$, we estimate:
\bee
&& |J_F^{(3)}(x)|\\ &\lesssim &\frac{1}{a\la z\ra^\frac{\theta}{2}}\left(\int_{\Bbb R^5}\la z'\ra^\theta F^2dx'\right)^{\frac 12}\left(\int_{|x_\perp-x'_\perp|\ge 1}dx'_\perp\int_{Z'\in \Bbb R}\frac{adZ'}{(|x_\perp-x_\perp'|^2+(Z-Z')^2)^{3}}\right)^{\frac 12},\\
\eee
thus
\[ |J_F^{(3)}(x)|\lesssim  \frac{1}{\sqrt{a}\la z\ra^\frac{\theta}{2}}\left(\int_{|x_\perp-x'_\perp|\ge 1}\frac{dx'_\perp}{|x_\perp-x'_\perp|^5}\right)^{\frac 12}\lesssim \frac{1}{\sqrt{a}\la z\ra^\frac{\theta}{2}}.
\]
The collection of the above bounds yields \eqref{estiinidoioeiog}
\end{proof}
We now show decay in $R$ in higher dimension and in the far away zone $R\gtrsim R_*$. This higher dimensional decay will come later on by trading the $\frac{K}{r^2}$ decay of the potential $W$ with $K\gg 1$ on $\mathbb{R}^3$ to get a higher dimensional Laplacian.
\begin{lemma}[$L^2$ bound to improved decay]
\label{neoieonvndklnvdv}
Under the assumptions of Proposition \ref{proppointizedfjifw}, consider 
\be
\label{asutmoptno}
\left|\begin{array}{l}
D\in \Bbb N_*, \ \ \eta<\frac{1}{D}  \\
\mu=\frac{D-4}{2} , \ \ \alpha\leq \mu \\
0\le \theta\le 1\\
R_*>1
\end{array}\right.
\ee
and define $$ I_G(x)=\int_{x'\in \Bbb R^{D+1}}\frac{|G(r',z')|}{\la r' \ra^\mu}\frac{dx'}{a|X-X'|^{D-1}},
$$ 
then the following estimates hold.\\

\noindent\underline{Case compact support}: if ${\rm Supp }G\subset \{R\le R_*\}$, then there exist universal constants $K_\nu\gg 1$ and $C_*=C_*(\mu,\nu,\theta,R_*)$ such that 
\be
\label{fondnanteta}
\la r\ra ^\mu I_G(x){\bf 1}_{R\ge K_\nu R_*}\leq \frac{C_*\|\la z\ra^{\frac{\theta}{2}} G\|_{L^2(\Bbb R^5)}}{a^{\frac 12}}\left[\frac{1}{|X|^{\frac{D}2}}+\frac{1}{\la z\ra^{\frac{\theta}{2}}\la r\ra^{\frac 12}R^{\frac{D}{2}}}\right].
\ee
\noindent\underline{Well localized source}: more generally, for $r\ge 1$ we have
\be
\label{fondnantetabis}
\la r\ra ^\mu I_G(x)\leq C_*\|G\|_{\alpha}\left[\frac{\la r\ra^\mu}{|X|^{\alpha+\mu-2}}+\frac{1}{\la z\ra^{\frac 12}\la r\ra^{\alpha-\frac 32}}\right].
\ee
\end{lemma}
\begin{proof}[Proof of Lemma \ref{neoieonvndklnvdv}]  We will systematically use the weighted Cauchy Schwarz:
\bee
&&\int_{r'>0}|G(r')F(r')|(r')^{D-1}dr'\\
&\lesssim & \left(\int G^2\frac{(r')^{D-1}}{(r')^{D-4} }dr'\right)^{\frac 12}\left(\int_{r'>0}F^2(r')^{D-4}(r')^{D-1}dr'\right)^{\frac 12}\\
&\lesssim & \|G\|_{L^2(r^3dr)}\left(\int_{r'>0}(r')^{D-4}F^2(r')^{D-1}dr'\right)^{\frac 12}.
\eee

\noindent{\bf step 1} Case ${\rm Supp G}\subset\{R\le R_*\}$.
We split cases. We recall that $R\gtrsim R_*$ implies $|X|\ge r\ge c_0 R_*\gtrsim 1$.
  \noindent\underline {$|X'|\ge 2|X|$}. In this region we have $$r'+|Z'|\ge 2(r+|Z|)\Rightarrow ar'+|z'|\ge 2|aX|\ge 2ac_0R^*,$$ and $R'\le R_*\Rightarrow r'\le R_* \nu(z')\le C_0 R_* \la z'\ra^{\eta}$ implies $$\left|\begin{array}{l}
  aC_0 R_* \la z'\ra^\eta+|z'|\ge 2 |aX|\Rightarrow |z'|\gtrsim  |aX|\\
  |aX-aX'|\gtrsim |aX'|\gtrsim \la z'\ra
  \end{array}\right..
  $$
  We therefore estimate the corresponding contribution
  \bee
  |I^{(1)}_G(x)|&\lesssim & \int_{|X'|\ge 2|X|}\frac{|G|{\bf 1}_{R'\le R_*}}{a|X-X'|^{D-1}}\frac{dx'}{\la r'\ra^\mu}\\
 &= &\frac{a^{D-1}}{a}\int_{|X'|\ge 2|X|}\frac{G{\bf 1}_{R'\le R_*}}{\la r'\ra^\mu|aX-aX'|^{D-1}}dx'\\
  & \lesssim &   a^{D-2}\left(\int_{\Bbb R^5}G^2dx\right)^{\frac 12}\left(\int_{|z'|\gtrsim |aX|}\left[\int_{r'\lesssim \la z'\ra^{\eta}}\frac{(r')^{D-1}(r')^{D-4}}{\la r'\ra^{2\mu}}dr'\right]\frac{dz'}{|z'|^{2(D-1)}}\right)^{\frac 12}\\
  &\lesssim & \frac{a^{D-2}\|G\|_{L^2(\Bbb R^5)}}{(a|X|)^{D-\frac 32-\frac{D}{2}\eta}}\lesssim \frac{\|G\|_{L^2(\Bbb R^5)}}{a^{\frac 12}|X|^{D-\frac 32-\frac{D}{2}\eta}},
 \eee
   \noindent\underline {$|X'|\le 2|X|$, $|X-X'|\ge \frac{|X|}{4}$}. In this region we estimate
   \bee
 |I^{(2)}_G(x)|&\lesssim & \frac{1}{a|X|^{D-1}}\int_{|X'|\le 2|X|}|G|{\bf 1}_{R\le R_*}\frac{dx}{\la r'\ra^\mu}\\
 & \lesssim & \frac{1}{a|X|^{D-1}}\left(\int_{\Bbb R^5}G^2dx\right)^{\frac 12}\left(\int_{|X'|\le 2|X|,r'\le R_*\nu'}\frac{a(r')^{D-4}dX'}{\la r'\ra^{2\mu}}\right)^{\frac 12}\\
 & \lesssim  & \frac{\|G\|_{L^2(\Bbb R^5)}}{a^{\frac 12}|X|^{D-1}}\left(\int_{|Z'|\le 2|X|}dZ'\int_{r'\le R_*\nu'}\frac{(r')^{D-1}(r')^{D-4}}{(r')^{2\mu}}dr'\right)^{\frac 12}\\
 &\lesssim & \frac{\|G\|_{L^2(\Bbb R^5)}}{a^{\frac 12}|X|^{D-\frac 32-\frac{D}{2}\eta}}.
 \eee
\noindent\underline {$|X'-X|\le \frac{|X|}{4}$}.  We distinguish cases.\\
\noindent{\em The Case of $Z\le r$}. Then $r\le |X|=r+|Z|\le 2r$  and hence $r\gtrsim |X|\gtrsim 1$ and $$
\left|\begin{array}{l}
|x_\perp-x_\perp'|\le\frac{|X|}{4}\le \frac{r}{2}\\
|Z'|\le |Z|+\frac{|X|}{2}\leq 4 r
\end{array}\right.
\Rightarrow \frac{r}{2}\le |x'_\perp|\le \frac{3r}{2},
$$ 
thus $$Z'=\frac{z'}{a}\lesssim r'\leq R_*\nu'\leq R_* C_0\la z'\ra^\eta\Rightarrow |z'|+r'+r+|z|\leq 1,$$ for $|a|<a_*(C_0,R_*)$ small enough, which implies $$R=\frac{r}{\nu}\le C_0\le K_\nu R_*,$$ for $K_\nu$ chosen large enough, and hence this case is absent.\\

\noindent{\em The case of $Z\ge r$}.
 Then $Z\le |X|=r+Z\le 2Z$ and hence $$|Z-Z'|\le\frac{|X|}{4}\le \frac{Z}{2}\Rightarrow \frac{Z}{2}\le |Z'|\le \frac{3Z}{2}\Rightarrow \frac{z}{2}\le |z'|\leq \frac{3z}{2}\Rightarrow c\nu'\le \nu\le C\nu',$$ where we used $$\frac{|\pa_z\nu|}{\nu}\leq \frac{C_1}{|z|},$$
and thus
\bee
& & |I^{(3)}_G(x)|\\
&\lesssim & \frac{\left(\int_{\Bbb R^5}\la z\ra^{\theta}G_2^2dx\right)^{\frac 12}}{a\la z\ra^{\frac{\theta}{2}}}\left(\int_{x'_\perp} \frac{(r')^{D-4}dx'_\perp}{\la r'\ra^{2\mu}}\int_{Z'\in \Bbb R}\frac{adZ'}{\left(|x_\perp'-x_\perp|^2+(Z-Z')^2\right)^{D-1}}\right)^{\frac 12}\\
&\lesssim &  \frac{\|\la z\ra^{\frac{\theta}{2}} G\|_{L^2(\Bbb R^5)}}{a^{\frac 12}\la z\ra^{\frac{\theta}{2}}}\left(\int_{r'\lesssim \nu}\frac{(r')^{D-4}dx'_\perp}{\la r'\ra^{2\mu}|x_\perp'-x_\perp|^{2D-3}}\right)^{\frac 12}.
\eee
We now observe that $$|x_\perp-x_\perp'|\le \frac{r}{2}\Rightarrow  \frac{r}{2}\le r'\le \frac{3r}{2}\Rightarrow R=\frac{r}{\nu}\leq \frac{2r'}{\nu}\leq\frac{2r'}{c'\nu'}<C_\gamma R_*$$  hence contradicting $R\ge K_\nu R_*,$ provided $K_\nu$ has been chosen large enough. We therefore estimate using $2\mu=D-4$
\bee
|I^{(3)}_G(x)|&\lesssim & \frac{\|\la z\ra^{\frac{\theta}{2}} G\|_{L^2(\Bbb R^5)}}{a^{\frac 12}\la z\ra^{\frac{\theta}{2}}}\left(\frac1{r^{2D-3}}\int_{r'\lesssim \nu}(r')^{D-1}dr'\right)^{\frac 12}\\
&\lesssim & \frac{\|\la z\ra^{\frac{\theta}{2}} G\|_{L^2(\Bbb R^5)}}{a^{\frac 12}\la z\ra^{\frac{\theta}{2}}}\left(\frac{\nu^{D}}{\la r\ra^{2D-3}}\right)^{\frac 12}\lesssim\frac{\|\la z\ra^{\frac{\theta}{2}} G\|_{L^2(\Bbb R^5)}}{a^{\frac 12}\la z\ra^{\frac{\theta}{2}}\la r\ra^\mu}\left(\frac{\nu^{D}\la r\ra^{D-4}}{\la r\ra^{2D-3}}\right)^{\frac 12}\\
&\lesssim & \frac{\|\la z\ra^{\frac{\theta}{2}} G\|_{L^2(\Bbb R^5)}}{a^{\frac 12}\la z\ra^{\frac{\theta}{2}}\la r\ra^\mu}\frac{1}{\la r\ra^{\frac 12}R^{\frac{D}{2}}}.
\eee
The collection of above bounds yields using $|X|=r+\frac{|z|}{a}\ge \la r\ra$:
\bee
&&\la r\ra^\mu\frac{|I_G(x)|}{\|\la z\ra^{\frac{\theta}{2}} G\|_{L^2(\Bbb R^5)}}\lesssim \frac{\la r\ra^\mu}{a^{\frac 12}}\left[\frac{1}{|X|^{D-\frac 32-\frac{D}{2}\eta}}+\frac{1}{{\la z\ra^{\frac{\theta}{2}}}\la r\ra^{\mu+\frac 12}R^{\frac{D}{2}}}\right]\\
&\lesssim& \frac{1}{a^{\frac 12}}\left[\frac{1}{|X|^{D-\frac 32-\frac{D}{2}\eta-\frac{D-4}{2}}}+\frac{1}{{\la z\ra^{\frac{\theta}{2}}}\la r\ra^{\frac 12}R^{\frac{D}{2}}}\right] \lesssim  \frac{1}{a^{\frac 12}}\left[\frac{1}{|X|^{\frac{D}2}}+\frac{1}{\la z\ra^{\frac{\theta}{2}}\la r\ra^{\frac 12}R^{\frac{D}{2}}}\right],
\eee
for $\eta<\frac{1}{D}$, and \eqref{fondnanteta} is proved.\\

\noindent{\bf step 2} Estimate for a well localized source.\\

\noindent\underline{The case $|X'|\ge 2|X|$.} If $r'\ge |Z'|$, then $2r'\ge r'+|Z'|=|X'|\ge r'$ and hence  $|X-X'|\gtrsim |X'|\sim r'$ so that we have
\bee
&&|I^{(1)}_G(x)|\lesssim\|G\|_{\alpha} \int_{|X'|\ge 2|X|}\frac{r'^{D-1}dr'dz'}{a|X-X'|^{D-1}\la r'\ra^{\mu+\alpha}}\\
&\lesssim &\|G\|_{\alpha}\int_{r'\gtrsim |X|} \left(\int_{|Z'|\le r'}dZ'\right)\frac{dr'}{(r')^{\mu+\alpha}} \lesssim  \|G\|_{\alpha} \int_{r'\gtrsim |X|}\frac{dr'}{(r')^{\mu+\alpha-1}}\lesssim \frac{\|G\|_{\alpha}}{|X|^{\mu+\alpha-2}}.
\eee
If $|Z'|\ge r'$, then $2|Z'|\ge |X'|\ge |Z'|$ and hence  $|X-X'|\gtrsim |X'|\sim |Z'|$ and since $\alpha\leq \mu$ (ensuring the integral in $r'$ diverges) we estimate
\bee
|I^{(1)}_G(x)|&\lesssim&\|G\|_{\alpha} \int_{|Z'|\gtrsim |X|}\int_{r'\le |Z'|}\frac{r'^{D-1}dr'dz'}{a|Z'|^{D-1}\la r'\ra^{\mu+\alpha}}\\
& \lesssim & \|G\|_{\alpha} \int_{|Z'|\gtrsim |X|}\frac{dZ'}{|Z'|^{\mu+\alpha-1}}\lesssim  \frac{\|G\|_{\alpha}}{|X|^{\mu+\alpha-2}}.
\eee
\noindent\underline{The case $|X'|\le 2|X|$, $|X-X'|\ge \frac{|X|}{4}$}. If $r'\ge |Z'|$, then since $r'\le r'+|Z'|=|X'|\le 2|X|$ we have
\bee
|I^{(2)}_G(x)|&\lesssim&\frac{\|G\|_{\alpha}}{|X|^{D-1}} \int_{r'\le 2|X|}\int_{|Z'|\le r'}\frac{r'^{D-1}dr'dz'}{a\la r'\ra^{\mu+\alpha}}\\
& \lesssim & \frac{\|G\|_{\alpha}}{|X|^{D-1}} \int_{r'\le 2|X|}(r')^{D-\mu-\alpha}dr'\lesssim \frac{\|G\|_{\alpha}}{|X|^{\mu+\alpha-2}}.
\eee
If $r'\le |Z'|$, then since $|Z'|\le r'+|Z'|=|X'|\le 2|X|$ we get
\bee
|I^{(2)}_G(x)|&\lesssim&\frac{\|G\|_{\alpha}}{|X|^{D-1}} \int_{|Z'|\le 2|X|}\int_{r'\le |Z'|}\frac{r'^{D-1}dr'dz'}{a\la r'\ra^{\mu+\alpha}}\\
& \lesssim & \frac{\|G\|_{\alpha}}{|X|^{D-1}} \int_{|Z'|\le 2|X|}(r')^{D-\mu-\alpha}dZ'\lesssim \frac{\|G\|_{\alpha}}{|X|^{\mu+\alpha-2}}.
\eee
\noindent\underline{The case $|X'|\le 2|X|$, $|X-X'|\le \frac{|X|}{4}$}.  We distinguish two cases.
\noindent{\em The case $|Z|\le r$}. In this region we have $r\le |X|=r+|Z|\le 2r$ and hence $r\ge\frac{|X|}{2}\gtrsim 1$  thus $$
\left|\begin{array}{l}
|x_\perp-x_\perp'|\le\frac{|X|}{4}\le \frac{r}{2}\\
|Z'|\le |Z|+\frac{|X|}{2}\leq 4 r
\end{array}\right.
\Rightarrow \frac{r}{2}\le |x'_\perp|\le \frac{3r}{2}.
$$ 
We then estimate
\bee
|I^{(3)}_G(x)|&\lesssim&\|G\|_{\alpha} \int_{|X-X'|\le \frac{|X|}{4}}\frac{dx_\perp'dz'}{a|X-X'|^{D-1}\la r'\ra^{\mu+\alpha}}\\
&\lesssim&\|G\|_{\alpha}\int_{|x_\perp-x_\perp'|\le \frac{r}{2}}\frac{dx_\perp'}{\la r'\ra^{\mu+\alpha}}\int_{Z'\in \Bbb R}\frac{dZ'}{\left(|x_\perp-x'_\perp|^2+(Z-Z')^2\right)^{\frac{D-1}{2}}}\\
& \lesssim & \|G\|_{\alpha}\int_{|x_\perp-x_\perp'|\le \frac{r}{2}}\frac{dx_\perp'}{\la r'\ra^{\mu+\alpha}|x_\perp'-x_\perp|^{D-2}}\\
& \lesssim &  \frac{\|G\|_{\alpha}}{r^{\alpha+\mu}}\int_{\frac{r}{2}\le r'\le 2r}\frac{(r')^{D-1}dr'}{(r')^{D-2}}\lesssim \frac{\|G\|_{\alpha}}{r^{\alpha+\mu-2}}\lesssim \frac{\|G\|_{\alpha}}{|X|^{\alpha+\mu-2}}.
\eee
\noindent{\em The case $|Z|\ge r$}. In this region we have $|Z|\le |X|=r+|Z|\le 2|Z|$ and hence $$|Z-Z'|\le\frac{|X|}{4}\le \frac{|Z|}{2}\Rightarrow \frac{|Z|}{2}\le |Z'|\le \frac{3|Z|}{2}\Rightarrow \frac{|z|}{2}\le |z'|\le \frac{3|z|}{2},$$
and thus we have the contribution
\bee
|I^{(4)}_G(x)|
&\lesssim \frac 1{a\la z\ra^{\frac 12}}\int_{x_\perp'\in \Bbb R^D}\frac{dx_\perp'}{\la r'\ra^\mu}\left(\int_{z'\in \Bbb R}\la z'\ra G^2(r',z')dz'\right)^{\frac 12}\\
&\times\left(\int_{z'\in \Bbb R}\frac{dz'}{(|x_\perp-x_\perp'|^2+(Z-Z')^2)^{D-1}}\right)^{\frac 12},
\eee
which gives
\[|I^{(4)}_G(x)|\lesssim  \frac{\|G\|_{\alpha}}{a^{\frac 12}\la z\ra^{\frac 12}} \int_{x_\perp'\in \Bbb R^D}\frac{dx_\perp'}{\la r'\ra^{\alpha+\mu}|x_\perp-x_\perp'|^{D-\frac 32}}.
\]
Next we evaluate the various contributions:
\bee
\int_{|x_\perp'|\ge 2r}\frac{dx'_\perp}{\la r'\ra^{\mu+\alpha}|x_\perp-x_\perp'|^{D-\frac 32}}\lesssim \int_{r'\ge 2r}\frac{(r')^{D-1}dr'}{(r')^{\mu+\alpha+D-\frac 32}}\lesssim \frac{1}{r^{\mu+\alpha-\frac 32}},
\eee
\bee
\int_{|x_\perp-x'_\perp|\ge\frac{r}{2}, |x'_\perp|\le 2r}\frac{dx'_\perp}{\la r'\ra^{\mu+\alpha}|x_\perp-x_\perp'|^{D-\frac 32}}\lesssim \frac{1}{r^{D-\frac 32}}\int_{r'\le 2r}\frac{(r')^{D-1}}{(r')^{\mu+\alpha}}dr'\lesssim \frac{1}{r^{\mu+\alpha-\frac32}},
\eee
and 
\bee
\int_{|x_\perp-x'_\perp|\le\frac{r}{2}}\frac{dx'_\perp}{\la r'\ra^{\mu+\alpha}|x_\perp-x_\perp'|^{D-\frac 32}}\lesssim \frac{1}{\la r\ra^{\mu+\alpha}}\int_{u\le \frac{r}{2}}\frac{u^{D-1}du}{u^{D-\frac 32}}\lesssim \frac{1}{r^{\mu+\alpha-\frac 32}},
\eee
and hence the bound
$$|I^{(4)}_G(x)|\lesssim \frac{\|G\|_{\alpha}}{\la z\ra^{\frac 12}\la r\ra^{\mu+\alpha-\frac 32}}.$$

\noindent\underline{Conclusion}. The collection of above bounds yields 
$$ |I_G|\lesssim \|G\|_{\alpha}\left[\frac{1}{|X|^{\alpha+\mu-2}}+\frac{1}{\la z\ra^{\frac 12}\la r\ra^{\mu+\alpha-\frac 32}}\right],
$$
and \eqref{fondnantetabis} follows.
 \end{proof}

 \subsection{Pointwise bound}
 We can now turn to the derivation of a pointwise bound for $\Psi$.

\begin{lemma}[Pointwise bound]
\label{fristinog}
Under the assumptions of Lemma \ref{neoieonvndklnvdv}, consider $0\le \theta\le 1$, and let  $K_\nu=K_{c_0,C_0}$ denote some large enough universal constant,
 then there exists $m^*\in \mathbb{N}$, such that for $\gamma=2+\frac 1m$, $m\in \Bbb N, m\geq m^*$ there exists a universal constant $\alpha(\gamma)$ such that for $1\ll \alpha\leq \alpha(\gamma)$  and $0<\eta<\eta^*(\gamma)$ sufficiently small such that the following holds.\\
 
 \noindent\underline{For $R\le K_\nu R_*$ and $|z|\ge 1$}, we have\\
 \be
 \label{cneiovnenoienevn}
 |\Psi(x)|\lesssim \frac{\|\la z\ra^\frac{\theta}{2} \pa_r\Psi\|_{L^2(\Bbb R^5)}+\|f\|_{\alpha}}{a^{\frac 12}\la z\ra^{\frac{\theta}{2}}},
 \ee
 
 \noindent\underline{and for $R\ge K_\nu R_*$},
\bea
\label{potinwonioen}
\nonumber |\Psi(x)|&\lesssim&  \frac{\|\la z\ra^\theta \pa_r\Psi\|_{L^2(\Bbb R^5)}+\|f\|_{\alpha}}{a^{\frac 12}}\left[\frac{1}{|X|^{\frac{D}2}}+\frac{1}{\la z\ra^{\frac{\theta}{2}}\la r\ra^{\frac 12}R^{\frac{D}{2}}}\right]\\
 &+&\|f\|_{\alpha}\left[\frac{\la r\ra^\mu}{|X|^{\alpha+\mu-2}}+\frac{1}{\la z\ra^{\frac 12}\la r\ra^{\alpha-\frac 32}}\right].
\eea
\end{lemma}

\begin{proof}[Proof of Lemma \ref{fristinog}] This will follow from the combination of the energy bound \eqref{enguybaopor} and the pointwise bound \eqref{fondnanteta}.\\

\noindent{\bf step 1} Explicit resolvent. We solve explicitly the linear problem for $D\in \Bbb N^*$: $$\left(-\pa_r^2-\frac{D-1}{r}\Psi-a^2\pa_z^2\right)\Psi=f.$$
Let  
\be
\label{neoitnoanintoa}
\left|\begin{array}{l}Z=\frac{z}{a}\\
\psi(r,z)=r^2\Psi=\psit(r,Z)\\
\Psi(r,z)=\Psit(r,Z)\\ f(r,z)=\tilde{f}(r,Z)\\
X=x_\perp+Z\ve_z\\
 \end{array}\right.,
 \ee 
 then for $\Psi\in \dot H^1(\Bbb R^{D+1})$ we have
 \bea
 \label{ceniovneovnenvoe}
\nonumber  &&\left(-\pa_r^2-\frac{D-1}{r}\Psi-a^2\pa_z^2\right)\Psi=f\Leftrightarrow -\Delta_{\Bbb R^{D+1}}\Psit=\tilde{f}\\
 \nonumber & \Leftrightarrow& \Psit(X)=-c\int_{X'\in \Bbb R^{D+1}}\frac{f(X')}{|X-X'|^{D-1}}dX'\\
 &\Leftrightarrow & \Psi(x)=c\int_{x\in \Bbb R^{D+1}}\frac{f(x)}{a|X-X'|^{D-1}}dx.
 \eea
\noindent{\bf step 2} Splitting of the flow. The potential $W(r,z)=\frac{W_*(R)}{\nu^2}$ satisfies 
\be
\label{estitmot}
W(r,z)\ge \frac{1}{\nu^2}\frac{L}{R^2}=\frac{L}{r^2}\ \ \mbox{for}\ \ R\ge R_*,
\ee 
for some $R_*(\gamma)$ large enough and $L=L(\gamma)\to +\infty$ as $\gamma \downarrow 2$, this decay of the potential will be key in getting the desired fast decay in $r$ quantified by $\alpha$. We therefore rewrite the flow \eqref{vneoneonenvoi} as
\be
\label{esuibtibg}
\left|\begin{array}{l}
\left[-\pa_r^2-\frac{3}{r}\pa_r+W_{\rm out}-a^2\pa_z^2\right]\Psi=f+D_fh+W_{\rm in}\Psi\\
\frac{L}{1+R^2}\le W_{\rm out}\leq\frac{C_\gamma}{1+R^2}\\
|W_{\rm in }|\lesssim \frac{{\bf 1}_{R\le R_*}}{\nu^2}
\end{array}\right..
\ee
Recall that $f\in \mathcal S$ ensures from standard elliptic regularity that $\Psi\in \mathcal C^\infty.$\\

\noindent{\bf step 3} Bound for $R\le K_\nu R_*$, $|z|\ge 1$. Since $W_{\rm out}\ge 0$ and $\Psi\in \mathcal C_0(\Bbb R^5)$, a standard application of the maximum principle and \eqref{esuibtibg} ensures
$$|\Psi(x)|\lesssim\int_{\Bbb R^5}\frac{|F|}{a|X-X'|^3}dx', \ \ F=f+D_fh+W_{\rm in}\Psi.$$ We invoque \eqref{estiinidoioeiog} with $R_0=K_\nu R_*$: 
$$|J_F(x)|{\bf 1}_{R\le K_\nu R_*,|z|\ge 1}\leq  C_{R_0}\left[\frac{\|\la z\ra^\theta F\|_{L^2(\Bbb R^5)}}{a^{\frac 12}\la z\ra^{\frac{\theta}{2}}}+ \int_{|x_\perp-x'_\perp|\le 1, \frac{|z|}{2}\le |z|\le\frac{3|z|}{2}}\frac{|F(x')|}{a|X-X'|^3}dx'\right],
$$
and now estimate the various contributions of $F$.\\

\noindent\underline{Source term}. Since $2\theta \le 1$ we have
$$\|\la z\ra^{\frac{\theta}{2}} f\|_{L^2(\Bbb R^5)}^2\lesssim \|f\|_{\alpha}^2\int_{r>0}\frac{r^3dr}{\la r\ra^{2\alpha}}\lesssim \|f\|_{\alpha}^2,$$
and 
\bee
&&\int_{|x_\perp-x'_\perp|\le 1, \frac{|z|}{2}\le |z'|\le\frac{3|z|}{2}}\frac{|f(x')|}{a|X-X'|^3}dx'\\
&\lesssim & \frac{\|f\|_{\alpha}}{a\la z\ra^{\frac 12}}\int_{|x_\perp-x'_\perp|\le 1}\frac{dx'_\perp}{\la r'\ra^\alpha}\left(\int_{Z'\in \Bbb R}\frac{a\la Z-Z'\ra dZ'}{(|x_\perp-x_\perp'|^2+(Z-Z')^2)^{3}}\right)^{\frac 12}\\
& \lesssim & \frac{\|f\|_{\alpha}}{a^{\frac 12}\la z\ra^{\frac 12}}\int_{|x_\perp-x'_\perp|\le 1}\frac{1}{\la r'\ra^\alpha}\left(\frac{1}{|x_\perp-x_\perp'|^{\frac 52}}+\frac{1}{|x_\perp-x_\perp'|^2}\right)dx'_\perp\lesssim  \|f\|_{\alpha},
\eee
where we recall that the last integral is over $\mathbb{R}^4$.\\

\noindent\underline{Potential term}. We use $|W_{\rm in }|\lesssim \frac{{\bf 1}_{r\le R_*\nu}}{\nu^2}\lesssim \frac{1}{\la r\ra^2}\lesssim \frac{1}{\la r\ra}$ and the four dimensional Hardy inequality to estimate
$$\|\la z\ra^{\frac{\theta}{2}} W_{\rm in}\Psi\|_{L^2(\Bbb R^5)}^2\lesssim \|\la z\ra^{\frac{\theta}{2}}\frac{\Psi}{r}\|_{L^2(\Bbb R^5)}^2\lesssim \|\la z\ra^{\frac{\theta}{2}}\pa_r\Psi\|_{L^2(\Bbb R^5)}^2.$$ We now use the radial Sobolev embedding 
$$\la z\ra^{\frac{\theta}{2}}|\Psi(r,z)|\lesssim \frac{1}{r}\left(\int_{r>0}\la z\ra^{\theta}(\pa_r\Psi)^2r^3dr\right)^{\frac 12},$$ to estimate
\bee
&&\int_{|x_\perp-x'_\perp|\le 1, \frac{|z|}{2}\le |z'|\le\frac{3|z|}{2}}\frac{|W_{\rm in}\Psi|}{a|X-X'|^3}dx'\\
&\lesssim &\frac{1}{a\la z\ra^{\frac{\theta}{2}}}\int_{|x_\perp-x'_\perp|\le 1, \frac{|z|}{2}\le |z'|\le\frac{3|z|}{2}}\frac{dx_\perp'dz'}{r'\la r'\ra^2|X-X'|^3}\left(\int_{u>0}\la z'\ra^{\theta}(\pa_r\Psi(u,z'))^2u^3du\right)^{\frac 12}\\
& \lesssim & \frac{1}{a\la z\ra^{\frac{\theta}{2}}}\int_{|x_\perp-x'_\perp|\le 1}\frac{dx_\perp'}{r'\la r'\ra^2}\left(\int_{z'\in \Bbb R}\int_{u>0}\la z'\ra^{\theta}(\pa_r\Psi(u,z'))^2u^3dudz'\right)^{\frac 12}\\
& & \times\left(\int_{z'\in \Bbb R}\frac{dz'}{(|x_\perp-x_\perp'|^2+(Z-Z')^2)^3}\right)^{\frac 12}\\
& \lesssim &  \frac{\|\la z\ra^{\frac{\theta}{2}}\pa_r\Psi\|_{L^2(\Bbb R^5)}}{a^{\frac 12}\la z\ra^{\frac{\theta}{2}}}\int_{|x_\perp-x'_\perp|\le 1}\frac{dx_\perp'}{r'\la r'\ra^2|x_\perp-x'_\perp|^{\frac 52}}\lesssim \frac{\|\la z\ra^{\frac{\theta}{2}}\pa_r\Psi\|_{L^2(\Bbb R^5)}}{a^{\frac 12}\la z\ra^{\frac{\theta}{2}}}.
\eee
\noindent\underline{$D_fh$ term}. We recall \eqref{vnovnnvenevbvbbvk} which implies
\bee
|D_1|&\lesssim& \frac{1}{\nu^6\la R\ra^{|\l_-|}}\int_{r'>0}\frac{|\Psi|}{\la R'\ra^{|\l_-|}}(r')^3dr'\\
&\lesssim& \frac{1}{\nu^6\la R\ra^{|\l_-|}}\left(\int_{r'>0}\frac{\Psi^2}{(r')^2}(r')^3dr'\right)^{\frac 12}\left(\int_{r'<\nu}\frac{(r')^3(r')^2}{\la R'\ra^{2|\l_-|}}dr'\right)^{\frac 12}\\
&\lesssim &\frac{1}{\nu^3\la R\ra^{|\l_-|}}\left(\int_{r'>0}\frac{\Psi^2}{(r')^2}(r')^3dr'\right)^{\frac 12},
\eee
and hence
\bea
\label{estinfiosfnoe}
\|\la z\ra^{\frac{\theta}{2}}  D_1\|^2_{L^2(\Bbb R^5)}&\lesssim & \int_{z>0}\int_{r>0}\la z\ra^{\theta}\frac{r^3dr}{\nu^6\la R\ra^{|\l_-|}}\left(\int_{r'>0}\frac{\Psi^2(r',z)}{(r')^2}(r')^3dr'\right)dz \nonumber \\
&\lesssim &\|\la z\ra^{\frac{\theta}{2}}\pa_r\Psi\|_{L^2(\Bbb R^5)}^2.
\eea
Moreover as above:
\bee
&&\int_{|x_\perp-x'_\perp|\le 1, \frac{|z|}{2}\le |z'|\le\frac{3|z|}{2}}\frac{|D_1|}{a|X-X'|^3}dx'\\
&\lesssim &\frac{1}{a\la z\ra^{\frac{\theta}{2}}}\int_{|x_\perp-x'_\perp|\le 1, \frac{|z|}{2}\le |z'|\le\frac{3|z|}{2}}\frac{dx_\perp'dz'}{\la r'\ra^3|X-X'|^3}\left(\int_{u>0}\la z'\ra^{\theta}(\pa_r\Psi(u,z'))^2u^3du\right)^{\frac 12}\\
& \lesssim &  \frac{\|\la z\ra^{\frac{\theta}{2}}\pa_r\Psi\|_{L^2(\Bbb R^5)}}{a^{\frac 12}\la z\ra^{\frac{\theta}{2}}}.
\eee
We now recall \eqref{vdjkvbndkndlknvdlnv}:
$$|D_2(r,z)|\lesssim \frac{1}{\nu^3\la R\ra^{\gamma-2+|\l_-|}}\left(\int_{r>0}f^2r^2r^3dr\right)^{\frac12}
$$
which implies using $\theta\le 1$:
\bea
\label{venoivneinveineonveni}
\nonumber \|\la z\ra^{\frac{\theta}{2}}  D_2\|^2_{L^2(\Bbb R^5)} & \lesssim & \int_{r>0,z\in \Bbb R}\frac{r^3drdz}{\nu^6\la R\ra^{2|\l_-|}}\int_{u>0}\la z\ra^{\theta}f^2u^2u^3du\\
\nonumber& \lesssim & \int_{u>0}\int_{z\in \Bbb \R}\la z\ra^{\theta}f^2u^2u^3\left[\int_{r>0}\frac{r^3dr}{\nu^6\la R\ra^{2|\l_-|}}\right]dudz\lesssim \|f\|_{\alpha}^2\int_{u>0}\frac{u^5}{\la u\ra^{2\alpha}}du\\
&\lesssim& \|f\|^2_{\alpha},
\eea
  which gives
 \bee
    & &\int_{|x_\perp-x'_\perp|\le 1, \frac{|z|}{2}\le |z'|\le\frac{3|z|}{2}}\frac{|D_2|}{a|X-X'|^3}dx'\\
& \lesssim &  \frac{1}{a\la z\ra^{\frac 12}}\int_{x'_\perp\in\Bbb R^4}\frac{dx_\perp'}{\la r'\ra^3}\left(\int_{z'\in \Bbb R}\int_{u>0}\la z'\ra f^2(u,z')u^2u^3du\right)^{\frac 12}\\
& &\times \left(\int_{Z'\in \Bbb R}\frac{adZ'}{(|x_\perp-x_\perp'|^2+(Z-Z')^2)^{3}}\right)^{\frac 12},
\eee
thus
\[
\int_{|x_\perp-x'_\perp|\le 1, \frac{|z|}{2}\le |z'|\le\frac{3|z|}{2}}\frac{|D_2|}{a|X-X'|^3}dx' \lesssim  \frac{\|f\|_{\alpha}}{\la z\ra^{\frac 12}}\int_{x'_\perp\in\Bbb R^4}\frac{dx_\perp'}{\la r'\ra^3|x_\perp-x_\perp'|^{\frac 52}}\lesssim \frac{\|f\|_{\alpha}}{\la z\ra^{\frac 12}}.
\]
\noindent\underline{Conclusion}. The collection of above bounds concludes the proof of \eqref{cneiovnenoienevn}.\\
 
\noindent{\bf step 4} Explicit upper bound\footnote{It is here that we will use in a crucial way the $\frac{L}{r^2}$ with $L\gg1$ tail of the potential $W$ to leverage fast anisotropic decay in the $r$ variable} for $R\ge K_\nu R_*$.  Let $\Psi=r^\mu\Phi$, then we have
\bee
&&\left[-\pa_r^2-\frac{3}{r}\pa_r+W_{\rm out}-a^2\pa_z^2\right]\Psi=f+D_fh+W_{\rm in}\Psi\\
&\Leftrightarrow& \left[-\pa_r^2-\frac{2\mu+ 3}{r}\pa_r-\frac{\mu(\mu+2)}{r^2}-a^2\pa_z^2+W_{\rm out}\right]\Phi=\frac{f+D_fh+W_{\rm in}\Psi}{r^\mu}\\&\Leftrightarrow& \left[-\pa_r^2-\frac{D-1}{r}\pa_r-\frac{\mu(\mu+2)}{r^2}-a^2\pa_z^2+W_{\rm out}\right]\Phi=\frac{f+D_fh+W_{\rm in}\Psi}{r^\mu}.
\eee
Let a cut off function 
\be
\label{vneiveneniovn}
\chi(r)=\left|\begin{array}{l} 0 \ \ \mbox{for}\ \ r\le r_0\\ 1\ \ \mbox{for}\ \ r\ge \frac{r_0}2
\end{array}\right.,
\ee
and commute with the equation to get 
\bee
&&\left[-\pa_r^2-\frac{D-1}{r}\pa_r-\frac{\mu(\mu+2)}{r^2}-a^2\pa_z^2+W_{\rm out}\right](\chi \Phi)\\
&=& \chi\frac{f+D_fh+W_{\rm in}\Psi}{r^\mu}-(\Delta_{\Bbb R^D}\chi)\Psi-2\pa_r\chi\pa_r\Psi\equiv \frac{F}{\la r\ra ^{\mu}},
\eee with $$\Phi=\frac{\Psi}{r^\mu}\to 0\ \ \mbox{as}\ \ (r,z)\to +\infty,$$ and $\Phi\in \mathcal C^\infty(\Bbb R^{D+1})$. We now observe using \eqref{esuibtibg}:
\bee
 W_{\rm out}-\frac{\mu(\mu+2)}{r^2}& \ge & \frac{L}{1+R^2}-\frac{\mu(\mu+2)}{r^2}=\frac{L\nu^2}{r^2+\nu^2}-\frac{\mu(\mu+2)}{r^2}\\
& = & \frac{\left[L\nu^2-\mu( \mu+2)\right]r^2-\mu(\mu+2)\nu^2}{r^2(r^2+\nu^2)}.
\eee
We may assume $L=L(\gamma)$ is large enough so that $$\mu(\mu+2)\le \frac{c^2_0L}{2} \le \frac{L\nu^2}{2},$$ 
and hence $$r^2>\frac{\mu(\mu+2)\nu^2}{\frac{L\nu^2}{2}}=\frac{2\mu(\mu+2)}{L}\Rightarrow W_{\rm out}-\frac{\mu(\mu+2)}{r^2}\ge 0.$$ We therefore choose $r_0=\sqrt{\frac{20\mu(\mu+2)}{L}}$ in  \eqref{vneiveneniovn} and conclude using the maximum principle and \eqref{ceniovneovnenvoe}:
\bee
\forall r>0, \ \ \chi|\Phi|&=&\frac{\chi|\Psi|}{r^\mu}\lesssim \int_{x\in \Bbb R^{D+1}}\frac{|F(x)|}{a\la r'\ra^\mu|X-X'|^{D-1}}dx\\
& \Rightarrow & \forall r\ge 2r_0, \ \ |\Psi(x)|\le \la r\ra^\mu \int_{x\in \Bbb R^{D+1}}\frac{|F(x)|}{a\la r'\ra^\mu|X-X'|^{D-1}}dx\equiv \mathcal I(x),
\eee
Now we will estimate all terms in the above right hand side.\\

\noindent\underline{Source term}. For the main source term, we have $$\left\|\frac{\la r\ra^{\mu}}{r^\mu}\chi f\right\|_{\alpha}\lesssim \|f\|_{\alpha},$$ and hence from \eqref{fondnantetabis} the corresponding contribution for $R\ge K_\nu R_*$:

$$|\mathcal I^{(1)}(x)|\lesssim \|f\|_{\alpha}\left[\frac{\la r\ra^\mu}{|X|^{\alpha+\mu-2}}+\frac{1}{\la z\ra^{\beta}\la r\ra^{\alpha-\frac 32}}\right].
$$

\noindent\underline{Localization term}. $\frac{r_0}2\le r\le 2r_0\Rightarrow R\le R_*$ provided $R_*(\gamma)$ has been chosen large enough, and we apply Lemma \ref{neoieonvndklnvdv} to estimate the corresponding contribution:
\bee
|\mathcal I^{(2)}(x)|
&\lesssim &  \frac{\|\la z\ra^{\frac{\theta}{2}} |(\Delta_{\Bbb R^D}\chi)\Psi+2\pa_r\chi\pa_r\Psi|\|_{L^2(\Bbb R^5)}}{a^{\frac 12}}\left[\frac{1}{|X|^{\frac{D}2}}+\frac{1}{\la z\ra^{\frac{\theta}{2}}\la r\ra^{\frac 12}R^{\frac{D}{2}}}\right]\\
& \lesssim &  \frac{\|\la z\ra^{\frac{\theta}{2}} \pa_r\Psi\|_{L^2(\Bbb R^5)}}{a^{\frac 12}}\left[\frac{1}{|X|^{\frac{D}2}}+\frac{1}{\la z\ra^{\frac{\theta}{2}}\la r\ra^{\frac 12}R^{\frac{D}{2}}}\right],
\eee

\noindent\underline{Potential term}. Observe that $$r\le 2( \nu^2+r^2)\Rightarrow |W_{\rm in }|\lesssim \frac{1}{\nu^2\left(1+\frac{r^2}{\nu^2}\right)}\leq \frac{2}{r}, $$ and since $W_{\rm in}$ is localized in $R\le R^*$ by definition, we apply Lemma \ref{neoieonvndklnvdv} to estimate the corresponding contribution:
\bee
|\mathcal I^{(3)}(x)|&\lesssim  \frac{\|\la z\ra^{\frac{\theta}{2}}  W_{\rm in }\Psi\|_{L^2(\Bbb R^5)}}{a^{\frac 12}}\left[\frac{1}{|X|^{\frac{D}2}}+\frac{1}{\la z\ra^{\frac{\theta}{2}}\la r\ra^{\frac 12}R^{\frac{D}{2}}}\right]\\
& \lesssim  \frac{\|\la z\ra^\theta \frac{\Psi}{r}\|_{L^2(\Bbb R^5)}}{a^{\frac 12}}\left[\frac{1}{|X|^{\frac{D}2}}+\frac{1}{\la z\ra^{\frac{\theta}{2}}\la r\ra^{\frac 12}R^{\frac{D}{2}}}\right]\\
&\lesssim  \frac{\|\la z\ra^{\frac{\theta}{2}} \pa_r\Psi\|_{L^2(\Bbb R^5)}}{a^{\frac 12}}\left[\frac{1}{|X|^{\frac{D}2}}+\frac{1}{\la z\ra^{\frac{\theta}{2}}\la r\ra^{\frac 12}R^{\frac{D}{2}}}\right].
\eee 
\noindent\underline{$D_fh$ term}. Since $h$ is supported in $R\le R_*$, we may apply Lemma \ref{neoieonvndklnvdv}. We recall \eqref{estinfiosfnoe} and \eqref{venoivneinveineonveni}:
$$\|\la z\ra^{\frac{\theta}{2}}  D_1\|_{L^2(\Bbb R^5)}+\|\la z\ra^{\frac{\theta}{2}}  D_2\|_{L^2(\Bbb R^5)}\lesssim \|\la z\ra^{\frac{\theta}{2}}\pa_r\Psi\|_{L^2(\Bbb R^5)}+\|f\|_{\alpha}.$$
 Hence the estimate of the corresponding contribution
$$|\mathcal I^{(4)}(x)|\lesssim \frac{\|\la z\ra^{\frac{\theta}{2}}\pa_r\Psi\|_{L^2(\Bbb R^5)}+\|f\|_{\alpha}}{a^{\frac 12}}\left[\frac{1}{|X|^{\frac{D}2}}+\frac{1}{\la z\ra^{\frac{\theta}{2}}\la r\ra^{\frac 12}R^{\frac{D}{2}}}\right].$$
The collection of the above bounds concludes the proof of \eqref{potinwonioen}.
\end{proof}

\subsection{Weighted $L^2$ bound}
We are now in position to improve the energy bound to a weighted $L^2$ bound.

\begin{lemma}[Weighted $L^2$ bound]
\label{lemmaltwobound}
Under the assumptions of Proposition \ref{proppointizedfjifw}, there holds
\be
\label{imporvedlrotwbound}
\left(\int_{\Bbb R^5}\la r\ra^{2(\alpha-6)}\la z\ra^{3-2\delta_*}\left[(\pa_r\Psi)^2+a^2(\pa_z\Psi)^2\right]dx\right)^{\frac 12}\lesssim\|f\|_{\alpha}.
\ee

\end{lemma}

\begin{proof}[Proof of Lemma \ref{lemmaltwobound}] This will follow from the pointwise bound \eqref{potinwonioen} and a $L^2$ multiplier argument.\\

\noindent{\bf step 1} Weight function in $z$. Consider an even multiplier 
\be
\label{vneonvnvoenenvo}
m_A(z)=\left|\begin{array}{l} 1\ \ \mbox{for}\ \ 0\le z\le \frac 12\\ z \ \ \mbox{for}\  \ 1\le z\leq A\\ 4A\ \ \mbox{for}\ \ z\ge 2A
\end{array}\right., \ \ m_A(z)>0,
\ee with the concavity 
\be
\label{concavcionefm}
\forall z\ge 1, \ \ m_A''(z)\le 0.
\ee
Since $\psi\in V$ we have $m_A\psi\in V$ and hence \eqref{estaimteV} yields:
$$\int_{r>0}\int_{z\in \Bbb R}(L_a\phi)m_A(z)\psi rdrdz=\int_{r>0}\int_{z\in \Bbb R} fm_A(z)\psi rdrdz,$$ which after passing to $\psi=r^2\Psi$ and integrating by parts gives
$$
\int_{\Bbb R^5}m_{A}\left[(\pa_r\Psi)^2+a^2(\pa_z\Psi)^2\right]dx+\frac 12\int_{\Bbb R^5}\left(-a^2\pa_z^2m_A+m_{A}W\right)\Psi^2dx=  \int_{\Bbb R^5}fm_A(z)\Psi dx.
$$
From the four dimensional Hardy and the uniform bound $|m_A|\lesssim \la z\ra$ we get for all $\e>0$, 
$$\left|\int_{\Bbb R^5}fm_A(z)\Psi dx\right|\lesssim \|\sqrt{m_A}rf\|_{L^2(\Bbb R^5)}\|\sqrt{m_A}\frac{\Psi}{r}\|_{L^2(\Bbb R^5)}\leq C_\e\|f\|_{\alpha}^2+\e\int_{\Bbb R^5}m_A(\pa_r\Psi)^2dx.$$
We now use in a fundamental way the pointwise bound \eqref{potinwonioen} for $|z|\le 1$. Indeed, for $R=\frac{r}{\nu}\le K_\nu R_*\Rightarrow r\le K_\nu R_*\nu(z)\lesssim1$, and hence the energy bound ensures
$$\int_{\Bbb R^5} a^2{\bf 1}_{|z|\le 1}\Psi^2dx\lesssim a^2\int_{\Bbb R^5}\frac{\Psi^2}{r^2}dx\lesssim a^2\|f\|_{\alpha}^2.$$ In the non compact in $r$ zone $R\ge K_\nu R_*$, \eqref{potinwonioen} with $\theta=0$ and the energy bound ensure the $A$ independent bound:
$$\int_{\Bbb R^5} a^2{\bf 1}_{|z|\le 1}\Psi^2dx\lesssim \|f\|_{\alpha}^2\left[\frac{a^2}{a}\int_{|z|\le 1}dz\int_{r\ge 1}\frac{r^3dr}{r^{2(\alpha-\frac 32)}}\right]\lesssim a\|f\|_{\alpha}^2.$$ Using the convexity \eqref{concavcionefm}, the collection of above bounds ensures the $A$ independent bound:
$$
\int_{\Bbb R^5}m_{A}(z)\left[(\pa_r\Psi)^2+a^2(\pa_z\Psi)^2+W\Psi^2\right]dx\lesssim \|f\|_{\alpha}^2.$$ Since $m_A$ depends only on $z$, Lemma \ref{coercm} with $\psi\in V$ ensures:
\bee & & c_*\int_{\Bbb R^5}m_A(z)\left[(\pa_r\Psi)^2+a^2(\pa_z\Psi)^2\right]dx  \\
&  \leq & \int_{\Bbb R^5} m_A(z)\left[(\pa_r\Psi)^2+a^2(\pa_z\Psi)^2\right]dx+\int_{\Bbb R^5} m_A(z)W\Psi^2 \eee for some universal constant $c_*>0$. We may therefore let $A\to+\infty$ and  conclude using Fatou's Lemma and the energy bound:
\be
\label{fihihbond}
\int_{\Bbb R^5}\la z\ra \left[(\pa_r\Psi)^2+a^2(\pa_z\Psi)^2\right]dx\lesssim \|f\|_{\alpha}^2.
\ee

\noindent{\bf step 2} First improved bound. We claim 
\be
\label{fheiofehehigeog}
\int_{\Bbb R^5}\la r\ra^{2\alpha-10}\la z\ra^{1-\delta_*} \left[(\pa_r\Psi)^2+a^2(\pa_z\Psi)^2\right]dx\lesssim \|f\|_{\alpha}^2.
\ee
For this we consider an even multiplier $$m_0(z)=\left|\begin{array}{l}1\ \ \mbox{for}\ \ z\le \frac 12\\ z^{1-\delta_*}\ \ \mbox{for}\ \ z\ge 1
\end{array}\right.,
$$
and for $B\ge 1$ a radial multiplier: $$\chi_B(r)=\left|\begin{array}{l} 1\ \ \mbox{for}\ \ r\le \frac 12\\ r^{2\alpha-10}\ \ \mbox{for}\ \ 1\le r\le B\\  (2B)^{\alpha-10}\ \ \ \ \mbox{for}\ \ r\ge 2B
\end{array}\right., \ \ \chi'\ge 0,$$ and define $$m_B(r,z)=m_0(z)\chi_B(r),$$ we integrate by parts in \eqref{vneoneonenvoi}  and compute:
\bea
\label{vneoivnevenvenveonenv}
 & & \int_{\Bbb R^5}m_{B}\left[(\pa_r\Psi)^2+a^2(\pa_z\Psi)^2\right]dx \nonumber \\ &+ &\frac 12\int_{\Bbb R^5}\left(-\pa_r^2m_{B}-\frac{3}{r}\pa_rm_{B}-a^2\pa_z^2m_B+m_{B}W\right)\Psi^2 \nonumber \\
 &  = & \int_{\Bbb R^5}(f+D_f(z)h)m_B\Psi dx.
 \eea
 We estimate all terms in the above identity.\\
 
 \noindent\underline{Source term}
\bee
    \left|\int_{\Bbb R^5}fm_B\Psi dx\right| &\lesssim &\left(\int_{\Bbb R^5}\la r\ra^{2\alpha-10}\la r\ra^2\la z\ra^{1-\delta_*}f^2dx\right)^{\frac 12}\left(\int_{\Bbb R^5}m_B\frac{\Psi^2}{\la r\ra^2}dx\right)^{\frac 12}\\
    &\lesssim &\|f\|_\alpha^2+\int_{\Bbb R^5}m_B\frac{\Psi^2}{\la r\ra^2}dx,
\eee
and the second term is treated below with the localization terms.\\

 \noindent\underline{Localization terms}. 
 We estimate for $r\ge 1$:
 $$\frac{1}{m_B}\left(|\pa_r^2m_{B}|+\frac{|\pa_rm_{B}|}r\right)\leq \frac{C\alpha^2}{r^2},$$ and hence for $R\ge R_*$ from \eqref{estitmot} and $\gamma$ close enough to 2: 
 $$m_B W\gg |\pa_r^2m_{B}|+\frac{|\pa_rm_{B}|}r.$$ For $R\le R_*$, we use the radial Sobolev inequality to get:
 $$|\Psi(r,z)|\lesssim \frac{1}{r\la z\ra^{\frac 12}}\left(\int_{u>0}\la z\ra(\pa_r\Psi)^2(u,z)u^3du\right)^{\frac 12},$$ with \eqref{fihihbond} to estimate
for $\eta$ small enough:
 \bee
 &&\int_{z\in \Bbb R}\int_{R\le R_*}m_B\Psi^2dx\\
 &\lesssim & \int_{z\in \Bbb R}\int_{r\le \nu}\frac{\la r\ra^{2\alpha-10}\la z\ra^{1-\delta_*}}{r^2\la z\ra}r^3\left(\int_{u>0}\la z\ra(\pa_r\Psi)^2(u,z)u^3du\right)drdz\\
 &\lesssim &  \int_{z\in \Bbb R}\int_{u>0}\la z\ra^{1+(2\alpha-8)\eta-\delta_*}(\pa_r\Psi)^2(u,z)u^3du\lesssim \|f\|_{\alpha}^2,
 \eee
 for $(2\alpha-8)\eta<\delta$. For the last term we use the bound $|\pa_z^2m_B|\lesssim \frac{1}{\la z\ra^{1+\delta_*}}$ and split into two regions. For $R\ge K_\nu R_*$, we use the pointwise bound \eqref{potinwonioen} with $\theta=0$ which implies the bound:
 $$|\Psi(x)|\lesssim \frac{\|f\|_{\alpha}}{a^{\frac 12}\la r\ra^{\alpha-2}},$$
 and hence the bound:
\bee
&&\int_{\Bbb R^5}a^2\la r\ra^{2\alpha-10}|\pa_z^2m_B|\Psi^2\\
&\lesssim& a^2 \int_{z\in \Bbb R}\int_{r\le R_*\nu}\la r\ra^{2\alpha-10}\frac{\Psi^2}{\la z\ra^{1+\delta_*}}r^3drdz+
\int_{r>0} \int_{z\in \Bbb R}\frac{a^2}{a}\|f\|_{\alpha}^2\frac{\la  r\ra^{2\alpha-10}r^3drdz}{\la r\ra^{2(\alpha-2)}\la z\ra^{1+\delta_*}}\\
& \lesssim & a^2 \int_{z\in \Bbb R}\int_{r>0}\frac{\Psi^2}{r^2}\frac{\la z\ra^{(2\alpha-10)\eta}}{\la z\ra^{1+\delta_*}}r^3drdz+\|f\|_\alpha^2\lesssim \|f\|_\alpha^2,
\eee
for $\eta$ sufficiently small.
\\

 \noindent\underline{$D_f$ terms}. Recall \eqref{estiaofnrdiune} which yields:
  $$r|D_1|\lesssim \frac{1}{\nu^2\la R\ra^{|\l_-|}}\left(\int \frac{\Psi^2}{r^2}r^3dr\right)^{\frac 12}\lesssim \frac{1}{\la R\ra^{|\l_-|}\la z\ra^{\frac 12}}\left(\int_{r>0}\la z\ra(\pa_r\Psi)^2r^3dr\right)^{\frac 12},$$ and hence using \eqref{fihihbond}:
 \bee
&& \left|\int_{\Bbb R^5}D_1m_B\Psi dx\right|\lesssim\left(\int_{\Bbb R^5}\la z\ra^{1-\delta_*}|D_1|^2r^2\la r\ra^{4\alpha-20}dx\right)^{\frac 12}\left(\int_{\Bbb R^5}\frac{\la z\ra^{1-\delta_*}\Psi^2}{r^2}dx\right)^{\frac 12}\\
 & \lesssim & \|f\|_\alpha\left[\int_{z>0}\frac{\la z\ra^{1-\delta_*}dz}{\la z\ra}\int_{u>0}\la z\ra(\pa_r\Psi)^2(u,z)u^3du\left(\int_{r>0}\frac{\la r\ra^{4\alpha-20}r^3dr}{\la R\ra^{|\l_-|}}\right)\right]^{\frac 12}\\
 & \lesssim & \|f\|_\alpha\left(\int_{z>0}\int_{u>0}\frac{\la z\ra(\pa_r\Psi)^2(u,z)}{\la z\ra^{\delta_*-(4\alpha-16)\eta}}u^3dudz\right)^{\frac 12}\lesssim  \|f\|_\alpha^2,
 \eee
 for $\eta$ sufficiently small and $\alpha \ll |\lambda_-|$. Next
 recall \eqref{vdjkvbndkndlknvdlnv}
$$|D_2(r,z)|\lesssim\frac{1}{\nu^3\la R\ra^{\gamma-2+|\l_-|}}\left(\int_{r>0}f^2r^2r^3dr\right)^{\frac12},
$$ which yields:
\bee
&& \left|\int_{\Bbb R^5}D_2m_B\Psi dx\right|\lesssim\left(\int_{\Bbb R^5}\la z\ra^{1-\delta_*}|D_2|^2r^2\la r\ra^{4\alpha-20}dx\right)^{\frac 12}\left(\int_{\Bbb R^5}\frac{\la z\ra^{1-\delta_*}\Psi^2}{r^2}dx\right)^{\frac 12}\\
& \lesssim & \|f\|_{\alpha}\left[\int_{z\in \Bbb R}\int_{u>0}\la z\ra^{1-\delta_*}f^2(u,z)u^5du\left(\int_{r\le \nu}\frac{\la r\ra^{4\alpha}}{\la R\ra^{|\l_-|}}\right)dz\right]^{\frac 12}\lesssim \|f\|_{\alpha}^2,
\eee
again for $\eta$ sufficiently small and $\alpha\ll |\lambda_-|$.\\

\noindent\underline{Conclusion} The collection of above bounds yields:
$$\int_{\Bbb R^5}m_{B}\left[(\pa_r\Psi)^2+a^2(\pa_z\Psi)^2\right]dx\lesssim \|f\|_{\alpha}^2,
$$ with a constant independent of $B$, and hence $B\to+\infty$ and Fatou's lemma yield \eqref{fheiofehehigeog}.\\

\noindent{\bf step 3} Improved decay in $z$. We claim
\be
\label{vneivneeionenenvoe}
\int_{\Bbb R^5}\la z\ra^{3-\frac{3}{2}\delta_*}\left[(\pa_r\Psi)^2+a^2(\pa_z\Psi)^2\right]dx\lesssim \|f\|_{\alpha}^2.
\ee
Indeed, we rerun step 1 with the even multiplier
\be
\label{vneonvnvoenenvobis}
m_A(z)=\left|\begin{array}{l} 1\ \ \mbox{for}\ \ 0\le z\le \frac 12\\ z^{3-\frac{3}{2}\delta_*} \ \ \mbox{for}\  \ 1\le z\leq A\\ (4A)^{2+\frac 14}\ \ \mbox{for}\ \ z\ge 2A
\end{array}\right., \ \ m_A(z)>0,
\ee 
and obtain since the multiplier is $r$ independent:
\bee
&&c_*\int_{\Bbb R^5}m_{A}(z)\left[(\pa_r\Psi)^2+a^2(\pa_z\Psi)^2\right]dx\\
& \leq& \int_{\Bbb R^5}m_{A}\left[(\pa_r\Psi)^2+a^2(\pa_z\Psi)^2\right]dx+\frac 12\int_{\Bbb R^5}\left(-a^2\pa_z^2m_A+m_{A}W\right)\Psi^2dx\\
&=&  \int_{\Bbb R^5}fm_A(z)\Psi dx.
\eee
Using $|m_A''|\lesssim \la z\ra^{1-\frac{3}{2}\delta_*}$, \eqref{fheiofehehigeog} and the radial Sobolev we get:
\bee
|\Psi|(r,z)&\lesssim& \frac{1}{\la z\ra^{\frac{1-\delta_*}{2}}}\left(\int_{u>0}(\pa_r\Psi^2)\la u\ra^{2\alpha-10}\la z\ra^{1-\delta_*}u^3du\right)^{\frac 12}\left(\int_{u>r}\frac{du}{u^3\la u\ra^{2\alpha-10}}\right)^{\frac 12}\\
&\lesssim& \frac{1}{\la z\ra^{\frac{1-\delta_*}{2}}r\la r\ra^{\alpha-5}}\left(\int_{u>0}(\pa_r\Psi^2)\la u\ra^{2\alpha-10}\la z\ra^{1-\delta_*}u^3du\right)^{\frac 12},
\eee
for $\alpha$ large enough and hence we get the uniform in $A$ bound using \eqref{fheiofehehigeog}:
\bee
&&a^2\int_{z\in \Bbb R}\int_{r>0}|\pa_z^2m_A|\Psi^2dx\\
&\lesssim& \int_{z\in \Bbb R}\int_{r>0}\frac{\la z\ra^{1-\frac{3}{2}\delta_*}r^3drdz}{\la z\ra^{1-\delta_*}r^2\la r\ra^{2\alpha-10}}\int_{u>0}(\pa_r\Psi^2)\la u\ra^{2\alpha-10}\la z\ra^{1-\delta_*}u^3du\\
& \lesssim & \int_{z\in \Bbb R}\int_{u>0}(\pa_r\Psi^2)\la u\ra^{2\alpha-10}\la z\ra^{1-\delta_*}u^3dudz\lesssim \|f\|_{\alpha}^2.
\eee
Then the four dimensional Hardy inequality ensures:
\bee
\left|\int_{\Bbb R^5}fm_A(z)\Psi dx\right|&\lesssim &\left(\int_{\Bbb R^5}r^2\la z\ra^{3-\frac{3}{2}\delta_*}f^2dx\right)^{\frac 12}\left(\int_{\Bbb R^5}m_A(z)\frac{\Psi^2}{r^2}dx\right)^{\frac 12}\\
&\le & C_\e\|f\|_{\alpha}^2+\e\int_{\Bbb R^5}m_A(z)(\pa_r\Psi)^2dx.\eee The collection of above bounds yields:
$$\int_{\Bbb R^5}m_{A}(z)\left[(\pa_r\Psi)^2+a^2(\pa_z\Psi)^2\right]dx\lesssim \|f\|_{\alpha}^2,$$ and letting $A\to +\infty$ and using Fatou's lemma yields \eqref{vneivneeionenenvoe}.\\

\noindent{\bf step 4} Proof of \eqref{imporvedlrotwbound}. We now rerun step 2 with the even multiplier $$m_0(z)=\left|\begin{array}{l}1\ \ \mbox{for}\ \ z\le \frac 12\\ z^{3-2\delta_*}\ \ \mbox{for}\ \ z\ge 1
\end{array}\right.,
$$
and for $B\ge 1$ a radial multiplier: $$\chi_B(r)=\left|\begin{array}{l} 1\ \ \mbox{for}\ \ r\le \frac 12\\ r^{2\alpha-12}\ \ \mbox{for}\ \ 1\le r\le B\\  (2B)^{\alpha-12}\ \ \ \ \mbox{for}\ \ r\ge 2B
\end{array}\right., \ \ \chi'\ge 0.$$ Let again $m_B(r,z)=m_0(z)\chi_B(r),$ we estimate all the terms in \eqref{vneoivnevenvenveonenv}.\\
 \noindent\underline{Source term}
\bee
\left|\int_{\Bbb R^5}fm_B\Psi dx\right|&\lesssim& \left(\int_{\Bbb R^5}\la r\ra^{2\alpha-12}\la r\ra^2\la z\ra^{3-2\delta_*}f^2dx\right)^{\frac 12}\left(\int_{\Bbb R^5}m_B\frac{\Psi^2}{\la r\ra^2}dx\right)^{\frac 12}\\
&\lesssim& \|f\|_\alpha^2+\int_{\Bbb R^5}m_B\frac{\Psi^2}{\la r\ra^2}dx,
\eee
again the second term is treated below with the localization terms.\\

 \noindent\underline{Localization terms}. 
 We estimate for $r\ge 1$:
 $$\frac{1}{m_B}\left(|\pa_r^2m_{B}|+\frac{|\pa_rm_{B}|}r\right)\leq \frac{C\alpha^2}{r^2},$$ and hence for $R\ge R_*$ from \eqref{estitmot} and $\gamma$ close enough to 2: 
 $$m_B W\gg |\pa_r^2m_{B}|+\frac{|\pa_rm_{B}|}r.$$ For $R\le R_*$, we use the radial Sobolev inequality to get
 $$|\Psi(r,z)|\lesssim \frac{1}{r\la z\ra^{\frac 32-\frac 34\delta_*}}\left(\int_{u>0}\la z\ra^{3-\frac{3}{2}\delta_*}(\pa_r\Psi)^2(u,z)u^3du\right)^{\frac 12},$$ with \eqref{vneivneeionenenvoe} to estimate
for $\eta$ small enough:
 \bee
 &&\int_{z\in \Bbb R}\int_{R\le R_*}m_B\Psi^2dx\\
 &\lesssim & \int_{z\in \Bbb R}\int_{r\le \nu}\frac{\la r\ra^{2\alpha-12}\la z\ra^{3-2\delta_*}}{r^2\la z\ra^{3-\frac 32\delta_*}}r^3\left(\int_{u>0}\la z\ra^{3-\frac 32\delta_*}(\pa_r\Psi)^2(u,z)u^3du\right)drdz\\
 &\lesssim &  \int_{z\in \Bbb R}\int_{u>0}\la z\ra^{3-\frac 32\delta_*}(\pa_r\Psi)^2(u,z)u^3du\lesssim \|f\|_{\alpha}^2.
 \eee
We now estimate $|\pa_z^2m_B|\lesssim \la z\ra^{1-2\delta_*}\la r\ra^{2\alpha-10}$ and integrate by parts in $r$ to evaluate using \eqref{fheiofehehigeog}:
\bee
&&\int_{\Bbb R^5}a^2\la r\ra^{2\alpha-12}|\pa_z^2m_B|\Psi^2\lesssim\int_{z\in \Bbb R}\int_{r>0}  \la z\ra^{1-2\delta_*}\la r\ra^{2\alpha-12}\Psi^2r^3dr\\
&\lesssim& \int_{z\in \Bbb R}\int_{r>0}\la z\ra^{1-2\delta_*} \la r\ra^{2\alpha-10}(\pa_r\Psi)^2r^3dr\lesssim \|f\|_{\alpha}^2.
\eee

 \noindent\underline{$D_f$ terms}. Recall \eqref{estiaofnrdiune} which yields:
 $$r|D_1|\lesssim \frac{1}{\nu^2\la R\ra^{|\l_-|}}\left(\int \frac{\Psi^2}{r^2}r^3dr\right)^{\frac 12}\lesssim \frac{1}{\la R\ra^{|\l_-|}\la z\ra^{\frac 32-\frac{3}{4}\delta_*}}\left(\int_{r>0}\la z\ra^{3-\frac 32\delta_*}(\pa_r\Psi)^2r^3dr\right)^{\frac 12},$$ and hence using \eqref{vneivneeionenenvoe}:
 \bee
&& \left|\int_{\Bbb R^5}D_1m_B\Psi dx\right|\lesssim\left(\int_{\Bbb R^5}\la z\ra^{3-2\delta_*}|D_1|^2r^2\la r\ra^{4\alpha-28}dx\right)^{\frac 12}\left(\int_{\Bbb R^5}\la z\ra^{3-\frac 32\delta_*}\frac{\Psi^2}{r^2}dx\right)^{\frac 12}\\
 & \lesssim & \|f\|_\alpha\left[\int_{z>0}\frac{\la z\ra^{3-2\delta_*}dz}{\la z\ra^{3-\frac32\delta_*}}\int_{u>0}\la z\ra^{3-\frac 32\delta_*}(\pa_r\Psi)^2(u,z)u^3du\left(\int_{r>0}\frac{\la r\ra^{4\alpha}r^3dr}{\la R\ra^{|\l_-|}}\right)\right]^{\frac 12}\\
 & \lesssim &  \|f\|_\alpha^2,
 \eee
 for $\eta$ small enough and $\alpha\ll |\lambda_-|$. Next recall \eqref{vdjkvbndkndlknvdlnv}
\bee
    |D_2(r,z)|&\lesssim &  \frac{1}{\nu^3\la R\ra^{\gamma-2+|\l_-|}}\left(\int_{r>0}f^2r^2r^3dr\right)^{\frac12}
\\
&\lesssim & \frac{{\bf 1}_{R\le 1}}{\la R\ra^{|\l_-|}\la z\ra^{\frac 32-\frac 12\delta_*}}\left(\int_{r>0}\la z\ra^{3-\delta_*}f^2r^2r^3dr\right)^{\frac12}\eee
which yields using \eqref{vneivneeionenenvoe}:
\bee
&& \left|\int_{\Bbb R^5}D_2m_B\Psi dx\right|\lesssim\left(\int_{\Bbb R^5}|D_2|^2\la z\ra^{3-2\delta_*}r^2\la r\ra^{4\alpha-28}dx\right)^{\frac 12}\left(\int_{\Bbb R^5}\la z\ra^{3-2\delta_*}\frac{\Psi^2}{r^2}dx\right)^{\frac 12}\\
& \lesssim & \|f\|_{\alpha}\left[\int_{z\in \Bbb R}\frac{\la z\ra^{3-2\delta_*}}{\la z\ra^{3-\delta_*}}\int_{u>0}\la z\ra^{3-\delta_*}f^2(u,z)u^5du\left(\int_{r\le \nu}\frac{\la r\ra^{4\alpha}}{\la R\ra^{|\l_-|}}\right)\right]^{\frac 12}\\
& \lesssim & \|f\|_{\alpha}^2,
\eee
again for $\eta$ small enough and $\alpha\ll |\lambda_-|$.\\

\noindent\underline{Conclusion} The collection of above bounds yields
$$\int_{\Bbb R^5}m_{B}\left[(\pa_r\Psi)^2+a^2(\pa_z\Psi)^2\right]dx\lesssim \|f\|_{\alpha}^2,
$$ with a constant independent of $B$, and hence $B\to+\infty$ and Fatou's lemma yield \eqref{imporvedlrotwbound}.
\end{proof}
\subsection{Proof of Proposition \ref{proppointizedfjifw}}

We are now in position to conclude the proof of Proposition \ref{proppointizedfjifw}.\\

\noindent{\bf step 1} Pointwise decay. We claim the pointwise decay :
\be
\label{poitneonninoboundpsi}
|\psi(x)|\leq K \frac{\|f\|_{\alpha}}{a^{\frac 12}}\left[\frac{r^{2}}{\la z\ra^{\frac{1}{2}}\la R\ra^{\frac D2}}+\frac{r^2}{\la r\ra^{\alpha-2}\la z\ra^{\frac 12}}\right].
\ee
We analyse two regions.\\

\noindent\underline{The case of $R\ge K_\nu R_*$}. We use the pointwise bound \eqref{potinwonioen} together with the weighted $L^2$ bound \eqref{imporvedlrotwbound} which imply for $\theta=1$:
$$|\Psi(x)|\lesssim  \frac{\|f\|_{\alpha}}{a^{\frac 12}}\left[\frac{1}{\la r\ra^{\frac 12}\la z\ra^{\frac{1}{2}}R^{\frac D2}}+\frac{1}{\la r\ra^{\alpha-\frac 32}\la z\ra^{\frac 12}}\right],
$$
and hence $$|\psi|=|r^2\Psi|\lesssim  \frac{\|f\|_{\alpha}}{a^{\frac 12}}\left[\frac{\la r\ra^{2}}{\la r\ra^{\frac 12}\la z\ra^{\frac{1}{2}}R^{\frac D2}}+\frac{r^2}{\la r\ra^{\alpha-2}\la z\ra^{\frac 12}}\right].$$
\noindent\underline{The case of $R\le K_\nu R_*$}. If $|z|\le 1$, then $r+|z|\lesssim 1$ and we use the $L^\infty$ bound \eqref{vneovneovneoneonvonbi} which implies in this zone
$$|\psi(x)|\lesssim |\Psi(x)|\lesssim \frac{\|f\|_{\alpha}}{a^{\frac 12}}.$$ For $|z|\ge 1$, we invoke \eqref{cneiovnenoienevn} and \eqref{imporvedlrotwbound} which yield
$$|\psi(x)|\lesssim  |r^2\Psi(x)|\lesssim \frac{\|f\|_{\alpha}}{a^{\frac 12}}\frac{r^2}{\la z\ra^{\frac{1}{2}}}.$$
The collection of the above bounds yields \eqref{poitneonninoboundpsi}.\\

\noindent{\bf step 2} Proof of \eqref{continieou}. From \eqref{imporvedlrotwbound}: $$\|\la r\ra^{\alpha-6}\la z\ra(|\pa_r\Psi|+a|\pa_z\Psi|)\|_{L^2(\Bbb R^5)}\lesssim \|f\|_{\alpha}.$$ From \eqref{poitneonninoboundpsi}, since $\nu\lesssim C_0\la z\ra^\eta$ and provided $D$ has been chosen large enough (i.e $\gamma$ close enough to $2$):
$$a^{\frac 12}|\psi(x)|\lesssim \frac{\|f\|_{\alpha}r^2}{\la z\ra^{\frac{1}{4}}\la r\ra^{\alpha-2}},$$ 
which together with \eqref{imporvedlrotwbound} concludes the proof of \eqref{continieou}  and of Proposition \ref{proppointizedfjifw}.


\section{Constructing $\psi_\nu$}
\label{sectionnonlinone}


This section is devoted to the non linear construction of $\psi_\nu$ for any admissible $\nu$. The main difficulty, as in \cite{MRS20}, is to understand decay in $z$, that is the choice of the $\nu$ function. For $(a,\mathcal C,\mathcal H)=(0,\mathcal C_0,\mathcal H_0)$, $\nu(z)$ is arbitrary. For $0<a\ll1$ and a given pair of deformations $(\mathcal C, \mathcal H)$, the bifurcation equation for $\nu$ is obtained by projecting the flow onto the scaling instability, or equivalently ensuring the vanishing of the additional $D_f$ term. The bifurcation equation will be  solved in section \ref{sectionnonlintwo}.


\subsection{Setting up the non linear problem}


Our target is to solve \eqref{vnbeioneinveneoivbis}:
 $$-\frac 1r\pa_r\left(\frac{1}{r}\pa_r\psi_{\rm tot}\right)-\frac{a^2}{r^2}\pa_z^2\psi_{\rm tot}=\left(\frac{\mathcal C\matchal C'}{r^2}-\mathcal H'\right)(\psi_{\rm tot}).$$
 
 \noindent{\em Explicit choice of $\mathcal C,\mathcal H$ deformations}. We recall \eqref{vneoinonnrinrbno} and specify explicitly\footnote{We make this choice of explicit power non-linearities for $\R_1$ and $\R_2$ for the sake of simplicity and clairity of the exposition, but obviously any small enough deformation of this choice would lead to the same conclustion. This is really a map $(a,\mathcal C,\mathcal H)\to \psi_{\rm tot}$.} 
\be
\label{vneobnveineoeonoenioe}
\left|\begin{array}{l}
\mathcal C\matchal C'(\psi_{\rm tot})=\mathcal C_0\matchal C'_0(\psi_{\rm tot})+a^2\R_1(\psi_{\rm tot})\\
\R_1(\psi_{\rm tot})=\psi_{\rm tot}^{\frac{\gamma}{\gamma-2}+m1}
\end{array}\right.\Rightarrow \mathcal C\matchal C'(\psi)=A(\gamma)\psi_{\rm tot}^{\frac{\gamma}{\gamma-2}}\left(1+\frac{a^2}{A(\gamma)}\psi_{\rm tot}^{m_1}\right),
\ee
and
\be
\label{vneobnveineoeonoenioebis}
\left|\begin{array}{l}\H'(\psi_{\rm tot})=\H_0'(\psi_{\rm tot})+a^2\R_2(\psi_{\rm tot})\\
\R_2(\psi_{\rm tot})=\psi_{\rm tot}^{\frac{\gamma+2}{\gamma-2}+m_2}
\end{array}\right.
\Rightarrow \H'=\psi_{\rm tot}^{\frac{\gamma+2}{\gamma-2}}\left(1+a^2\psi_{\rm tot}^{m_2}\right).
\ee

\noi{\em Linearized flow}. We define $$\left|\begin{array}{l}
\psi_{\rm tot}=\psi_\nu+a^2\psi\\
\psi_\nu(r,z)=\frac{1}{\nu^{\gamma-2}(z)}\psi_*\left(\frac{r}{\nu(z)}\right)
\end{array}\right..
$$
This yields after division by $a^2$ the non linear equation
\be
\label{vneioneioneinveonvie}
\left|\begin{array}{l}
L_a\psi=F(\nu,\psi)\\
F(\nu,\psi)=\frac{1}{r^2}\pa_z^2\psi_\nu+{\rm NL}(\nu,\psi)
\end{array}\right.,
\ee with 
$${\rm NL}(\nu,\psi)=\left[\frac{\R_1(\psi_\nu+a^2\psi)}{r^2}-\R_2(\psi_\nu+a^2\psi)\right]+\frac{\mathcal F_1(\nu,\psi)-\mathcal F_2(\nu,\psi)}{a^2},
$$
and $$\left|\begin{array}{l}
\mathcal F_1(\nu,\psi)=\frac{\C_0\C'_0(\psi_\nu+a^2\psi)-\C_0\C'_0(\psi_\nu)-a^2\left[(\C'_0)^2+\C_0\C_0''\right](\psi_\nu)\psi}{r^2}\\
\mathcal F_2(\nu,\psi)=\H'_0(\psi_\nu+a^2\psi)-\H'_0(\psi_\nu)-a^2\H''(\psi_\nu)\psi
\end{array}\right..
$$
\subsection{Lipschitz regularity for the non linear terms} We study the Lipschitz regularity of the non linear term ${\rm NL}(\phi)$ in $\|\cdot\|_{E_\alpha}$ and $\|\cdot\|_\alpha$ norms.
\begin{lemma}[Lipschitz regularity for the second order non linear terms]
\label{lemmanonlin}
Consider $\gamma=2+\frac{1}{m}$, $m\in \mathbb{N}$ where $m\geq m^*\gg1$ where $m^*$ is given by Proposition \ref{proppointizedfjifw} as well as $\alpha(\gamma),a^*$ and $\eta^*(\gamma,\delta_*)$. Then for $1\ll \alpha\leq \alpha(\gamma)$, $0<a<a^*$, $0<\eta<\eta^*$, $\nu_1,\nu_2\in \mathcal N^{\eta}$ and $$\psi_1,\psi_2\in B_M=\{\Psi\in E_{\alpha},\ \  \|\Psi\|_{E_{\alpha}}\leq M\},$$ the following bound holds
\bea
\label{firsingengeo}
\nonumber &&\left\|\la z\ra^{\frac{1}{20}}\frac{\mathcal F_1(\nu_2,\psi_2)-\mathcal F_1(\nu_1,\psi_1)}{a^2}\right\|_{\alpha+6}+\left\|\la z\ra^{\frac{1}{20}}\frac{\mathcal F_2(\nu_2,\psi_2)-\mathcal F_2(\nu_1,\psi_1)}{a^2}\right\|_{\alpha+6}\\
&\lesssim &  aM\left(\|\Psi_2-\Psi_1\|_{E_{\alpha}}+M\left\|\frac{\nu_2-\nu_1}{\la z\ra^{\frac{1}{6}}}\right\|_{L^\infty}\right).
\eea
\end{lemma}
\begin{remark}
    The extra decay quantified by $\la z \ra^{\frac{1}{20}}$ is crucial to close the bifurcation equation in the next section.
\end{remark}

\begin{proof}[Proof of Lemma \ref{lemmanonlin}] We estimate in brute force the non linear terms.\\

\noindent{\bf step 1} Uniform control of the tails. First observe using the lower bound $$\Psi_\nu=\frac{1}{\nu^{\gamma}}\Psi_*(R)\gtrsim \frac{1}{\nu^{\gamma}(1+R^{\gamma})}=\frac{1}{\nu^\gamma+r^{\gamma}},$$ the pointwise bound for $\Phi\in B_{M}$: 
\be
\label{vnoivneinevnie}
\left|\frac{a^2\psi}{\psi_\nu}\right|=\left|\frac{a^2\Psi}{\Psi_\nu}\right|\lesssim \frac{a^2M}{a^{\frac 12}}\frac{\nu^{\gamma}+r^\gamma}{\la z\ra^{\frac{1}{4}}\la r\ra^{\alpha}}\lesssim a^{\frac 32}M\le \frac 12,
\ee 
for $\eta<\eta^*$ small enough and $\gamma$ close enough to 2. In particular we have the global control
\[
\frac{1}{2} \frac{1}{\nu(z)^{\gamma-2}}\psi_*\left(\frac{r}{\nu(z)}\right)\leq \psi(r,z)\leq \frac{3}{2} \frac{1}{\nu(z)^{\gamma-2}}\psi_*\left(\frac{r}{\nu(z)}\right).
\]

\noindent{\bf step 2} Power nonlinearites. Let $$\left|\begin{array}{l}
H(u)=u^\ell, \ \ \ell\gg 1\\
F(\nu,\psi)=H(\psi_\nu+a^2\psi)-H(\psi_\nu)-a^2H'(\psi_{\nu})\psi
\end{array}\right..
$$
and define $f(t)=H(\psi_\nu+a^2t\psi)$, then 
\bee
F(\nu,\psi)&=&f(1)-f(0)-f'(0)=\int_0^1f'(t)dt-f'(0)=-\int_0^1tf''(t)dt\\
& = & -(a^2\psi)^2\int_0^1 H''(\psi_\nu+a^2t\psi)dt,
\eee
and hence
$$\left|\begin{array}{l}
\pa_\nu F(\nu,\psi)=-a^4\psi^2\int_0^1 \pa_\nu \psi_\nu H'''(\psi_\nu+a^2t\psi)dt\\
\pa_\psi F(\nu,\psi)=-a^4\left[2\psi\int_0^1 H''(\psi_\nu+a^2t\psi)dt+\psi^3\int_0^1 a^2tH'''(\psi_\nu+a^2t\psi)dt\right]
\end{array}\right..
$$
We then express
\[F(\nu_2,\psi_2)-F(\nu_1,\psi_1)=F(\nu_2,\psi_2)-F(\nu_2,\psi_1)+F(\nu_2,\psi_1)-F(\nu_1,\psi_1),\]
thus
\bee
F(\nu_2,\psi_2)-F(\nu_1,\psi_1)&= &(\psi_2-\psi_1) \int_0^1\pa_\psi F(\nu_2,\psi_1+h(\psi_2-\psi_1))dh\\ &+ &(\nu_2-\nu_1)\int_0^1\pa_\nu F(\nu_1+h(\nu_2-\nu_1),\psi_1)dh,
\eee
and estimate all terms.\\

\noindent\underline{$\pa_\psi F$ term}. We estimate from \eqref{vnoivneinevnie}:
$$
|\pa_\psi F|\lesssim a^4\left[|\psi||\psi_\nu|^{\ell-2}+a^2|\psi|^3|\psi_\nu|^{\ell-3}\right]\lesssim a^4|\psi||\psi_\nu|^{\ell-2},
$$

\noindent\underline{$\pa_\nu F$ term}. We estimate 
\be
\label{veniovneneoevkdhkhdv}
|\pa_\nu\psi_\nu|=\left|-\frac{1}{\nu^{\gamma-1}}(\gamma-2+R\partial_R)\psi_*(R)\right|\lesssim \frac{\psi_\nu}{\nu}\lesssim \psi_\nu,
\ee
 and hence
$$|\pa_\nu F|\lesssim a^4\psi^2|\pa_\nu \psi_\nu||\psi_\nu|^{\ell-3}\lesssim a^4\psi^2|\psi_\nu|^{\ell-2}.$$

\noindent{\bf step 3} Proof of \eqref{firsingengeo} for $\mathcal F_1$. We apply the estimates of Step 2 with $$\ell= \frac{2(\gamma-1)}{\gamma-2}-1=\frac{\gamma}{\gamma-2},$$ so that
$$\left|\begin{array}{l}
|\pa_\psi F|\lesssim a^4|\psi||\psi_\nu|^{\ell-2}\lesssim a^4|\psi|\left(\frac{1}{\nu^{\gamma-2}\la R\ra^{\gamma-2}}\right)^{\frac{4-\gamma}{\gamma-2}}\lesssim \frac{a^4|\psi|}{\nu^2\la R\ra^2}\lesssim a^4|\psi|\\
|\pa_\nu F|\lesssim a^4\psi^2|\psi_\nu|^{\ell-2}\lesssim \frac{a^4\psi^2}{\nu^2\la R\ra^2}\lesssim a^4\psi^2
\end{array}\right..
$$
Which gives the pointwise bound
\bee
\frac{1}{r^2}|\mathcal F_1(\nu_2,\psi_2)-\matchal F_1(\nu_1,\psi_1)|\lesssim \frac{a^4|\psi_2-\psi_1|(|\psi_2|+|\psi_1|)}{r^2}+\frac{a^4|\nu_2-\nu_1|(|\psi_2|^2+|\psi_1|^2)}{r^2}.
\eee
In the ball $B_M$ we have
\be
\label{venivnoenvevxnvxnv}
\left\|\la z\ra^{\frac{1}{4}}\frac{\la r\ra^{\alpha+2}}{r^2}\psi\right\|_{L^\infty}\lesssim \frac{M}{\sqrt{a}},
\ee 
thus
\bee
&&\frac{\la r\ra^{\alpha+6}}{r^2}|\mathcal F_1(\nu_2,\psi_2)-\mathcal F_1(\nu_1,\psi_1)|\\
&\lesssim& \frac{a^4M}{a^{\frac 12}}\frac{\la r\ra^{\alpha+6}|\psi_1-\psi_2|}{\la z\ra^{\frac{1}{4}}\la r\ra^{\alpha+2}}+\frac{a^4M|\nu_2-\nu_1|\la r\ra^{\alpha+6}(|\psi_2|+|\psi_1|)}{ a^{\frac 12}\la z\ra^{\frac{1}{4}}\la r\ra^{\alpha+2}}\\
&\lesssim& \frac{a^3M}{\la r\ra^{\alpha-4}\la z\ra^{\frac{1}{4}}}\left[a^{\frac 12}\la r\ra^{\alpha}|\psi_1-\psi_2|+a^{\frac 12}|\nu_2-\nu_1|\la r\ra^{\alpha}(|\psi_2|+|\psi_1|)\right],
\eee
which ensures for $\alpha$ large enough:
\bee
& &\left\|\la r\ra^{\alpha+6}\la z\ra^{\frac{1}{20}}|\mathcal F_1(\nu_2,\psi_2)-\mathcal F_1(\nu_1,\psi_1)|\right\|_{L^\infty}\\
&\lesssim & a^3M\Bigg\{a^{\frac 12}\|\la r\ra^{\alpha}(\Psi_2-\Psi_1)\|_{L^\infty}\\
& & + \left\|\frac{\nu_2-\nu_1}{\la z\ra^{\frac{1}{6}}}\right\|_{L^\infty}\left[a^{\frac 12}\|\la z\ra^{\frac 14}\la r\ra^{\alpha}\psi_2\|_{L^\infty}+a^{\frac 12}\|\la z\ra^{\frac 14}\la r\ra^{\alpha}\psi_1\|_{L^\infty}\right]\Bigg\}\\
& \lesssim &  a^3M\left(\|\Psi_2-\Psi_1\|_{E_{\alpha}}+M\left\|\frac{\nu_2-\nu_1}{\la z\ra^{\frac{1}{6}}}\right\|_{L^\infty}\right).
\eee
We now observe from the four dimensional radial Sobolev embedding that 
\bee
|\Psi(r,z)|&=&\left|\int_r^{+\infty}\pa_r\Psi(r,z)dr\right|\lesssim \left(\int_r^{+\infty}r^3\la r\ra^{2\alpha}(\pa_r\Psi)^2dr\right)^{\frac 12}\left(\int_{r}^{+\infty}\frac{dr}{r^3\la r\ra^{2\alpha}}\right)^{\frac 12}\\
&\lesssim &\frac{1}{ r\la r \ra^\alpha}\left(\int_r^{+\infty}r^3\la r\ra^{2\alpha}(\pa_r\Psi)^2dr\right)^{\frac 12},
\eee
 and hence for $\psi\in B_M$:
\be
\label{ekvivnonrirh}
\int_{z\in\Bbb R}\la z\ra^{3-2\delta_*}\psi^2(r,z)dz\lesssim \left(\frac{r}{\la r \ra^\alpha}\right)^2\int_{r>0,z\in \Bbb R}\la r\ra^{2\alpha}\la z \ra^{3-2\delta_*}(\pa_r\Psi)^2r^3drdz\lesssim \frac{\|\Psi\|_{E_\alpha}^2}{\la r\ra^{2\alpha-2}},
\ee
from which we get
\bee
   & & \la r\ra^{\alpha+6}\left(\int_{z>0}\la z\ra^{3-\delta_*+\frac 1{10}} |\mathcal F_1(\nu_2,\psi_2)-\mathcal F_1(\nu_1,\psi_1)|^2dz\right)^{\frac 12}\\  
&\lesssim & \left\{\int_{z>0}\left[\frac{a^3M}{\la r\ra^{\alpha-4}\la z\ra^{\frac{1}{4}}}aux_1\right]^2\la z\ra^{3-\delta_*+\frac 1{10}} dz\right\}^{\frac 12}\lesssim   a^{3+\frac 12}M\left\{\int_{z>0}aux_2\la z\ra^{3-2\delta_*} dz\right\}^{\frac 12}\\
& \lesssim & a^3M\left(\|\Psi_2-\Psi_1\|_{E_{\alpha}}+M\left\|\frac{\nu_2-\nu_1}{\la z\ra^{\frac{1}{6}}}\right\|_{L^\infty}\right),
\eee
with 
\[
aux_1=a^{\frac 12}\la r\ra^{\alpha}|\psi_1-\psi_2|+a^{\frac 12}|\nu_2-\nu_1|\la r\ra^{\alpha}(|\psi_2|+|\psi_1|),
\]
and 
\[
aux_2=\la r\ra^{2(\alpha-1)}|\Psi_1-\Psi_2|^2+\left(\frac{|\nu_2-\nu_1|}{|z|^{\frac 16}}\right)^2\la r\ra^{2(\alpha-1)}(|\psi_2|^2+|\psi_1|^2).
\]
The above estimates imply \eqref{firsingengeo} for $\mathcal F_1$ after division by $a^2$.\\

\noindent{\bf step 4} Proof of \eqref{firsingengeo} for $\mathcal F_2$. We apply the estimates of Step 2 with $$\ell=\frac{\gamma+2}{\gamma-2},$$ 
so that 
$$\left|\begin{array}{l}
|\pa_\psi F|\lesssim a^4|\psi||\psi_\nu|^{\ell-2}\lesssim a^4|\psi|\left(\frac{1}{\nu^{\gamma-2}\la R\ra^{\gamma-2}}\right)^{\frac{6-\gamma}{\gamma-2}}\lesssim \frac{a^4|\psi|}{\nu^4\la R\ra^4}\lesssim a^4|\psi|\\
|\pa_\nu F|\lesssim a^4\psi^2|\psi_\nu|^{\ell-2}\lesssim \frac{a^4\psi^2}{\nu^4\la R\ra^4}\lesssim a^4\psi^2
\end{array}\right.,
$$
Hence the pointwise bound
\bee
|\mathcal F_2(\nu_2,\psi_2)-\mathcal F_2(\nu_1,\psi_1)|\lesssim a^4|\psi_2-\psi_1|(|\psi_2|+|\psi_1|)+a^4|\nu_2-\nu_1|(|\psi_2|^2+|\psi_1|^2),
\eee
which using \eqref{venivnoenvevxnvxnv} yields:
\bee
&&\la r\ra^{\alpha+6}\left|\mathcal F_2(\nu_2,\psi_2)-\mathcal F_2(\nu_1,\psi_1)\right|\\
&\lesssim & \frac{a^4M}{a^{\frac 12}}\frac{\la r\ra^{\alpha+6}|\psi_1-\psi_2|}{\la z\ra^{\frac{1}{4}}\la r\ra^{\alpha}}+\frac{a^4M|\nu_2-\nu_1|\la r\ra^{\alpha+6}(|\psi_2|+|\psi_1|)}{ a^{\frac 12}\la z\ra^{\frac{1}{4}}\la r\ra^{\alpha}}\\
&\lesssim& \frac{a^3M}{\la r\ra^{\alpha-6}\la z\ra^{\frac{1}{4}}}\left[a^{\frac 12}\la r\ra^{\alpha}|\psi_1-\psi_2|+a^{\frac 12}|\nu_2-\nu_1|\la r\ra^{\alpha}(|\psi_2|+|\psi_1|)\right],
\eee
which ensures for $\alpha$ large enough:
\bee
& & \left\|\la r\ra^{\alpha+6}\la z\ra^{\frac{1}{20}}|\mathcal F_2(\nu_2,\psi_2)-\mathcal F_2(\nu_1,\psi_1)|\right\|_{L^\infty}\\
&\lesssim & a^3M\Bigg\{a^{\frac 12}\|\la r\ra^{\alpha}(\psi_2-\psi_1)\|_{L^\infty}+\left\|\frac{\nu_2-\nu_1}{\la z\ra^{\frac{1}{6}}}\right\|_{L^\infty}\left[a^{\frac 12}\|\la z\ra^{\frac 14}\la r\ra^{\alpha}\psi_2\|_{L^\infty}+a^{\frac 12}\|\la z\ra^{\frac 14}\la r\ra^{\alpha}\psi_1\|_{L^\infty}\right]\Bigg\}\\
& \lesssim & a^3M\left(\|\Psi_2-\Psi_1\|_{E_{\alpha}}+M\left\|\frac{\nu_2-\nu_1}{\la z\ra^{\frac{1}{6}}}\right\|_{L^\infty}\right),
\eee
now using \eqref{ekvivnonrirh} we get:\bee
& & \la r\ra^{\alpha+6}\left(\int_{z>0}\la z\ra^{3-\delta_*+\frac1{10}} |\mathcal F_2(\nu_2,\psi_2)-\mathcal F_2(\nu_1,\psi_1)|^2dz\right)^{\frac 12}\\
&\lesssim & \left[\int_{z>0}\left(\frac{a^3M}{\la r\ra^{\alpha-6}\la z\ra^{\frac{1}{4}}}aux_3\right)^2\la z\ra^{3-\delta_*+\frac1{10}} dz\right]^{\frac 12} \lesssim  a^{3+\frac 12}M\left[\int_{z>0}aux_4\la z\ra^{3-2\delta_*} dz\right]^{\frac 12}\\
& \lesssim &  a^3M\left(\|\Psi_2-\Psi_1\|_{E_{\alpha}}+M\left\|\frac{\nu_2-\nu_1}{\la z\ra^{\frac{1}{6}}}\right\|_{L^\infty}\right),
\eee
with 
\[
aux_3=a^{\frac 12}\la r\ra^{\alpha}|\psi_1-\psi_2|+a^{\frac 12}|\nu_2-\nu_1|\la r\ra^{\alpha}(|\psi_2|+|\psi_1|),
\]
and 
\[
aux_4=\la r\ra^{2(\alpha-1)}|\Psi_1-\Psi_2|^2+\left(\frac{|\nu_2-\nu_1|}{|z|^{\frac 12}}\right)^2\la r\ra^{2(\alpha-1)}(|\psi_2|^2+|\psi_1|^2),
\]
thus \eqref{firsingengeo} is proved for $\mathcal F_2$.
\end{proof}

We now turn to the control of $\mathcal R_1,\mathcal R_2$ corrections.

\begin{lemma}[Lipschitz regularity for non linear corrections]
\label{lemmanonlinbis}
Under the assumptions of Lemma \ref{lemmanonlin}, assume moreover 
\be
\label{limitiationpariamters}
\left|\begin{array}{l} \frac{\gamma}{\gamma-2}+m_1>\frac{\gamma+2}{\gamma-2}+m_2\\
\eta(\gamma-2)\left(\frac{\gamma+2}{\gamma-2}+m_2\right)\ge 2
\end{array}\right..
\ee
then we have the bounds
\be
\label{eninnone}
\left\|\frac{\matchal R_1(\psi_{\nu_i}+a^2\psi_{i})}{r^2}\right\|_{\alpha+6}+\|\matchal R_2(\psi_{\nu_i}+a^2\psi_{i}))\|_{\alpha+6}\lesssim 1, \ \ i\in \{1,2\}
\ee
and
\bea
\label{firsingengebis}
\nonumber &\left\|\frac{\mathcal R_1(\psi_{\nu_2}+a^2\psi_{2})-\mathcal R_1(\psi_{\nu_1}+a^2\psi_{1})}{r^2}\right\|_{\alpha+6}+\left\|\mathcal R_2(\psi_{\nu_2}+a^2\psi_{2})-\mathcal R_2(\psi_{\nu_1}+a^2\psi_{1})\right\|_{\alpha+6}\\
&\lesssim   a^2\|\Psi_2-\Psi_1\|_{E_{\alpha}}+\left\|\frac{\nu_2-\nu_1}{\la z\ra^{\frac{\delta_*}{6}}}\right\|_{L^\infty}.
\eea
\end{lemma}

\begin{proof}[Proof of Lemma \ref{lemmanonlinbis}] In this proof $\psi$ will denote either $\psi_1$ or $\psi_2$ with $\psi_i\in B_M$.\\

\noindent{\bf step 1} $\mathcal R_1$.
Let $$
\left|\begin{array}{l}
F(\nu,\psi)=(\psi_\nu+a^2\psi)^\ell\\
\ell=\frac{\gamma}{\gamma-2}
\end{array}\right.,
$$
then using \eqref{veniovneneoevkdhkhdv} and \eqref{vnoivneinevnie} we have:
$$\left|\begin{array}{l}
|F(\nu,\psi)|\lesssim \psi_\nu^\ell\\
|\pa_\nu F(\nu,\psi)|=\left|\ell \pa_\nu \psi_\nu (\psi_\nu+a^2\psi)^{\ell-1}\right|\lesssim \psi_\nu^{\ell}\\
|\pa_\psi F(\nu,\psi)|=\left|\ell a^2\psi(\psi_\nu+a^2\psi)^{\ell-1}\right|\lesssim a^2|\psi|\psi_\nu^{\ell-1}\lesssim a^2|\psi_\nu^{\ell}|
\end{array}\right.,
$$
for $r\ge 1$ using \eqref{limitiationpariamters} we get:
\bee
|\psi_\nu^{\ell}|& = &\left|\frac{1}{\nu^{\gamma-2}}\psi_*(R)\right|^{\frac{\gamma}{\gamma-2}+m_1}\lesssim \left|\frac{1}{\nu^{\gamma-2}\la R\ra^{\gamma-2}}\right|^{\frac{\gamma}{\gamma-2}+m_1}\\
& \lesssim & \frac{1}{\nu^{(\gamma-2)(\frac{\gamma}{\gamma-2}+m_1)}+r^{(\gamma-2)(\frac{\gamma}{\gamma-2}+m_1)}}\lesssim \frac{1}{\la z\ra^{2}+\la r\ra^{\frac{2}{\eta}}},
\eee
and for $r\le 1$ using $R\lesssim r$ gives:
$$|\psi_\nu^{\ell}|=\left|\frac{1}{\nu^{\gamma-2}}\psi_*(R)\right|^{\frac{\gamma}{\gamma-2}+m_1}\lesssim\left|\frac{r^2}{\nu^{\gamma-2}}\right|^{\frac{\gamma}{\gamma-2}+m_1}\lesssim \frac{r^2}{\la z\ra^{2}},$$
and hence $$\frac{|\psi_\nu|^\ell}{r^2}\lesssim \frac{1}{\la z\ra^{2}+\la r\ra^{\frac{2}{\eta}}}.$$
\noindent\underline{Estimate for $\mathcal R_1$}. This yields
$$\frac{|\mathcal R_1(\psi_\nu+a^2\psi)|}{r^2}\lesssim |\psi_\nu^{\ell}|\lesssim \frac{1}{\la z\ra^{2}+\la r\ra^{\frac{2}{\eta}}},$$
from which we get
$$\|\la r\ra^{\alpha+6}\mathcal R_1\|_{L^\infty}\lesssim 1,$$ and
\bee
&&\la r\ra^{2\alpha+12}\int_{z>0}\la z\ra^{3-\delta_*}\left(\frac{\mathcal R_1}{r^2}\right)^2dz\lesssim \la r\ra^{2\alpha+12}\int_{z>0}\frac{\la z\ra^{3-\delta_*} dz}{\left(\la z\ra^{2}+\la r\ra^{\frac{2}{\eta}}\right)^2}\\
&\lesssim&  \la r\ra^{2\alpha+12}\int_{z>0}\frac{\la z\ra^{3-\delta_*} dz}{\la z\ra^4+\la r\ra^{\frac{4}{\eta}}} \lesssim 1,
\eee
for $\eta$ small enough and \eqref{eninnone} is proved for $\matchal R_1$.\\

\noindent\underline{Estimate for the difference}. We estimate:
\bee
&&\left|\frac{\mathcal R_1(\psi_{\nu_2}+a^2\psi_{\nu_2})-\matchal R_1(\psi_{\nu_1}+a^2\psi_{\nu_1})}{r^2}\right|\\
&=&\frac{1}{r^2}\bigg|(\psi_2-\psi_1) \int_0^1\pa_\psi F(\nu_2,\psi_1+h(\psi_2-\psi_1))dh\\
&&+(\nu_2-\nu_1)\int_0^1\pa_\nu F(\nu_1+h(\nu_2-\nu_1),\psi_1)dh\bigg|\\
&\lesssim & \frac{|\psi_\nu|^\ell}{r^2}[a^2|\psi_2-\psi_1|+|\nu_2-\nu_1|]\lesssim\frac{a^2|\psi_2-\psi_1|+|\nu_2-\nu_1|}{\la z\ra^{2}+\la r\ra^{\frac{2}{\eta}}}.
\eee
Hence we have
\bee
   & & \|\la r\ra^{\alpha+6}\left[\mathcal R_1(\psi_{\nu_2}+a^2\psi_{2})-\matchal R_1(\psi_{\nu_1}+a^2\psi_{1})\right]\|_{L^\infty}\\ & & \lesssim a^2\|\la r\ra^\alpha(\Psi_2-\Psi_1)\|_{L^\infty}+\left\|\frac{\nu_2-\nu_1}{\la z\ra^{\frac 12}}\right\|_{L^\infty}
\eee
and recalling \eqref{ekvivnonrirh} we get:
\bee
&&\la r\ra^{2\alpha+12}\int_{z\in \Bbb R}\la z\ra^{3-\delta_*}\left|\mathcal R_1(\psi_{\nu_2}+a^2\psi_{2})-\matchal R_1(\psi_{\nu_1}+a^2\psi_{1})\right|^2dz\\
&\lesssim & \la r\ra^{2\alpha+12}\int_{z\in \Bbb R}\la z\ra^{3-\delta_*}\frac{a^4|\psi_2-\psi_1|^2+|\nu_2-\nu_1|^2}{\la z\ra^{4}+\la r\ra^{\frac{4}{\eta}}}dz\\
& \lesssim & a^4\int_{z>0}\la z\ra^{3-2\delta_*}\la r\ra^{2(\alpha-1)}|\Psi_2-\Psi_1|^2dz+\la r\ra^{2\alpha+10}\left\|\frac{\nu_2-\nu_1}{\la z\ra^{\frac{\delta_*}{6}}}\right\|^2_{L^\infty}\int_{z>0}\frac{\la z\ra^{3-\delta_*+\frac{\delta_*}{3}} dz}{\la z\ra^{4}+\la r\ra^{\frac{2}{\eta}}}\\
& \lesssim & a^4\|\Psi_2-\Psi_1\|_{E_\alpha}^2+\left\|\frac{\nu_2-\nu_1}{\la z\ra^{\frac{\delta_*}{6}}}\right\|^2_{L^\infty}.
\eee
This concludes the proof of \eqref{firsingengebis} for $\mathcal R_1$.\\

\noindent{\bf step 2} $\mathcal R_2$.
Let $$
\left|\begin{array}{l}
F(\nu,\psi)=(\psi_\nu+a^2\psi)^\ell\\
\ell=\frac{\gamma+2}{\gamma-2}+m_2
\end{array}\right.
$$
then using \eqref{veniovneneoevkdhkhdv} and \eqref{vnoivneinevnie} we get:
$$\left|\begin{array}{l}
|F(\nu,\psi)|\lesssim \psi_\nu^\ell\\
|\pa_\nu F(\nu,\psi)|=\left|\ell \pa_\nu \psi_\nu (\psi_\nu+a^2\psi)^{\ell-1}\right|\lesssim \psi_\nu^{\ell}\\
|\pa_\psi F(\nu,\psi)|=\left|\ell a^2\psi(\psi_\nu+a^2\psi)^{\ell-1}\right|\lesssim a^2|\psi|\psi_\nu^{\ell-1}\lesssim a^2|\psi_\nu^{\ell}|
\end{array}\right.
$$
and
\bee
& &|\psi_\nu^{\ell}|=\left|\frac{1}{\nu^{\gamma-2}}\psi_*(R)\right|^{\frac{\gamma+2}{\gamma-2}+m_2}\lesssim \left|\frac{1}{\nu^{\gamma-2}\la R\ra^{\gamma-2}}\right|^{\frac{\gamma+2}{\gamma-2}+m_2}\\
& & \lesssim \frac{1}{\nu^{(\gamma-2)(\frac{\gamma+2}{\gamma-2}+m_2)}+r^{(\gamma-2)(\frac{\gamma+2}{\gamma-2}+m_2)}}\lesssim \frac{1}{\la z\ra^{2}+\la r\ra^{\frac{2}{\eta}}}.
\eee
\noindent\underline{Estimate for $\mathcal R_2$}. This yields
$$\frac{|\mathcal R_2(\psi_\nu+a^2\psi_\nu)|}{r^2}\lesssim |\psi_\nu^{\ell}|\lesssim \frac{1}{\la z\ra^{2}+\la r\ra^{\frac{2}{\eta}}},$$
and we argue in verbatim like for $\mathcal R_1$ to derive \eqref{eninnone} for $\matchal R_2$.

\noindent\underline{Estimate for the difference}. We estimate:
\bee
&&\left|\mathcal R_2(\psi_{\nu_2}+a^2\psi_{\nu_2})-\matchal R_2(\psi_{\nu_1}+a^2\psi_{\nu_1})\right|\\
&=&\Bigg|(\psi_2-\psi_1) \int_0^1\pa_\psi F(\nu_2,\psi_1+h(\psi_2-\psi_1))dh\\
&&+(\nu_2-\nu_1)\int_0^1\pa_\nu F(\nu_1+h(\nu_2-\nu_1),\psi_1)dh\Bigg|\\
&\lesssim & |\psi_\nu|^\ell[a^2|\psi_2-\psi_1|+|\nu_2-\nu_1|]\lesssim\frac{a^2|\psi_2-\psi_1|+|\nu_2-\nu_1|}{\la z\ra^{2}+\la r\ra^{\frac{2}{\eta}}},
\eee
and we argue in verbatim like for $\mathcal R_1$ to derive \eqref{firsingengebis} for $\matchal R_2$.
\end{proof}

\subsection{Existence and regularity of the first non linear resolvent map}
We first solve the generalized version of \eqref{vneioneioneinveonvie} 
\be
\label{nvioenveoneiomdvlmd}
\left|\begin{array}{l}\psi=r^2\Psi\\
L_a\psi=F(\psi)+D_{F(\psi)}h\Leftrightarrow\mathcal L_a\Psi=F(\psi)\\
F(\psi)=\frac{1}{r^2}\pa_z^2\psi_\nu+{\rm NL}(\psi),
\end{array}\right.
\ee
 using a contraction mapping argument.  

\begin{lemma}[Inversion \`a la Picard]
\label{lemmainversion}
Under the assumptions of Lemma \ref{lemmanonlinbis}, and provided $M$ is a large enough universal constant, then for all $0<a<a^*$ small enough, the mapping $$\Psi \mapsto \mathcal L_a^{-1}F(\psi)$$ is a contraction mapping on $B_{M}$. Let then $\psi_{{\rm sol},\nu}$ be the unique associated fixed point in $B_M$, then for $\nu_1,\nu_2\in \mathcal N^\eta$:
\be
\label{vneovneionienvoei}
\|\Psi_{{\rm sol},\nu_1}-\Psi_{{\rm sol,\nu_2}}\|_{E_\alpha}\lesssim \left\|\frac{\nu_2-\nu_1}{\la z\ra^{\frac{\delta_*}{6}}}\right\|_{L^\infty}+\left\|\pa_z^2(\Psi_{\nu_2}-\Psi_{\nu_1})\right\|_{\alpha+6}.
\ee
\end{lemma}

\begin{remark}
\label{remarkregularity} Let us stress that smoothness in $(r,z)$ of the constructed fixed point $\psi_{{\rm sol},\nu}$ follows from standard elliptic regularity estimates.
\end{remark}

\begin{proof}[Proof of Lemma \ref{lemmainversion}] This is a simple consequence of Proposition \ref{proppointizedfjifw}, Lemma \ref{lemmanonlin} and Lemma \ref{lemmanonlinbis}.\\

\noindent{\bf step 1} Inhomogeneous term. We compute:
$$
\left|\begin{array}{l}
\pa_z\Psi_\nu=-\frac{\pa_z\nu}{\nu}\frac{1}{\nu^{\gamma}}(\gamma+R\partial_R)\Psi_*(R)\\
\pa^2_z\Psi_\nu=\frac{1}{\nu^{\gamma}}\left[-\pa_z\left(\frac{\pa_z\nu}{\nu}\right)(\gamma+R\partial_R)+\left(\frac{\pa_z\nu}{\nu}\right)^2\left(\gamma+R\partial_R\right)^2\right]\Psi_*(R)
\end{array}\right.,
$$
and hence obtain the decay $$\left|\la R\ra^k\frac{d^k}{dR^k}(\gamma-2+\Lamdba_R)\Psi_*\right|\lesssim \frac{1}{\la R\ra^{\gamma+|\l_-|}},$$ which together with $\lim_{\gamma\to _2}|\l_-(\gamma)|= +\infty$ and \eqref{estoatmgnoi} ensure the pointwise bound
$$|\pa^2_z\Psi_\nu|\leq \frac{C(c_0,C_1,C_2)}{\la z\ra^2\la R\ra^{\gamma+|\l_-|}}\Rightarrow \|\pa^2_z\Psi_\nu\|_{\alpha+6}\lesssim C(c_0,C_1,C_2,\gamma),$$ 
for $\gamma$ close enough to 2,
 and hence 
 \be
 \label{vbdjkvbdjkbvdjbkdv}
 \|\pa^2_z\Psi_\nu\|_{\alpha+6}\lesssim 1.
 \ee

\noindent{\bf step 2} Strict contraction. We conclude from \eqref{continieou}, \eqref{vbdjkvbdjkbvdjbkdv}, \eqref{firsingengeo} with $\nu_2=\nu_1$, $\psi_1=\psi$, $\psi_2=0$ and \eqref{eninnone} that
$$\|\mathcal L_a^{-1}F\|_{E_{\alpha}}\lesssim \|F\|_{\alpha+6}\lesssim 1+aM^2+a^2M\lesssim 1 $$ and hence $\mathcal L_a^{-1}F\subset B_M$. We moreover conclude from \eqref{continieou}, \eqref{firsingengeo} and \eqref{firsingengebis} with $\nu_1=\nu_2$:
$$\left\|\mathcal L_a^{-1}\left[F(\nu,\psi_2)-F(\nu,\psi_1)\right]\right\|_{E_{\alpha}}\lesssim \left\|F(\nu,\psi_2)-F(\nu,\psi_1)\right\|_{\alpha+6}\lesssim aM\|\psi_2-\psi_1\|_{E_{\alpha}}$$ and the strict contraction property is proved.\\

\noindent{\bf step 3} Lipschitz regularity in $\nu$. Two fixed point solutions satisfy:
\bee
    \psi_{{\rm sol},\nu_2}-\psi_{{\rm sol},\nu_1}&=&\mathcal L_a^{-1}\Bigg\{ \frac{1}{r^2}\pa_z^2(\psi_{\nu_2}-\psi_{\nu_1})+  \frac{\R_1(\psi_{\nu_2}+a^2\psi_{{\rm sol},\nu_2})-\R_1(\psi_{\nu_1}+a^2\psi_{{\rm sol},\nu_1})}{r^2}\\
&-&\left[\R_2(\psi_{\nu_2}+a^2\psi_{{\rm sol},\nu_2})-\R_2(\psi_{\nu_1}+a^2\psi_{{\rm sol},\nu_1})\right]\\
&+& \frac{\mathcal F_1(\nu_2,\psi_{{\rm sol},\nu_2})-\mathcal F_1(\nu_1,\psi_{{\rm sol},\nu_1})}{a^2}-\frac{\mathcal F_2(\nu_2,\psi_{{\rm sol},\nu_2})-\mathcal F_2(\nu_1,\psi_{{\rm sol},\nu_1})}{a^2}\Bigg\}
\eee
and hence the bound:
\bee
\|\Psi_{{\rm sol},\nu_2}-\Psi_{{\rm sol},\nu_1}\|_{E_\alpha}&\lesssim &\left\|\frac{1}{r^2}\pa_z^2(\psi_{\nu_2}-\psi_{\nu_2})\right\|_{\alpha+6}+ a^2\|\psi_{{\rm sol},\nu_2}-\psi_{{\rm sol},\nu_1}\|_{E_{\alpha}}+\left\|\frac{\nu_2-\nu_1}{\la z\ra^{\frac{\delta_*}{6}}}\right\|_{L^\infty}\\
&+&aM\left(\|\Psi_{{\rm sol},\nu_2}-\Psi_{{\rm sol},\nu_1}\|_{E_{\alpha}}+M\left\|\frac{\nu_2-\nu_1}{\la z\ra^{\frac{\delta_*}{6}}}\right\|_{L^\infty}\right)
\eee
which implies \eqref{vneovneionienvoei} for $0<a<a^*$ small enough.
\end{proof}

\section{Solving the bifurcation equation}
\label{sectionnonlintwo}

In this section, we conclude the proof of Theorem \ref{thmmain} by solving the bifurcation equation for $\nu$ 
\be
\label{cneioneineionvoe}
D_{F(\nu,\Psi_{\rm sol,\nu})}h=0.
\ee
We will systematically use the notation 
\be
\beta_1=\frac{\gamma}{\gamma-2}+m_1>\beta_2=\frac{\gamma+2}{\gamma-2}+m_2, \text{ and }\label{defntuilde}
\nut=\frac{1}{\nu^{\gamma-2}}.
\ee


\subsection{Ode formulation of the bifurcation equation}


\begin{lemma}[Ode formulation]
\label{vneoineoneoinioev} 
Let $$\left|\begin{array}{l}
m_0^*=\int_0^{+\infty}\frac{\zeta_*^2}{R}dR>0\\
m_1^*=(\gamma-2)\int_{R>0}\frac{(\psi_*)^{\beta_1+1}}{R^2}RdR>0\\
m_2^*=\left(\gamma-2-\frac{2}{\beta_2+1}\right)\int_{R>0}(\psi_*)^{\beta_2+1}RdR>0\\
d^*_1=\frac{(\gamma-2)m_1^*}{m_0^*}>0\\
d^*_2=\frac{(\gamma-2)m_2^*}{m_0^*}>0,\\
\end{array}\right.
$$ 
then \eqref{cneioneineionvoe} is 
\be
\label{equaitonoineog}
\left|\begin{array}{l}
-\pa_z^2\nut-d_1^*\nut^{\beta_1}+d_2^*\nut^{\beta_2}=\mathcal G_a(\nu)\\
\mathcal G_a(\nu)=\frac{\gamma-2}{m_0^*}\int_{r>0}\left[\frac{a^2}{r^2}\psi_{{\rm sol},\nu}\frac{\pa_z^2h}{h}+{\rm NL}(\nu,\psi_{{\rm sol},\nu})-{\rm NL}(\nu,0)\right]\zeta_*(R)rdr.
\end{array}\right.
\ee
\end{lemma}

\begin{proof}[Proof of Lemma \ref{vneoineoneoinioev}] Using \eqref{defdfz} and \eqref{vneioneioneinveonvie}, 
\bee
&&D_{F(\nu,\Psi_{\rm sol,\nu})}h=0\\
\nonumber &\Leftrightarrow & \forall z\in \Bbb R, \ \ \int_{r>0} \left\{\frac{a^2}{r^2}\psi_{{\rm sol},\nu}\pa_z^2h+\left[\frac{1}{r^2}\pa_z^2\psi_\nu+{\rm NL}(\nu,\psi)\right]h\right\}rdr=0\\
\nonumber &\Leftrightarrow & \forall z\in \Bbb R, \ \ \int_{r>0} \left[\frac{a^2}{r^2}\psi_{{\rm sol},\nu}\frac{\pa_z^2h}{h}+\frac{1}{r^2}\pa_z^2\psi_\nu+{\rm NL}(\nu,\psi)\right]\zeta_*(R)rdr=0.
\eee
First recall
 $$
 \pa^2_{z}\psi_{\nu}(r,z)=\frac{1}{\nu^{\gamma-2}}\left[-\pa_z\left(\frac{\pa_z\nu}{\nu}\right)(\gamma-2+R\partial_R)+\left(\frac{\pa_z\nu}{\nu}\right)^2(\gamma-2+R\partial_R)^2\right]\psi_*(R),
 $$
 which yields the contribution:
 \bee
& &\int_{r>0}\frac{\pa_z^2\psi_\nu}{r^2}\zeta_*(R)rdr\\
 & = &\frac{1}{\nu^{\gamma-2}}\int_0^{+\infty}\left[-\pa_z\left(\frac{\pa_z\nu}{\nu}\right)\zeta_*(R)+\left(\frac{\pa_z\nu}{\nu}\right)^2(\gamma-2+R\partial_R)\zeta_*(R)\right]\zeta_*(R)\frac{dR}{R}\\
& = & \frac{1}{\nu^{\gamma-2}}\left[-\pa_z\left(\frac{\pa_z\nu}{\nu}\right)+(\gamma-2)\left(\frac{\pa_z\nu}{\nu}\right)^2\right]\int_0^{+\infty}\frac{\zeta_*^2}{R}dR\\
& = & \frac{m_0^*}{\nu^{\gamma-2}}\left[-\pa_z\left(\frac{\pa_z\nu}{\nu}\right)+(\gamma-2)\left(\frac{\pa_z\nu}{\nu}\right)^2\right].
 \eee
  From \eqref{defntuilde}: $$\left|\begin{array}{l} \frac{\pa_z\nut}{\nut}=-(\gamma-2)\frac{\pa_z\nu}{\nu}\\
\frac{\pa_z^2\nut}{\nut}=-(\gamma-2)\pa_z\left(\frac{\pa_z\nu}{\nu}\right)+\left(-(\gamma-2)\frac{\pa_z\nu}{\nu}\right)^2=(\gamma-2)\left[-\pa_z\left(\frac{\pa_z\nu}{\nu}\right)+(\gamma-2)\left(\frac{\pa_z\nu}{\nu}\right)^2\right]
\end{array}\right.,
$$ and hence \eqref{cneioneineionvoe} admits the equivalent formulation:
\[m_0^*\pa^2_z\nut+\int_{r>0}\left[\frac{a^2}{r^2}\psi_{{\rm sol},\nu}\frac{\pa_z^2h}{h}+{\rm NL}(\nu,\psi_{{\rm sol},\nu})\right]\zeta_*(R)rdr=0,\]
which gives
\bee
   &&  m_0^*\pa^2_z\nut+\int_{r>0}{\rm NL}(\nu,0)\zeta_*(R)rdr\\
&= &-\int_{r>0}\left[\frac{a^2}{r^2}\psi_{{\rm sol},\nu}\frac{\pa_z^2h}{h}+{\rm NL}(\nu,\psi_{{\rm sol},\nu})-{\rm NL}(\nu,0)\right]\zeta_*(R)rdr,
\eee
where by a slight abuse of notation ${\rm NL}(\nu,0)={\rm NL}(\nu,\psi_\nu,a=0)$. We recall \eqref{vneioneioneinveonvie} and compute:
\bee
&&\int_{r>0}{\rm NL}(\nu,0)\zeta_*(R)rdr =  \int_{r>0}\left[\frac{\R_1(\psi_\nu)}{r^2}-\R_2(\psi_\nu)\right]\zeta_*(R)rdr\\
&=& \int_{R>0}\left[\frac{\R_1\left(\frac{\psi_*}{\nu^{\gamma-2}}\right)}{\nu^2R^2}-\R_2\left(\frac{\psi_*}{\nu^{\gamma-2}}\right)\right]\zeta_*(R)\nu^2R dR\\
& = &  \int_{R>0}\left[\frac{1}{R^2}\left(\frac{\psi_*}{\nu^{\gamma-2}}\right)^{\beta_1}-\nu^2\left(\frac{\psi_*}{\nu^{\gamma-2}}\right)^{\beta_2}\right]\zeta_*(R)RdR,
\eee
thus
\bee
    \int_{r>0}{\rm NL}(\nu,0)\zeta_*(R)rdr &= & \nut^{\beta_1}\int_{R>0}\frac{(\psi_*)^{\beta_1}}{R^2}\left[(\gamma-2)\psi_*+R\pa_R\psi_*\right]RdR\\ & &-  \nut^{\beta_2}\int_{R>0}(\psi_*)^{\beta_2}\left[(\gamma-2)\psi_*+R\pa_R\psi_*\right]RdR,
\eee
hence
\bee
\int_{r>0}{\rm NL}(\nu,0)\zeta_*(R)rdr &=& \nut^{\beta_1}\left[(\gamma-2)\int_{R>0}\frac{(\psi_*)^{\beta_1+1}}{R^2}RdR\right]\\
& & -\nut^{\beta_2}\left[\left(\gamma-2-\frac{2}{\beta_2+1}\right)\int_{R>0}(\psi_*)^{\beta_2+1}RdR\right],
\eee
finally giving 
\[\int_{r>0}{\rm NL}(\nu,0)\zeta_*(R)rdr =  m_1^*\nut^{\beta_1}-m_2^*\nut^{\beta_2}.\] Observe that $m_1^*,m_2^*$ are finite and strictly positive  from \eqref{limitiationpariamters}. We have therefore obtained the equivalent formulation of \eqref{cneioneineionvoe}:
$$
\frac{m_0^*}{\gamma-2}\pa_z^2\nut+m_1^*\nut^{\beta_1}-m_2^*\nut^{\beta_2}=-\int_{r>0}\left[\frac{a^2}{r^2}\psi_{{\rm sol},\nu}\frac{\pa_z^2h}{h}+{\rm NL}(\nu,\psi_{{\rm sol},\nu})-{\rm NL}(\nu,0)\right]\zeta_*(R)rdr,
$$ which is \eqref{equaitonoineog}.
\end{proof}


\subsection{Approximate solution to \eqref{equaitonoineog}}


We start with neglecting the $a$ dependence of the right hand side of \eqref{equaitonoineog} and solve the approximate bifurcation equation. 

\begin{lemma}[Solving the approximate flow]
\label{lemmaapprsolution}
Assume $\beta_1>\beta_2>1$, then there exists a unique non trivial even solution  $\nut_\infty$ to
\be
\label{eqauivejentformual}
\pa_z^2\nut+d_1^*\nut^{\beta_1}-d_2^*\nut^{\beta_2}=0,
\ee 
which decays as $z\to+\infty$. Moreover we have 
\be
\label{vnenvenenvenv}
\forall z>0, \ \ \nu_\infty'(z)>0,
\ee
and letting 
\be
\label{valueetainfty}
\eta_\infty=\frac{2}{(\beta_2-1)(\gamma-2)},
\ee then $\nu_\infty\in \mathcal N^{\eta_\infty}(c_0,C_0,(C_i)_{1\le i\le2})$ for some universal constant $c_0,C_0,C_1,C_2>0$ depending on $(\gamma,\beta_1,\beta_2)$.
\end{lemma}

\begin{proof}[Proof of Lemma \ref{lemmaapprsolution}] We recall from \eqref{limitiationpariamters} the fundamental constraint 
\be
\label{vnevnnvenenv}
\beta_1>\beta_2.
\ee We restrict our discussion of \eqref{eqauivejentformual} to even solutions and hence $\nut'(0)=0$. We integrate the Hamiltonian system \eqref{eqauivejentformual}:
$$\left|\begin{array}{l}
\frac12(\nut')^2=G(\nut(0))-G(\nut)\\
G(\nut)=-\frac{d^*_2}{\beta_2+2}\nut^{\beta_2+1}+\frac{d^*_1}{\beta_1+1}\nut^{\beta_1+1}
\end{array}\right..
$$
We have $$G'(\nut)=-d^*_2\nut^{\beta_2}+d^*_1\nut^{\beta_1}=d_2\nut^{\beta_2}\left[\frac{d^*_1}{d^*_2}\nut^{\beta_1-\beta_2}-1\right]$$ and hence from \eqref{vnevnnvenenv}, $G'$ vanishes exactly once on $\Bbb R_*^+$ and is strictly negative near the origin, hence $G<0$ near the origin and vanishes exactly once on $\Bbb R^*_+$ with $\lim_{\nut\to +\infty}G(\nut)=+\infty$. Moreover we compue
\be
\label{vneioneioneinveoi}
-\frac{d^*_2}{\beta_2+1}\nut^{\beta_2+1}+\frac{d^*_1}{\beta_1+1}\nut^{\beta_1+1}=0\Leftrightarrow \nut=\nut_*=\left(\frac{d_2^*(\beta_1+1)}{d_1^*(\beta_2+1)}\right)^{\frac{1}{\beta_1-\beta_2}},
\ee
and note that by \eqref{vneoivneionenoeivn}
\[
\partial^2_z\nut(0)=-G'(\nut_*)<0
\]
Thus the trajectory emanating from $(\nut(0)=\nut_*,\nut'(0)=0)$ solves 
\[
\nut'=-\sqrt{-2G(\nut(z))},
\]
and is the unique orbit entering the unstable endpoint $$\lim_{z\to+\infty}(\nut,\nut')=(0,0)$$ and then from the phase portrait:
 \be
\label{neovenovneonvenv}
\forall z>0, \ \ \nut'(z)<0, \ \ \nut(z)>0.
\ee
Let the homogeneous solution:
$$\left|\begin{array}{l}
\nut_\infty''-d^*_2\nut_\infty^{\beta_2}=0\\
\nut_\infty=\frac{c_\infty}{z^{\beta}}
\end{array}\right.\Leftrightarrow \frac{c_\infty\beta(\beta+1)}{z^{\beta+2}}-d^*_2\left(\frac{c_\infty}{z^\beta}\right)^{\beta_2}=0\Rightarrow \left|\begin{array}{l}\beta=\frac{2}{\beta_2-1}\\ c_\infty=\left(\frac{\beta(\beta+1)}{d^*_2}\right)^{\frac{1}{\beta_2-1}},\end{array}\right.
$$
then since $\beta_1>\beta_2$, an elementary fixed point argument near the critical point
 ensures that the leading order near $+\infty$ is given by $\nut_\infty$: $$\nut(z)=\frac{c_\infty(1+o(1))}{z^{\frac{2}{\beta_2-1}}}\ \ \mbox{as}\ \ z\to +\infty.$$ We then easily infer from the equation $$\left|\frac{\la z\ra^k}{\nut}\frac{d^k}{dz^k}\nut\right|\lesssim1, \ \ k=1,2.$$ The collection of above bounds
  together with \eqref{defntuilde}, \eqref{neovenovneonvenv} concludes the proof of \eqref{valueetainfty} and Lemma \ref{lemmaapprsolution}.
\end{proof}

\begin{remark} Let $\nut_\infty(0)=\nu_*$ given by \eqref{vneioneioneinveoi}, then
\bea
\label{vneoivneionenoeivn}
\nonumber \nut''_\infty(0)&=&d^*_2\nut_*^{\beta_2}-d^*_1\nut_*^{\beta_1}=-d_2\nut^{\beta_2}\left[\frac{d^*_1}{d^*_2}\nut_*^{\beta_1-\beta_2}-1\right]\\
&=& -d_2\nut_*^{\beta_2}\left[\frac{\beta_1+1}{\beta_2+1}-1\right]<0.
\eea
\end{remark}

\subsection{Fixed point formulation}
We now reformulate \eqref{equaitonoineog} as a fixed point. \\

\noindent{\em Linearized operator}. Let $$H=-\pa_z^2-\beta_1d_1^*\nut^{\beta_1-1}+\beta_2d_2^*\nut^{\beta_2-1},$$ then by translation invariance of \eqref{eqauivejentformual} we have $$H\nut_\infty'=0.$$ Let $\Gamma$ be the even solution to $$\left|\begin{array}{l}
H\Gamma=0\\
\Gamma(0)=1\\
\Gamma'(0)=0
\end{array}\right.,
$$
then the Wronskian is constant $$W=\left[\Gamma' \nut'_\infty-\Gamma\nut''_\infty\right](x)=\left[\Gamma' \nut_\infty'-\Gamma\nut_\infty''\right](0)=-\nut_\infty''(0)>0,$$ from \eqref{vneoivneionenoeivn} from which since $\nu'$ does not vanish $$\left(\frac{\Gamma}{\nut_\infty'}\right)'=\frac{W}{(\nut_\infty')^2},$$ and hence for some universal constant $z_0>0$:
$$\Gamma(z)=\nut_\infty'(z)\int_{z_0}^z\frac{W}{(\nut_\infty')^2(\tau)}d\tau.$$ In the limit $z\to +\infty$ we have $$\nut_\infty\sim \frac{1}{\la z\ra^{(\gamma-2)\eta_\infty}}$$ 
which yields the asymptotics near $+\infty$: $$\Gamma(z)\sim\frac{c}{z^{(\gamma-2)\eta_\infty+1}}\int_{z_0}^{z}\tau^{2[(\gamma-2)\eta_\infty+1]}d\tau\sim cz^{(\gamma-2)\eta_\infty+2}.$$  The unique even solution to $Hu=f$ decaying at $+\infty$  is therefore given by the formula for $z\ge0$:
\be
\label{eniovneineoneonv}
u(z)=-\Gamma\int_z^{+\infty}\frac{f\nut_\infty'}{W}d\tau-\nut_\infty'\int_0^z\frac{f\Gamma}{W}d\tau\equiv H^{-1}f.
\ee

\noindent{\em Reformulation of \eqref{equaitonoineog}}. We define $$\nut=\nut_\infty(1+\e).$$ We therefore reformulate \eqref{equaitonoineog} as the fixed point equation:

\be
\label{equaigneoignengon}
\left|\begin{array}{l}
\nut=\nut_\infty(1+\e)=\frac{1}{\nu^{\gamma-2}}\\
\e=G_1(\e)-G_2(\e)+G_3(\e)+G_4(\e)\\
G_1(\e)=\frac{1}{\nut_\infty}H^{-1}\left\{d_1^*\nut_\infty^{\beta_1}\left[(1+\e)^{\beta_1}-1-\beta_1\e\right]\right\}\\
G_2(\e)=\frac{1}{\nut_\infty}H^{-1}\left\{d_2^*\nut_\infty^{\beta_2}\left[(1+\e)^{\beta_2}-1-\beta_2\e\right]\right\}\\
G_3(\e)=\frac{1}{\nut_\infty}H^{-1}\left\{\frac{\gamma-2}{m_0^*}\int_{r>0}\left[\frac{a^2}{r^2}\psi_{{\rm sol},\nu}\frac{\pa_z^2h}{h}\right]\zeta_*(R)rdr\right\}\\
G_4(\e)=\frac{\gamma-2}{m_0^*\nut_\infty}H^{-1}\left[\int_{r>0}\left[{\rm NL}(\nu,\psi_{{\rm sol},\nu})-{\rm NL}(\nu,0)\right]\zeta_*(R)rdr\right].
\end{array}\right.
\ee

Pick a small enough universal constant $\delta_0$, we introduce the norm $$\|\e\|\equiv\sum_{k=0}^2\|\la z\ra^{\delta_0}\la z\ra^k\e^{(k)}\|_{L^\infty}.$$

\begin{proposition}[Contraction mapping]
\label{vneknenenvoeevndnvdkngdklnower} Let $K_a=\sqrt{a}$, then for $0<a<a^*$ small enough, \eqref{equaigneoignengon} admits a unique solution in the ball of radius $K_a$ of $ C^2(\Bbb R,\Bbb R)$ equipped with the norm $\|\e\|.$
\end{proposition}

The rest of this section is devoted to the proof of Proposition \ref{vneknenenvoeevndnvdkngdklnower}.

\subsection{Continuity of the resolvent}

\begin{lemma}[Continuity of the resolvent]
\label{conoeingiorwoskevtn}
Let 
\be
\label{choieviobiebt}
0<\delta_0\le \min\{1,\beta-\frac 32-(\gamma-2)\eta_\infty\}
\ee then we have the bound 
\be
\label{estnoagnrieodkvente}
\left\|\frac{1}{\nut_\infty}H^{-1}f\right\|\lesssim \|\la z\ra^{\beta}f\|_{L^2}.
\ee
\end{lemma}

\begin{proof}[Proof of Lemma \ref{conoeingiorwoskevtn}] Recall \eqref{eniovneineoneonv} which gives
\bee
&&\frac{1}{\nut_\infty}\left|\Gamma\int_z^{+\infty}\frac{f\nu'}{W}d\tau\right|\\
&\lesssim& \la z\ra^{2(\gamma-2)\eta_\infty+2}\left(\int_{z\in \Bbb R}\la z\ra^{2\beta}f^2dz\right)^{\frac 12}\left(\int_z^{+\infty}\frac{d\tau}{\la \tau\ra^{2[(\gamma-2)\eta_\infty+1]+2\beta}}\right)^{\frac 12}\\
& \lesssim & \frac{\la z\ra^{2(\gamma-2)\eta_\infty+2}\|\la z\ra^\beta f\|_{L^2}}{\la z\ra^{(\gamma-2)\eta_\infty+1+\beta-\frac 12}}=\frac{\|\la z\ra^\beta f\|_{L^2}}{\la z\ra^{\beta-\frac 32-(\gamma-2)\eta_\infty}}\lesssim \frac{\|\la z\ra^\beta f\|_{L^2}}{\la z\ra^{\delta_0}},
\eee
and 
\bee
&&\frac{1}{\nut_\infty}\left|\nu'\int_0^z\frac{f\Gamma}{W}d\tau\right|\lesssim \frac{\la z\ra^{(\gamma-2)\eta_\infty}}{\la z\ra^{(\gamma-2)\eta_\infty+1}}\|\la z\ra^\beta f\|_{L^2}\left(\int_0^z\frac{\la \tau\ra^{4+2(\gamma-2)\eta_\infty}}{\la \tau\ra^{2\beta}}d\tau \right)^{\frac 12}\\
&\lesssim &\left|\begin{array}{l}\frac{\|\la z\ra^\beta f\|_{L^2}}{\la z\ra} \ \ \mbox{for}\ \ \beta>\frac 52+(\gamma-2)\eta_\infty\\
\frac{\|\la z\ra^\beta f\|_{L^2}}{\la z\ra^{\beta-\frac 32-(\gamma-2)\eta_\infty}}\lesssim \frac{\|\la z\ra^\beta f\|_{L^2}}{\la z\ra^{\delta_0}} \ \ \mbox{for}\ \ \beta<\frac 52+(\gamma-2)\eta_\infty
\end{array}\right..
\eee
We now take a derivative $$ z\pa_zu =-z\Gamma'\int_z^{+\infty}\frac{f\nu'}{W}d\tau-z\nu''\int_0^z\frac{f\Gamma}{W}d\tau,$$ which yields the corresponding estimate for $\la z\ra\pa_zu$, and the second derivative is estimated from the equation. This concludes the proof of \eqref{estnoagnrieodkvente}.
\end{proof}

\subsection{Estimate for the source term}
We start with estimating the dominant source term.

\begin{lemma}[Source term]
\label{sorucetnent}
 There holds the bound for $\e=\e_1,\e_2\in B_{K_a}$:
$$\|G_3(\e)\|\lesssim a^2M$$ and $$\|G_3(\e_2)-G_3(\e_1)\|\lesssim a^2M\|\e_2-\e_1\|.$$
\end{lemma}

\begin{proof}[Proof of Lemma \ref{sorucetnent}]  We keep track of the $\nu$ dependance of all terms.\\

\noindent{\bf step 1} Boundedness in $B_{K_a}$. We compute:
$$
\left|\begin{array}{l}
\frac{\pa_zh}{h}=-p\frac{\pa_z\nu}{\nu}-\frac{\pa_z\nu}{\nu}\frac{R\partial_R\zeta_*}{\zeta_*}=-\frac{\pa_z\nu}{\nu}\left(p+\frac{R\partial_R\zeta_*}{\zeta_*}\right)\\
\frac{\pa_z^2h}{h}-\left(\frac{\pa_zh}{h}\right)^2=-\pa_z\left(\frac{\pa_z\nu}{\nu}\right)\left(p+\frac{R\partial_R\zeta_*}{\zeta_*}\right)+\left(\frac{\pa_z\nu}{\nu}\right)^2R\partial_R\left[\frac{R\partial_R\zeta_*}{\zeta_*}\right]
\end{array}\right.,
$$ and hence
$$\left|\begin{array}{l}
\frac{\pa_z^2h}{h}\zeta_*=\frac{\pa^2_z\nu}{\nu}H_1(R)+\left(\frac{\pa_z\nu}{\nu}\right)^2H_2(R)\\
H_1=p\zeta_*+R\partial_R\zeta_*\\
H_2=\left\{\left(p+\frac{R\partial_R\zeta_*}{\zeta_*}\right)^2+p+\frac{R\partial_R\zeta_*}{\zeta_*}+R\partial_R\left[\frac{R\partial_R\zeta_*}{\zeta_*}\right]\right\}\zeta_*
\end{array}\right..
$$
Next compute
$$\frac{\pa_z\nu}{\nu}=\frac{\pa_z\nu_\infty}{\nu_\infty}+\frac{\pa_z\e}{1+\e}$$ and thus in $B_{K_a}$ we have:
\be
\label{vniovnenvenvenvenvenvennevoe}
\left|\begin{array}{l}
\left(\frac{\pa_z\nu}{\nu}\right)^2\lesssim \frac{1}{\la z\ra^2}\\
\left|\left(\frac{\pa_z\nu_2}{\nu_2}\right)^2-\left(\frac{\pa_z\nu_1}{\nu_1}\right)^2\right|\lesssim \frac{1}{\la z\ra}\left|\frac{\pa_z\e_2}{1+\e_2}-\frac{\pa_z\e_1}{1+\e_1}\right|\lesssim \frac{\|\e_2-\e_1\|}{\la z\ra^{2+\delta_0}}
\end{array}\right.,
\ee
and similarly
\be
\label{vniovnenvenvenvenvenvennevoebis}
\left|\begin{array}{l}
\left|\frac{\pa^2_z\nu}{\nu}\right|\lesssim \frac{1}{\la z\ra^2}\\
\left|\frac{\pa^2_z\nu_2}{\nu_2}-\frac{\pa^2_z\nu_1}{\nu_1}\right|\lesssim \frac{\|\e_2-\e_1\|}{\la z\ra^{2+\delta_0}}
\end{array}\right..
\ee
This yields the bound in $B_{K_a}$:
$$\left|\frac{a^2}{r^2}\psi_{{\rm sol},\nu}\frac{\pa_z^2h}{h}\zeta_*(R)\right|\lesssim \frac{a^2R^2|\Psi_{{\rm sol},\nu}|}{\la z\ra^{2}\la R\ra^{|\l_-|}},$$
and hence 
\bee
&&\left\|\la z\ra^\beta\int_{r>0}\frac{a^2}{r^2}\psi_{{\rm sol},\nu}\frac{\pa_z^2h}{h}\zeta_*(R)rdr\right\|_{L^2}\lesssim a^2\left\|\la z\ra^\beta\int_{r>0}\frac{a^2R^2|\Psi_{{\rm sol},\nu}|rdr}{\la z\ra^{2}\la R\ra^{|\l_-|}}\right\|_{L^2}\\
& \lesssim & a^2\left[\int_{z\in \Bbb R}\frac{1}{\la z\ra^{4+2\delta}}\left(\int_{r>0}\frac{\la z\ra^{2\beta}\Psi_{{\rm sol},\nu}^2}{r^2}r^3dr\right)\left(\int_{r>0}\frac{R^4r^3dr}{r^3R^{2|\l_-|}}\right)dz\right]^{\frac 12}\\
& \lesssim & a^2\left[\int_{z\in \Bbb R}\frac{\la z\ra^{2\beta}}{\la z\ra^{4-(2|\lambda_-|-4)\eta}}(\pa_r\Psi_{{\rm sol},\nu})^2r^3drdz\right]^{\frac 12}\lesssim a^2M,
\eee
as long as recalling \eqref{vneionveoiveovei} $$2\beta-(4-(2|\lambda_-|-4)\eta_\infty)<3-2\delta_*\Longrightarrow \frac{3}{2}+(\gamma-2)\eta_\infty<\beta<\frac{7}{2}-\delta_*-(|\lambda_-|-2)\eta_\infty.$$  Thus for $\eta_\infty$ sufficiently small (thus $\beta_2\gg1$) we may therefore choose $\beta=3$ and  conclude from \eqref{estnoagnrieodkvente}:
$$\|G_3(\e)\|\lesssim a^2M.$$

\noindent{\bf step 2} Lipschitz regularity.\\

\noindent\underline{$\Psi_{{\rm sol},\nu}$ term}. We claim that
\be
\label{venivnevneionvennvenveonv}
a^2\left\|\la z\ra^3\int_{r>0}\left[\left(\Psi_{{\rm sol},\nu_2}-\Psi_{{\rm sol},\nu_1}\right)\frac{\pa_z^2h}{h}\right]\zeta_*(R)rdr\right\|_{L^2}\lesssim a^2\|\e_2-\e_1\|.
\ee
Indeed, we first estimate with $\beta=3$ as above and recalling \eqref{vneovneionienvoei}:
\bee
&&a^2\left\|\la z\ra^\beta\int_{r>0}\left[\left(\Psi_{{\rm sol},\nu_2}-\Psi_{{\rm sol},\nu_1}\right)\frac{\pa_z^2h}{h}\right]\zeta_*(R)rdr\right\|_{L^2}\\
& \lesssim & a^2\|\Psi_{{\rm sol},\nu_2}-\Psi_{{\rm sol},\nu_1}\|_{E_\alpha}\lesssim a^2\left[ \left\|\frac{\nu_2-\nu_1}{\la z\ra^{\frac{\delta_*}{6}}}\right\|_{L^\infty}+\left\|\pa_z^2(\Psi_{\nu_2}-\Psi_{\nu_1})\right\|_{\alpha+6}\right].
\eee
We estimate for $\delta_0$ chosen such that $\eta_\infty-\frac{\delta}{6}<\delta_0$
$$\left\|\frac{\nu_2-\nu_1}{\la z\ra^{\frac{\delta_*}{6}}}\right\|_{L^\infty}\lesssim \|\e_2-\e_1\|.$$ Next recall $$\Psi_\nu=\frac{1}{\nu^{\gamma}}\Psi_*(R),$$ and hence
$$
\pa_z\Psi_\nu=-\frac{\pa_z\nu}{\nu}\frac{1}{\nu^{\gamma}}\left(\gamma+\Lamdba_R\right)\Psi_*,$$
and 
\bee
\pa^2_z\Psi_\nu &= &\frac{1}{\nu^{\gamma}}\left[-\pa_z\left(\frac{\pa_z\nu}{\nu}\right)\left(\gamma+\Lamdba_R\right)+\left(\frac{\pa_z\nu}{\nu}\right)^2\left(\gamma+\Lamdba_R\right)^2\right]\Psi_*\\
&= &\frac{1}{\nu^\gamma}\left[\frac{\pa^2_z\nu}{\nu}N_1(R)+\left(\frac{\pa_z\nu}{\nu}\right)^2 N_2(R)\right].
\eee
Let $R_\infty=\frac{r}{\nu_\infty},$ then we estimate pointwise for $i=1,2$:
\bee
|N_i(R_2)-N_i(R_1)|&=&\left|(\nu_2-\nu_1)\int_0^1 \pa_\nu G_i(\nu_2+t(\nu_2-\nu_1))dt\right|\\
&\lesssim& \frac{\nu_\infty\|\e_2-\e_1\|}{\la z\ra^{\delta_0}}\frac{1}{\nu_\infty\la R_\infty\ra^{|\l_-|}}\\
&\lesssim &\frac{\|\e_2-\e_1\|}{\la z\ra^{\delta_0}\la R_\infty\ra^{|\l_-|}}.
\eee
Hence the pointwise bound with $R_\infty=\frac{r}{\nu_\infty}$ and recalling \eqref{vniovnenvenvenvenvenvennevoe}:
\bee
\left|\pa_z^2(\Psi_{\nu_2}-\Psi_{\nu_1})\right|\lesssim \frac{\|\e_2-\e_1\|}{\la z\ra^{2+\delta_0}\la R_\infty\ra^{\l_-}},
\eee which ensures the bound
\be
\label{venneonveneovi}
\left\|\pa_z^2(\Psi_{\nu_2}-\Psi_{\nu_1})\right\|_{\alpha+6}\lesssim \|\e_2-\e_1\|\left(\int_{z\in \Bbb R}\int_{r>0}\frac{\la z\ra^{3-2\delta_*}}{\la z\ra^{4+2\delta_0}}\frac{r^3\la r\ra^{2\alpha}drdz}{\la R_\infty\ra^{2\l_-}}\right)^{\frac 12}\lesssim \|\e_2-\e_1\|,
\ee 
provided $\delta_0>\lambda_-\eta_\infty-\delta_*$, which can always be guaranteed by taking $\beta_2$ sufficiently large. The collection of above bounds yields \eqref{venivnevneionvennvenveonv}.\\

\noindent\underline{$\frac{\pa_z^2h}{h}\zeta_*$ term}. We argue as above to estimate pointwise:
$$\left|\frac{\pa_z^2h_2}{h_2}\zeta_*(R_2)-\frac{\pa_z^2h_1}{h_1}\zeta_*(R_1)\right|\lesssim \frac{\|\e_2-\e_1\|R^2}{\la z\ra^{2+\delta_0}\la R_\infty\ra^{|\l_-|}},$$ and hence the bound with $\beta=3$ as above:
\bee
&&a^2\left\|\la z\ra^\beta\int_{r>0}\Psi_{{\rm sol},\nu_i}\left[\frac{\pa_z^2h_2}{h_2}\zeta_*(R_2)-\frac{\pa_z^2h_1}{h_1}\zeta_*(R_1)\right]rdr\right\|_{L^2}\\
& \lesssim & a^2\|\e_2-\e_1\|\left[\int_{z\in \Bbb R}\frac{\la z\ra^{2\beta}}{\la z\ra^{4+2\delta_0}}\left(\int_{r>0}\frac{\Psi^2_{{\rm sol},\nu_1}r^3dr}{r^2}\right)\left(\int_{r>0}\frac{R^4r^3dr}{r^2\la R\ra^{2\l_-}}\right)dz\right]^{\frac 12}\\
& \lesssim & a^2\|\e_2-\e_1\| M.
\eee
The collection of above bounds together with \eqref{sorucetnent} concludes the proof of Lemma \ref{sorucetnent}.
\end{proof}
\subsection{Lipschitz regularity of $G_1,G_2$}
We now estimate the power nonlinearity corrections.

\begin{lemma}[Lipschitz regularity of $G_1,G_2$]
\label{reghogneo}
Consider $\e=\e_1,\e_2\in B_{K_a}$ then for $i=1,2$ the following bounds hold:
$$\|G_i(\e)\|\lesssim K_a^2,$$ and $$\|G_i(\e_2)-G_i(\e_1)\|\lesssim K_a\|\e_2-\e_1\|.$$
\end{lemma}

\begin{proof}[Proof of Lemma \ref{reghogneo}] Let $F(\e)=(1+\e)^\ell-1-\e$, then we have
$$|F(\e_2)-F(\e_1)|=\left|(\e_2-\e_1)\int_0^1F'(\e_1+t(\e_2-\e_1))dt\right|\lesssim \frac{K_a\|\e_2-\e_1\|}{\la z\ra^{2\delta_0}},$$ and higher derivatives are estimated analogously. Moreover the bound $$\nut_\infty^{\beta_i}\lesssim \frac{1}{\la z\ra^{\eta_\infty(\gamma-2)\beta_i}}\lesssim \frac{1}{\la z\ra^{\frac{2\beta_2}{\beta_2-1}}}\lesssim \frac{1}{\la z\ra^2},
$$ensures in $B_{K_a}$ with $\frac{3}{2}+(\gamma-2)\eta_\infty<\beta=\frac{3}{2}+\tilde{\epsilon}$:
$$\left\|\la z\ra^{\beta}\nut_\infty^{\beta_i}\left[(1+\e)^{\beta_i}-1-\beta_i\e\right]\right\|_{L^2}\lesssim K_a^2\left\|\frac{1}{\la z\ra^{\frac{1}{2}+2\delta_0-\tilde{\epsilon}}}\right\|_{L^2}\lesssim  K_a^2,$$ 
under the condition $\tilde{\epsilon}<2\delta_0$ and thus $\eta_\infty\ll1$ and $\beta_2\gg 1$. Thus combined with \eqref{estnoagnrieodkvente} this concludes the proof of Lemma \ref{estnoagnrieodkvente}. 
\end{proof}
\subsection{Lipschitz regularity of $G_4$}  We now estimate the remaining nonlinear interactions.
\begin{lemma}[Lipschitz regularity of $G_4$]
\label{reghogneofour}
Consider $\e=\e_1,\e_2\in B_{K_a}$ then for $i=1,2$ the following bounds hold:
\be
\label{venoneinempbiorjgroitj}
\|G_4(\e)\|\lesssim K_a^2,
\ee and 
\be
\label{venoneinempbiorjgroitjbis}
\|G_4(\e_2)-G_4(\e_1)\|\lesssim K_a\|\e_2-\e_1\|.
\ee
\end{lemma}
\begin{proof}[Proof of Lemma \ref{reghogneofour}] First we write 
\bee
 {\rm NL}(\nu,\psi_{{\rm sol},\nu})-{\rm NL}(\nu,0)&=& \frac{\R_1(\psi_\nu+a^2\psi_{{\rm sol},\nu})-\R_1(\psi_\nu)}{r^2}\\
& - &\left[\R_2(\psi_\nu+a^2\psi_{{\rm sol},\nu})-\R_2(\psi_\nu)\right]+\frac{\mathcal F_1(\nu,\psi_{{\rm sol},\nu})-\mathcal F_2(\nu,\psi_{{\rm sol},\nu})}{a^2}.
\eee
\noindent{\bf step 1} Estimate for $\mathcal F_i$. We estimate from \eqref{firsingengeo}
$$\left\|\la z\ra^{\frac 1{20}}\frac{\mathcal F_1(\nu_2,\psi_{{\rm sol},\nu_2})-\mathcal F_1(\nu_1,\psi_{{\rm sol},\nu_1})}{a^2}\right\|_{\alpha+6}\lesssim aM^2+a^2M\|\e_2-\e_1\|\lesssim aM^2.
$$
Next define 
$$f=\frac{\mathcal F_1(\nu_2,\psi_{{\rm sol},\nu_2})-\mathcal F_1(\nu_1,\psi_{{\rm sol},\nu_1})}{a^2},$$
and estimate
\bee
&&\int_{z\in \Bbb R}\la z\ra^{2\beta}\left(\int_{r>0}f(r,z)\zeta_*(R)rdr\right)^2dz\lesssim \int_{z\in \Bbb R}\la z\ra^{2\beta}\left(\int_{r>0}f^2(r,z)dz\right)\left(\int_{r>0}\zeta_*^2(R)rdr\right)\\
& \lesssim & \int_{r>0} \int_{z\in \Bbb R}\la z\ra^{2\beta+2\eta_\infty}f^2(r,z)dz\lesssim \|\la z\ra^{\frac 1{20}}f\|_{\alpha+6}
\eee
as long as $$2\beta+2\eta_\infty<\frac{1}{10}+3-\delta_*\Longrightarrow \frac{3}{2}+(\gamma-2)\eta_\infty<\beta<\frac{3}{2}+\frac{1}{20}-\eta_\infty-\delta_*.$$ 
We therefore can choose with $\eta_\infty\ll1$, $2\beta=\frac{1}{20}+3-\delta_*$.
\begin{remark}
    It is here that the crucial extra gain of $\frac{1}{\la z\ra^{\frac{1}{20}}}$ of Lemma \ref{lemmanonlin} is used.
\end{remark}
\noindent Thus we get the estimate from \eqref{choieviobiebt} and \eqref{choieviobiebt} for $\delta_0>0$ universal small enough
\bee
&&\left\|\frac{1}{\nut_\infty}H^{-1}\left[\int_{r>0}f(r,z)\zeta_*(R)rdr\right]\right\|\lesssim \|\la z\ra^\beta \left[\int_{r>0}f(r,z)\zeta_*(R)rdr\right]\|_{L^2}\\
&\lesssim& \left\|\la z\ra^{\frac 1{20}}\frac{\mathcal F_1(\nu_2,\psi_{{\rm sol},\nu_2})-\mathcal F_1(\nu_1,\psi_{{\rm sol},\nu_1})}{a^2}\right\|_{\alpha+6}\lesssim aM^2.
\eee
We argue similarly for $\matchal F_2$, and the collection of above bounds injected into \eqref{estnoagnrieodkvente} concludes the proof of \eqref{venoneinempbiorjgroitj}, \eqref{venoneinempbiorjgroitjbis} for $\matchal F_1,\mathcal F_2$.\\

\noindent{\bf step 2} Estimate for $\mathcal R_i$. We now revisit \eqref{firsingengebis} to take into account the $a^2$ smallness. Let $$
\left|\begin{array}{l}
H(u)=u^{\ell}\\
\ell=\frac{\gamma}{\gamma-2}+m_1=\beta_1
\end{array}\right.,
$$
We estimate:
$$\frac{\psi_{\nu}^{\ell-1}}{r^2}=\left(\frac{1}{\nu^{\gamma-2}}\psi_*(R)\right)^{\beta_1}\lesssim\frac{1}{\nu^2R^2}\left( \frac{R^2}{\la R\ra^{\gamma-4}\nu^{\gamma-2}}\right)^{\beta_1-1}\lesssim \frac{1}{\la z\ra^2\la R_\infty\ra^{\frac{2}{\eta_\infty}},}$$
and
$$\frac{\psi_{\nu}^{\ell-2}}{r^2}=\left(\frac{1}{\nu^{\gamma-2}}\psi_*(R)\right)^{\beta_1-1}\lesssim\frac{1}{\nu^2R^2}\left( \frac{R^2}{\la R\ra^{\gamma-4}\nu^{\gamma-2}}\right)^{\beta_1-1}\lesssim \frac{1}{\la z\ra^{2-\eta_\infty(\gamma-2)}\la R_\infty\ra^{\frac{1}{\eta_\infty}}},$$
Next we have 
$$\left|\R_1(\psi_\nu+a^2\psi_{{\rm sol},\nu})-\R_1(\psi_\nu)\right|=\left|a^2\psi_{{\rm sol},\nu}\int_0^1H'(\psi_\nu+a^2t\psi_{{\rm sol},\nu})dt\right|\lesssim \frac{a^2\psi_{{\rm sol},\nu}}{\la z\ra^2\la R_\infty\ra^{\frac{2}{\eta_\infty}}}$$ and

\bea
\label{vnevnenveineinvenevnoegnroegn}
\nonumber &&\frac{1}{r^2}\left\{\R_1(\psi_{\nu_2}+a^2\psi_{{\rm sol},\nu_2})-\R_1(\psi_{\nu_2})-\left[\R_1(\psi_{\nu_2}+a^2\psi_{{\rm sol},\nu_2})-\R_1(\psi_{\nu_1})\right]\right\}\\
& \lesssim & a^2\frac{|\psi_{{\rm sol},\nu_2}-\psi_{{\rm sol},\nu_1}|}{\la z\ra^2\la R_\infty\ra^{\frac{2}{\eta_\infty}}}+a^2\frac{|\psi_{\nu_2}-\psi_{\nu_1}|(|\psi_{{\rm sol},\nu_2}|+|\psi_{{\rm sol},\nu_1}|)}{\la z\ra^{2-\eta_\infty(\gamma-2)}\la R_\infty\ra^{\frac{1}{\eta_\infty}}}.
\eea
We now use in the ball $B_{K_a}$: $$|\psi_{\nu_2}-\psi_{\nu_1}|=|\nu_2-\nu_1|\left|\int_0^1\pa_\nu \psi_\nu(\nu_2+t(\nu_2-\nu_1)dt\right|\lesssim |\e_2-\e_1|,$$ 
and hence the bound:
\bee
&&\frac{1}{r^2}\left\{\R_1(\psi_{\nu_2}+a^2\psi_{{\rm sol},\nu_2})-\R_1(\psi_{\nu_2})-\left[\R_1(\psi_{\nu_2}+a^2\psi_{{\rm sol},\nu_2})-\R_1(\psi_{\nu_1})\right]\right\}\\
&\lesssim & a^2\frac{|\psi_{{\rm sol},\nu_2}-\psi_{{\rm sol},\nu_1}|}{\la z\ra^2\la R_\infty\ra^{\frac{2}{\eta_\infty}}}+a^2\frac{|\psi_{\nu_2}-\psi_{\nu_{1}}|(|\psi_{{\rm sol},\nu_2}|+|\psi_{{\rm sol},\nu_1}|)}{\la z\ra^{2-\eta_\infty(\gamma-2)}\la R_\infty\ra^{\frac{1}{\eta_\infty}}},
\eee
Now we estimate 
\bee
&&\int_{z\in \Bbb R}\la z\ra^{2\beta}\left(\int_{r>0} \frac{|\psi_{{\rm sol},\nu_2}-\psi_{{\rm sol},\nu_1}|}{\la z\ra^2\la R\ra^{\frac{2}{\eta_\infty}}}\zeta_*(R)rdr\right)^2dz\\
&\lesssim &\int_{z\in \Bbb R}\la z\ra^{2\beta}\left(\int_{r>0}\frac{|\Psi_{{\rm sol},\nu_2}-\Psi_{{\rm sol},\nu_1}|^2}{\la z\ra^4}r^3dr\right)\left(\int_{r>0}\zeta_*^2(R)\la r\ra^3dr\right)\\
& \lesssim & \|\Psi_{{\rm sol},\nu_2}-\Psi_{{\rm sol},\nu_1}\|_{E_\alpha}\lesssim \|\e_2-\e_1\|
\eee
as long as $$2\beta-4+4\eta<3-2\delta_*\Longrightarrow \frac{3}{2}+(\gamma-2)\eta_\infty<\beta<\frac{7}{2}-\delta_*-2\eta_\infty,$$ we may therefore take $\beta=3$, and we used \eqref{venneonveneovi} and \eqref{vneovneionienvoei} in the last step. All other terms are treated similarly
and this concludes the proof of \eqref{venoneinempbiorjgroitj} and \eqref{venoneinempbiorjgroitjbis} for $\matchal R_1,\mathcal R_2$ terms.
\end{proof}
\subsection{Proof of Proposition \ref{vneknenenvoeevndnvdkngdklnower} and Theorem \ref{thmmain}}
Lemma \ref{sorucetnent}, Lemma \ref{reghogneo} and Lemma \ref{reghogneofour} ensure from Picard's contraction mapping Theorem that \eqref{equaigneoignengon} admits a unique solution $\nu_{{\rm sol}}$ in the ball $B_{K_a}$ with $K_a=\sqrt{a}$, and this concludes the proof of Proposition \ref{vneknenenvoeevndnvdkngdklnower}.\\
 
\begin{proof}[Proof of Theorem \ref{thmmain}] The profile $$\psi_{\rm tot}=\psi_{\nu_{\rm sol}}+a^2\psi_{{\rm sol},\nu_{\rm sol}}$$ is a stationary solution to the elliptic equation \eqref{vnbeioneinveneoivbis}.\\

\noindent\underline{Regularity for $\psi_{\rm tot}$}. We have recalling \eqref{formulasink} with $\gamma=2+\frac{1}{m}$ :
\[
\frac{\mathcal C\matchal C'}{r^2}=\frac{A(\gamma)}{r^2}\psi_{\rm tot}^{1+2m}\left(1+\frac{a^2}{A(\gamma)}\psi_{\rm tot}^{m_1}\right)
=  A(\gamma)r^{2m}\Psi^{1+2m}{\rm tot}\left(1+\frac{a^2}{A(\gamma)}r^{2m_1}\Psi_{\rm tot}^{m_1}\right)
\]
and similarly from \eqref{vneobnveineoeonoenioebis}:
$$\H'=\psi_{\rm tot}^{\frac{\gamma+2}{\gamma-2}}\left(1+a^2\psi_{\rm tot}^{m_2}\right)=r^{8m+2}\Psi_{\rm tot}^{4m+1}\left(1+a^2r^{2m_2}\Psi_{\rm tot}^{m_2}\right)$$ and the integers $m,m_1$ and $m_2$ are such that $$m_1+\frac{\gamma}{\gamma-2}>m_2+\frac{\gamma+2}{\gamma-2}$$ and large enough to verify all the required hypothesis. This yields the equation 
\bee
    & & -\pa_r^2\Psi_{\rm tot}-\frac{3}{r}\pa_r\Psi_{\rm tot}-\pa_z^2\Psi_{\rm tot}\\
     & &=A(\gamma)r^{2m}\Psi^{2m+1}\left(1+\frac{a^2}{A(\gamma)}r^{2m_1}\Psi_{\rm tot}^{m_1}\right)-r^{8m+2}\Psi_{\rm tot}^{4m+1}\left(1+a^2r^{2m_2}\Psi_{\rm tot}^{m_2}\right).\eee The $\matchal C^\infty$ regularity of $\Psi_{\rm tot}$ now follows from standard elliptic regularity.\\
     
\noindent\underline{Computation of the swirl}. From \eqref{vneobnveineoeonoenioe}
\bee
&&\mathcal C\matchal C'=\frac{A(\gamma)}{r^2}\psi_{\rm tot}^{1+2m}\left(1+\frac{a^2}{A(\gamma)}\psi_{\rm tot}^{m_1}\right)\\
&\Rightarrow& \frac{\mathcal C^2}{2}=\frac{A(\gamma)}{r^2}\psi_{\rm tot}^{2+2m}\left(\frac{1}{2+2m}+\frac{1}{2+2m+m_1}\frac{a^2}{A(\gamma)}\psi_{\rm tot}^{m_1}\right),\eee
thus
\bee
u_\phi&= &\frac{\mathcal C}{r}=\frac{\sqrt{A(\gamma)}}{r}\psi_{\rm tot}^{1+m}\left(\frac{1}{2+2m}+\frac{1}{2+2m+m_1}\frac{a^2}{A(\gamma)}\psi_{\rm tot}^{m_1}\right)^{\frac{1}{2}}\\
& = & \sqrt{A(\gamma)}r^{1+2m}\Psi_{\rm tot}^{1+m}\left(\frac{1}{2+2m}+\frac{1}{2+2m+m_1}\frac{a^2}{A(\gamma)}\psi_{\rm tot}^{m_1}\right)^{\frac{1}{2}},
\eee
 This structure implies the regularity of the vector field \eqref{aysitniflow} which satisfies the stationary equation \eqref{remainingequationsbis} by Lemma \ref{vneoivnovnkvnnve}.\\

\noindent\underline{Non vanishing}. $\Psi_{\rm tot}$ does not vanish since $a^2|\Psi_{{\rm sol},\nu}|\leq \frac{|\Psi_\nu|}{2}$ by construction and $\Psi_\nu$ does not vanish since $\Psi_*$ does not vanish. The same conclusion holds for the swirl $u_\phi$.\\

This concludes the proof of the vortex construction of Theorem \ref{thmmain}.
\end{proof}


\section{Bifurcation of a self similar profile}
\label{sectionselfsim}

We now turn to the construction of the self similar profile of Theorem \ref{thmmain}.


\subsection{Bifurcation}


Let swirl and curl be given exactly by the homogeneous non linearities \eqref{veonoeinoineonevo}, then a brute force computation reveals that the vector field generated by the stream function 
\be
\label{vneionoennioev}
\psi(r,z)=\frac{1}{(\gamma+1)|z|^{\gamma-2}}\Psi\left(\frac{r}{|z|}\right)
\ee is a solution to the self similar equation \eqref{selfsim} whenever $$\left|\bear
\Psi=r^{2-\gamma}\phi(x)\\
x=\log\left(\frac{r}{|z|}\right)
\ear\right.
$$ 
solves the non linear ode $$\left[\pa_x^2-2(\gamma-1)\pa_x+\gamma(\gamma-2)\right]\phi+e^{2x}\left(\pa_x^2+\pa_x\right)\phi=|\phi|^{\frac{\gamma+2}{\gamma-2}}-A|\phi|^{\frac{\gamma}{\gamma-2}}$$ which up to a translation in $x$ is equivalent to
\be
\label{eq: closed ODE}
\left[\pa_x^2-2(\gamma-1)\pa_x+\gamma(\gamma-2)\right]\phi+a^2e^{2x}\left(\pa_x^2+\pa_x\right)\phi=|\phi|^{\frac{\gamma+2}{\gamma-2}}-A|\phi|^{\frac{\gamma}{\gamma-2}}.
\ee 
One easily checks that the corresponding velocity field satisfies \eqref{veionoienoinoevnoive} and hence these solutions are both stationary and self similar.\\

\noi\und{Degenerate vortex}. For $a=0$, this is verbatim the Emden formulation \eqref{emdnenfonro} of the degenerate vortex equation. We thus let $(\phi=\phi_\gamma, A=A_\gamma)$ be the degenerate vortex profile of Proposition \ref{propappendix} which has  the asymptotic behaviour
\be
\label{waiofneonei}
\phi_\gamma=\left|\bear \left[1+o_{x\to -\infty}(1)\right]e^{\gamma x}\\
\phi_{\gamma,+}\left[1+O_{x\to +\infty}(e^{\l_-x})\right], \ \ \phi_{\gamma,+}>0.
\ear\right.
\ee

\noi\und{Bifurcation}. The conclusion of Theorem \ref{thmmain} now follows from the following bifurcation claim.

\begin{proposition}[Bifurcated self similar profile]  
\label{thm: exist of sol}
  Let 
  \be
  \label{vieonoienvnoevnnevo}
  \gamma=2+\frac{1}{m}, \ m\in \mathbb{N}^*,
  \ee
  then for $m\geq m_*\gg1 $ universal large enough, 
  there exists $0<a_*(m)\ll1 $ such that for $0\leq a<a_*$ there exists $A_{\gamma,a}=A_{\gamma}+O(a^2)>0$ and a unique positive smooth solution $\phi_{\gamma,a}$ to \eqref{eq: closed ODE} verifying 
  \be
  \label{ieononveoompmev}
  \phi_{\gamma,a}(x)=\left|\bear c_{\gamma,a}\left[1+o_{x\to -\infty}(1)\right]e^{\gamma x}\\
  \phi^+_{\gamma,a}+\sum_{j=1}^k d_j e^{-2j x}+O_{x\to+\infty}\left[e^{-2(k+1)x}\right],\ \ \forall k\ge 1
  \ear\right.
  \ee
  for some constants $\phi^+_{\gamma,a}=\phi^+_{\gamma}+O(a^2)$ and computable sequence $(d_j)_{j\ge 1}.$
  \end{proposition}

\begin{remark}[Regularity and decay of the stream function \eqref{vneionveoinioenvnev}] The behaviour near $-\infty$ of the profile together with the quantized choice \eqref{vieonoienvnoevnnevo} ensures the $\mathcal C^\infty$ regularity of the velocity field \eqref{aysitniflow}, \eqref{vneionoennioev} at $r=0$, $z\neq 0$. Self similar decay at $+\infty$ follows from \eqref{ieononveoompmev}. Let us insist that the asymptotic expansion  \eqref{ieononveoompmev} near $+\infty$ encodes {\em a cancellation at the heart of our bifurcation analysis}, see Lemma \ref{elnannn}, since the canonical decay at $+\infty$ of solutions to \eqref{eq: closed ODE} near $\phi^+_{\gamma}$ should involve powers of $e^{-x}$, not $e^{-2x}$. Pick now $r>0$, then the stream function \eqref{vneionveoinioenvnev}
  $$\psi(r,z)=\frac{1}{|aZ|^{\gamma-2}}\psi_a\left(\frac{r}{|Z|}\right)=\frac{1}{|aZ|^{\gamma-2}}\left(\frac{r}{|Z|}\right)^{2-\gamma}\phi_{\gamma,a}\left(\frac{r}{|Z|}\right)=\frac{1}{(ra)^{\gamma-2}}\phi_{\gamma,a}\left(\frac{r}{|Z|}\right)
  $$
  admits near $Z=0$ from \eqref{vieonoienvnoevnnevo} an expansion where $x=\log\left(\frac{r}{|Z|}\right)$:
  \bee
  \psi(r,z)&=&\frac{1}{(ra)^{\gamma-2}}\left[\phi^+_{\gamma,a}+\sum_{j=1}^k d_j e^{-2jx}+O_{x\to+\infty}\left[e^{-2(k+2)x}\right]\right]\\
  &=& \frac{1}{(ra)^{\gamma-2}}\left[\phi^+_{\gamma,a}+\sum_{j=1}^k d_j \left(\frac{Z}{r}\right)^{2j}+O\left(\frac{|Z|}{r}\right)^{2k+2}\right]
  \eee
  which ensures the $\matchal C^\infty$ regularity through $z=0$ of the full stream function. The regularity of swirl follows from the smoothness of the homogeneous non linearity, hence the $\matchal C^\infty$ smoothness of the full velocity field in $\Bbb R^3\backslash \{0\}$.
\end{remark}

The rest of this section is devoted to the proof of Proposition \ref{thm: exist of sol} which follows from a non linear perturbative argument. In all this section, $\gamma$ is fixed once and for all close to enough to $2$ according to \eqref{vieonoienvnoevnnevo}. Following \eqref{vneioneinveoinenoier}, we let 
 \be
 \label{febionoienionoef}
 \left|\bear
 \l_{+,\gamma}+\l_{-,\gamma}=2(\gamma-1)
\ear\right.
 \ee be the unstable eigenvalue at $+\infty$ of the degenerate vortex profile and recall \eqref{defnumberK}, \eqref{fneionfenvoinoe}:
  \be
  \label{vneinveonioenvnnievonove}
  \left|\bear
  \l_{-,\gamma}=\gamma-1-\sqrt{1+K_\gamma}\\
K_{\gamma}=\frac{(\phi^+_{\gamma})^{\frac{2}{\gamma-2}}}{\gamma-2}\left[(\gamma+2)(\phi^+_{\gamma})^{\frac{2}{\gamma-2}}-\gamma A_{\gamma}\right]\underset{\gamma \downarrow 2}{\to}+\infty.
\ear\right.
\ee
Our strategy is to construct for all $0<a<a_*$ a family of inner solutions from $-\infty$ which is are then shown not to exit a small tube around $\phi_{\gamma,+}$, and for a suitable choice of parameters have the anomalous asymptotic expansion \eqref{waiofneonei}


\subsection{Inner solutions}


We construct a family of  inner solutions which bifurcate from $\phi_{\mu}$ on $(-\infty, x_*]$, $x_*=x_*(a)\gg1$. 

\begin{lemma}[Inner solutions]
\label{inenrsotuon}
There exist constants $$c_\gamma>0, v_{1,\gamma}>0, v_{2,\gamma}>0$$ such that the following holds. Pick  $0<\eta_*\ll 1$, $\Lambda_*\gg1 $ be  respectively small and large enough. Then there exists $a_*\equiv a_*(\eta_*,\Lambda_*)>0$ such that for all $0<a<a_*$, the following holds. Let 
\be
\label{conosnmslalnes}
x_{*}=\ln\left(\frac{\eta_*}{a}\right)\gg 1.
\ee 
then 
\be
\label{eq: eps bound}
\forall \eps\in \Bbb R \ \ \mbox{with}\ \ |\epsilon e^{\lambda_{+,\gamma}x_*}| \leq \Lambda_* a^2 ,
\ee

 there exist
   \be
   \label{venovenineoniennoev}
   \left|
   \bear
   c\equiv c(a,x_*,\epsilon)=-c_{\gamma}a^2+O(a^2\eta_* )\\
   c_{\gamma}>0\\
   A\equiv A_{\gamma,a,\epsilon}=A_{\gamma}+c+\epsilon\\
   \phi_{\gamma,a,\epsilon}=\phi_{\gamma}+ v 
   \ear\right.
   \ee
  such that $\phi_{\gamma,a,\epsilon}$ solves \eqref{eq: closed ODE} on $(-\infty,x_*]$ with the endpoint values 
   \be
  \label{eq: exp v x*}
  \left|\bear
  v(x_*)=v_{1,\gamma}a^2-v_{2,\gamma}e^{\l_{+,\gamma}x_*}\eps+O\left(a^2\eta_*\right)\\
  \pa_xv(x_*)=-\l_{+,\gamma}v_{2,\gamma}e^{\l_{+,\gamma}x_*}\eps+O\left(a^2\eta_*\right)
  \ear\right.
  \ee
  where $O$ constants do not depend on $\Lambda_*$.

 \end{lemma}

\begin{remark} The linear term in $\eps$ dominates in the expansion \eqref{eq: exp v x*} at the extremities of the condition   \eqref{eq: eps bound} provided $\Lambda_*$ has been chosen large enough.
\end{remark}

\begin{proof}[Proof of Lemma \ref{inenrsotuon}] This is a bifurcation analysis. We define for $k\in \Bbb N$.
\[
\left\Vert f \right\Vert_{E^{k,x_*}_{-}}=\sup_{0\leq i\leq k}\left\Vert \left[1+e^{-(\gamma+2)x}\right]\partial_x^i f(x)\right\Vert_{L^\infty\left((-\infty,x_*)\right)}
\]
We will systematically use from \eqref{conosnmslalnes}, \eqref{febionoienionoef}, \eqref{vneinveonioenvnnievonove}: $$
\left|\bear
a e^{x_*}=\eta_*\\
1\ll |\l_{-,\gamma}|\le \l_{+,\gamma}\Rightarrow e^{2x_*}e^{-\l_{+,\gamma}x_*}\leq e^{2x_*}e^{-|\lambda_{-,\gamma}|x_*}=\left(\frac{a}{\eta_*}\right)^{|\l_{-,\gamma}|-2}\ll a^2.
\ear\right.
$$
It is understood in the sequel that $C_\gamma$ is a universal constant that does not depend on $\Lambda_*$ in \eqref{eq: eps bound}.\\

\noi{\bf step 1} The linearized flow.\\

\noi\und{Linearized formulation}. We let $$A\equiv A_{a,\gamma}=A_\gamma+c+\epsilon$$ and look for a solution to \eqref{eq: closed ODE} in the form $\phi=\phi_\gamma +v$ so that 
\bea
\label{lienarinoing}
\nonumber L^{-}_{\gamma} v&=&-a^2e^{2x}(\partial^2_x+\partial_x)\left(\phi_{\gamma}+v\right)-(c+\epsilon)|\phi_\gamma|^{\frac{\gamma}{\gamma-2}}-(c+\epsilon)\frac{\gamma \int_0^{1}\left(\phi_\gamma+s  v\right)|\phi_\gamma+s v|^{\frac{4-\gamma}{\gamma-2}}ds }{\gamma-2}v\\
&+&\frac{\int_0^1 \left[4(\gamma+2)|\phi_\gamma+sv|^{\frac{6-\gamma}{\gamma-2}}-2\gamma A_{\gamma} |\phi_\gamma+sv|^{\frac{4-\gamma}{\gamma-2}}\right](1-s)ds}{(\gamma-2)^2}  v^2
\eea
where $L^{-}_{\gamma}$ is the linearized operator close to $\phi_\gamma$ of \eqref{eq: closed ODE} for $a=0$:
$$L^{-}_{\gamma}h= \left[\partial_x^2-2(\gamma-1)\partial_x+\gamma(\gamma-2)\right]h-\frac{\gamma+2}{\gamma-2}\phi_\gamma|\phi_\gamma|^{\frac{6-\gamma}{\gamma-2}}h+\frac{A_{\gamma}\gamma}{\gamma-2}\phi_\gamma|\phi_\gamma|^{\frac{4-\gamma}{\gamma-2}}h.
$$

\noi\und{Fundamental basis}. By translation invariance of the vortex equation \eqref{emdnenfonro}:
$$L^{-}_{\gamma}\phi_\gamma'=0$$ which implies the asymptotic expansion 
\be
\label{eq:int basis asymp 1}
\phi_\gamma'(x)= \left|\bear
\gamma \left[1+o_{x\to -\infty}(1)\right]e^{\gamma x}\\
d_\gamma e^{\lambda_{-,\gamma} x}+O_{x\to+\infty}\left(e^{2\lambda_{-,\gamma} x}\right), \ \\ d_{\gamma}>0.
\ear\right.
\ee
The non degeneracy $\phi_\gamma'>0$ allows us to exhibit a basis of fundamental solutions  $(\phi_\gamma',h_2)$ with normalized Wronskian ${\rm Wr}(\phi_\gamma',h_2)=e^{2(\gamma-1)x}$:
\be
\label{esthtwto}
h_2(x)=\phi'_\gamma(x)\int_{0}^{x}\frac{ e^{2(\gamma-1)y}}{(\phi_\gamma')^2(y)}dy=\left|\bear
-\frac{1+o_{x\to -\infty}(1)}{2\gamma}e^{(\gamma-2) x}\\
\frac{1+O_{x\to +\infty}(e^{\lambda_{-,\gamma}x})}{2d_\gamma \left(\gamma-1-\lambda_{-,\gamma}\right)} e^{\lambda_{+,\gamma}  x}
\ear\right.
\ee
where we used \eqref{febionoienionoef}, and similarly for derivatives.\\

\noi\und{Resolvent}. The solutions to
    \[y''+py'+qy=g\]
    are
    \be
    \label{variontocnstante}
    y=c_1y_1+c_2y_2-y_1(t)\int^{t}_{t_0}\frac{y_2(s)}{W[y_1,y_2](s)}g(s)ds+y_2(t)\int^{t}_{t_0}\frac{y_1(s)}{W[y_1,y_2](s)}g(s)ds
    \ee
    where $(y_1,y_2)$ is a basis of fundamental solutions to the homogeneous equation with Wronskian 
    \be
    \label{fenoinoenineovnoive}
    W[y_1,y_2]\equiv y_1 y_2'-y_1'y_2=Ce^{-\int p}.
    \ee
We freeze an inverse by defining the resolvent
$$R^{-}_{\gamma}f=-\phi_\gamma'(x)\int^{x}_{-\infty} h_2(y) e^{-2(\gamma-1)y} f(y)dy+h_2(x)\int^{x}_{-\infty} \phi_\gamma'(y) e^{-2(\gamma-1)y} f(y)dy.
$$

\noi\und{Projection operator.} Given  a smooth increasing cut-off function 
$$
\chi_{+}(x)=0 \text{ for }x\leq -1 \text{ and }\chi_{+}(x)=1 \text{ for }x\geq 0,
$$
we introduce the projection 
\be
\label{eq: def proj}
    \mathcal{P}^{x_*}_{\phi'_{\gamma}}f=h_2(x)\chi_{+}(x)\int^{x_*}_{-\infty} \phi_\gamma'(y) e^{-2(\gamma-1)y} f(y)dy.
\ee

\noi{\bf step 2} Continuity of the resolvent. We claim the bounds
\be
\label{contionfiut}
\left\Vert R^{-}_{\gamma} f\right\Vert_{E^{2,x_*}_{-}} \leq C_{\gamma}(1+e^{\lambda_{+,\gamma} x_*})\left\Vert  f\right\Vert_{E^{0,x_*}_{-}}
\ee
and 
\be
\label{cointorissbis}
\left\Vert \left(R^{-}_{\gamma}- \mathcal{P}^{x_*}_{\phi'_{\gamma}}\right) f\right\Vert_{E^{2,x_*}_{-}} \leq C_{\gamma}\left\Vert  f\right\Vert_{E^{0,x_*}_{-}}.
\ee

\noi{\em Proof of \eqref{contionfiut}}. We estimate near $-\infty$ from \eqref{eq:int basis asymp 1}, \eqref{esthtwto}:

\begin{align}\label{eq: est R- phi 1}
\left|\phi_\gamma'(x)\int^{x}_{-\infty} h_2(y) e^{-2(\gamma-1)y} f(y)dy\right|&\leq C_{\gamma} \left\Vert  f\right\Vert_{E^{0,x_*}_{-}} e^{\gamma x}\int^{x}_{-\infty} e^{(\gamma-2) y} e^{-2(\gamma-1)y} e^{(\gamma+2) y}dy\nonumber \\
&\leq C_{\gamma} \left\Vert  f\right\Vert_{E^{0,x_*}_{-}} e^{(\gamma+2) x},
\end{align}
and 
\begin{align}\label{eq: est R- h2 1}
\left|h_2(x)\int^{x}_{-\infty} \phi_\gamma'(y) e^{-2(\gamma-1)y} f(y)dy\right|&\leq C_{\gamma} \left\Vert  f\right\Vert_{E^{0,x_*}_{-}} e^{(\gamma-2) x}\int^{x}_{-\infty} e^{\gamma y} e^{-2(\gamma-1)y} e^{(\gamma+2) y}dy\nonumber \\
&\leq C_{\gamma} \left\Vert  f\right\Vert_{E^{0,x_*}_{-}} e^{(\gamma+2) x}.
\end{align}
Then for $x\gg1 $ using again \eqref{eq:int basis asymp 1}, \eqref{esthtwto}
\begin{align}\label{eq: est R- phi 2}
\left|\phi_\gamma'(x)\int^{x}_{-\infty} h_2(y) e^{-2(\gamma-1)y} f(y)dy\right|&\leq C_{\gamma} \left\Vert  f\right\Vert_{E^{0,x_*}_{-}} e^{\lambda_{-,\gamma} x}\left[1+\int^{x}_{0} e^{\lambda_{+,\gamma} y} e^{-2(\gamma-1)y} dy\right]\nonumber\\
&\leq C_{\gamma} \left\Vert  f\right\Vert_{E^{0,x_*}_{-}},
\end{align}
and 
\begin{align*}
\left|h_2(x)\int^{x}_{-\infty} \phi_\gamma'(y) e^{-2(\gamma-1)y} f(y)dy\right|&\leq C_{\gamma} \left\Vert  f\right\Vert_{E^{0,x_*}_{-}} e^{\lambda_{+,\gamma} x}\left[1+\int^{x}_{0} e^{\lambda_{-,\gamma} y} e^{-2(\gamma-1)y} dy\right]\\
&\leq C_{\gamma} \left\Vert  f\right\Vert_{E^{0,x_*}_{-}} e^{\lambda_{+,\gamma} x}.
\end{align*}
Putting the previous estimates together gives 
\[
\left\Vert R^{-}_{\gamma} f\right\Vert_{E^{0,x_*}_{-}} \leq C_{\gamma}(1+e^{\lambda_{+,\gamma} x_*})\left\Vert  f\right\Vert_{E^{0,x_*}_{-}}.
\]
First and second derivatives are controlled similarly to derive \eqref{contionfiut}.\\

\noi{\em Proof of \eqref{cointorissbis}} We write 
    \begin{align*}
        &R^{-}_{\gamma}f-h_2(x)\chi_{+}(x)\int^{x_*}_{-\infty} \phi_\gamma'(y) e^{-2(\gamma-1)y} f(y)dy=-\phi_\gamma'(x)\int^{x}_{-\infty} h_2(y) e^{-2(\gamma-1)y} f(y)dy\\
        &+h_2(x)\left[1-\chi_{+}(x)\right]\int^{x}_{-\infty} \phi_\gamma'(y) e^{-2(\gamma-1)y} f(y)dy +h_2(x)\chi_{+}(x)\int^{x}_{x_*} \phi_\gamma'(y) e^{-2(\gamma-1)y} f(y)dy.
    \end{align*}
   From \eqref{eq: est R- phi 1} and \eqref{eq: est R- phi 2}
    \[
\left\Vert \phi_\gamma'(x)\int^{x}_{-\infty} h_2(y) e^{-2(\gamma-1)y} f(y)dy\right\Vert_{E^{0,x_*}_{-}} \leq C_{\gamma}\left\Vert  f\right\Vert_{E^{0,x_*}_{-}}
\]
and from \eqref{eq: est R- h2 1} 
\[
\left\Vert h_2(x)(1-\chi_{+}(x))\int^{x}_{-\infty} \phi_\gamma'(y) e^{-2(\gamma-1)y} f(y)dy\right\Vert_{E^{0,x_*}_{-}} \leq C_{\gamma}\left\Vert  f\right\Vert_{E^{0,x_*}_{-}}.
\]
For the last term, we note that it suffices to work with $1\ll x\leq x_{*}$ as $\chi_{+}$ vanishes at $-\infty$ giving the required decay.  In this range $\chi_+(x)=1$ and
    \[
    \left\vert h_2(x)\int^{x}_{x_{*}} \phi_\gamma'(y) e^{-2(\gamma-1)y} f(y)dy\right\vert \leq C_{\gamma}\underbrace{e^{\lambda_{+,\gamma}x}\left(e^{\lambda_{-,\gamma}x}e^{-2(\gamma-1)x}+e^{\lambda_{-,\gamma}x_*}e^{-2(\gamma-1)x_{*}}\right)}_{\leq 2}\left\Vert  f\right\Vert_{E^{x_*}_{-}}.
    \]
    Putting the previous estimates together ensures
  $$
\left\Vert R^{-}_{\gamma} f-h_2(x)\chi_{+}(x)\int^{x_*}_{-\infty} \phi_\gamma'(y) e^{-2(\gamma-1)y} f(y)dy\right\Vert_{E^{0,x_*}_{-}} \leq C_{\gamma}\left\Vert  f\right\Vert_{E^{0,x_*}_{-}}.
$$
First and second derivatives are controlled similarly to derive \eqref{cointorissbis}.\\

\noi{\bf step 3} Source term. We now estimate the source term generated by $(a,c,\eps)$ corrections in the rhs of \eqref{lienarinoing} through the resolvent operator and claim
\be
\label{ionoeinoe}
   \left|\bear     R^{-}_{\gamma}\left[e^{2x}(\partial^2_x+\partial_x)\phi_{\gamma}\right]=\left[1+(\gamma-2)\right]\left[\int^{+\infty}_{-\infty}  e^{2(2-\gamma)y} \left[\phi_\gamma'(y)\right]^2dy\right]h_2(x)+O\left(e^{(2+\lambda_{-,\gamma}) x}\right)\\
   R^{-}_{\gamma}\left[|\phi_\gamma|^{\frac{\gamma}{\gamma-2}}\right]=\left[\int^{+\infty}_{-\infty} \phi_\gamma'(y) e^{-2(\gamma-1)y}|\phi_\gamma|^{\frac{\gamma}{\gamma-2}}(y)dy\right]h_2(x)
   +O(1)
 \ear\right.
 \ee
 and
\be
\label{finaneonoexp}
\left|\bear
 \left(R^{-}_{\gamma}-\mathcal{P}^{x_*}_{\phi'_{\gamma}}\right)\left[e^{2x}(\partial^2_x+\partial_x)\phi_{\gamma}\right]=O_{W^{1,\infty}}\left[e^{(2+\lambda_{-,\gamma}) x}\right]\\
\left(R^{-}_{\gamma}-\mathcal{P}^{x_*}_{\phi'_{\gamma}}\right)\left(|\phi_\gamma|^{\frac{\gamma}{\gamma-2}}\right)=-\frac{|\phi^+_\gamma|^{\frac{\gamma}{\gamma-2}}}{2 \left(\gamma-1-\lambda_{-,\gamma}\right)\left[\lambda_{+,\gamma}-2(\gamma-1)\right]}   +O\left(\la x\ra e^{\lambda_{-,\gamma} x}\right)\\
 \partial_x\left[\left(R^{-}_{\gamma}-\mathcal{P}^{x_*}_{\phi'_{\gamma}}\right)\left(|\phi_\gamma|^{\frac{\gamma}{\gamma-2}}\right)\right]=O(\la x\ra e^{\lambda_{-,\gamma} x}).
 \ear\right.
 \ee

 \noi{\em Proof of \eqref{ionoeinoe}.} We systematically use \eqref{eq:int basis asymp 1}, \eqref{esthtwto}. First
    \[
R^{-}_{\gamma}\left(e^{2x}\partial^2_x\phi_{\gamma}\right)=-\phi_\gamma'(x)\int^{x}_{-\infty} h_2(y) e^{-2(\gamma-1)y} e^{2y}\partial^2_x\phi_{\gamma}(y) dy+h_2(x)\int^{x}_{-\infty} \phi_\gamma'(y) e^{-2(\gamma-1)y} e^{2y}\partial^2_x\phi_{\gamma}(y)dy,
    \]
  We estimate for $1\ll y $ 
  \bee
  &&
     \left|h_2(y) e^{-2(\gamma-1)y} e^{2y}\partial^2_x\phi_{\gamma}(y)\right|\leq C_{\gamma}e^{\lambda_{+,\gamma}  y}e^{-2(\gamma-1)y} e^{2y}  \lambda_{-,\gamma} e^{\lambda_{-,\gamma} y}=C_{\gamma}e^{2y}\\
     &\Rightarrow& 
  \left\vert \phi_\gamma'(x)\int^{x}_{-\infty} h_2(y) e^{-2(\gamma-1)y} e^{2y}\partial^2_x\phi_{\gamma}(y) dy\right\vert \leq C_\gamma e^{(2+\lambda_{-,\gamma}) x}.
  \eee
  Then
  \begin{align*}
  &\int^{x}_{-\infty} \phi_\gamma'(y) e^{-2(\gamma-1)y} e^{2y}\partial^2_x\phi_{\gamma}(y)dy=\frac{1}{2}\int^{x}_{-\infty}  e^{2(2-\gamma)y} \partial_x\left[\phi_\gamma'(y)\right]^2dy\\
  &=\frac{1}{2}\phi_\gamma'(y)^2 e^{2(2-\gamma)y} +(\gamma-2)\int^{x}_{-\infty}  e^{2(2-\gamma)y} \left[\phi_\gamma'(y)\right]^2dy>0,
  \end{align*}
  and hence 
  \[
  h_2(x)\int^{x}_{-\infty} \phi_\gamma'(y) e^{-2(\gamma-1)y} e^{2y}\partial^2_x\phi_{\gamma}(y)dy=(\gamma-2)\left[\int^{+\infty}_{-\infty}  e^{2(2-\gamma)y} \left[\phi_\gamma'(y)\right]^2dy\right]h_2(x)+O\left(e^{(2+\lambda_{-,\gamma}) x}\right).
  \]
  We analogously get 
  \begin{align*}   R^{-}_{\gamma}\left(e^{2x}\partial_x\phi_{\gamma}\right)&=-\phi_\gamma'(x)\int^{x}_{-\infty} h_2(y) e^{-2(\gamma-1)y} e^{2y}\partial_x\phi_{\gamma}(y) dy+h_2(x)\int^{x}_{-\infty} \phi_\gamma'(y) e^{-2(\gamma-1)y} e^{2y}\partial_x\phi_{\gamma}(y)dy,\\
    &=\left[\int^{+\infty}_{-\infty}  e^{2(2-\gamma)y} \left[\phi_\gamma'(y)^2\right]dy\right]h_2(x)+O\left(e^{(2+\lambda_{-,\gamma}) x}\right)
   \end{align*}
   and the first estimate in \eqref{ionoeinoe} follows. We now turn to 
   \[
R^{-}_{\gamma}\left(|\phi_\gamma|^{\frac{\gamma}{\gamma-2}}\right)=-\phi_\gamma'(x)\int^{x}_{-\infty} h_2(y) e^{-2(\gamma-1)y} |\phi_\gamma|^{\frac{\gamma}{\gamma-2}}(y) dy+h_2(x)\int^{x}_{-\infty} \phi_\gamma'(y) e^{-2(\gamma-1)y}|\phi_\gamma|^{\frac{\gamma}{\gamma-2}}(y)dy,
   \]
   and use the bound $
   |\phi_\gamma|^{\frac{\gamma}{\gamma-2}}(y)\leq C_{\gamma}$ to estimate
   \[
   \left\vert\phi_\gamma'(x)\int^{x}_{-\infty} h_2(y) e^{-2(\gamma-1)y} |\phi_\gamma|^{\frac{\gamma}{\gamma-2}}(y) dy\right\vert\leq C_\gamma.
   \]
 Next 
 \bee
 &&
    \int^{x}_{-\infty} \phi_\gamma'(y) e^{-2(\gamma-1)y}|\phi_\gamma|^{\frac{\gamma}{\gamma-2}}(y)dy\\
    &=&   \int^{+\infty}_{-\infty} \phi_\gamma'(y) e^{-2(\gamma-1)y}|\phi_\gamma|^{\frac{\gamma}{\gamma-2}}(y)dy-  \int^{+\infty}_{x} \phi_\gamma'(y) e^{-2(\gamma-1)y}|\phi_\gamma|^{\frac{\gamma}{\gamma-2}}(y)dy
 \eee
 and 
 \begin{align*}
\int^{+\infty}_{x} \phi_\gamma'(y) e^{-2(\gamma-1)y}|\phi_\gamma|^{\frac{\gamma}{\gamma-2}}(y)dy=O\left[e^{(\lambda_{-,\gamma}-2(\gamma-1))x}\right],
 \end{align*}
 ensure
 \begin{align*}
   R^{-}_{\gamma}\left(|\phi_\gamma|^{\frac{\gamma}{\gamma-2}}\right)&=\left[\int^{+\infty}_{-\infty} \phi_\gamma'(y) e^{-2(\gamma-1)y}|\phi_\gamma|^{\frac{\gamma}{\gamma-2}}(y)dy\right]h_2(x)+O(1).
 \end{align*}
and \eqref{ionoeinoe} is proved.\\

\noi{Proof of \eqref{finaneonoexp}}. We work in the range $x\ge 0$. The first estimate follows from \eqref{ionoeinoe}. We then compute
    \[
\left(R^{-}_{\gamma}-\mathcal{P}^{x_*}_{\phi'_{\gamma}}\right)\left(|\phi_\gamma|^{\frac{\gamma}{\gamma-2}}\right)=-\phi_\gamma'(x)\int^{x}_{-\infty} h_2(y) e^{-2(\gamma-1)y} |\phi_\gamma|^{\frac{\gamma}{\gamma-2}}(y) dy,
    \]
    and use
    \[
    h_2(x)= \frac{1}{2d_\gamma \left(\gamma-1-\lambda_{-,\gamma}\right)} e^{\lambda_{+,\gamma}  x}+O_{x\to+\infty}(e^{\left(\lambda_{-,\gamma}+2(\gamma-1)\right)x}),
    \]
    and 
    \bee
&&\phi_\gamma(x)=\phi^+_\gamma+\frac{d_\gamma}{\lambda_{-,\gamma}} e^{\lambda_{-,\gamma}x}+O_{x\to+\infty}(e^{2\lambda_{-,\gamma}x})\\
&\Rightarrow&     |\phi_\gamma|^{\frac{\gamma}{\gamma-2}}=|\phi^+_\gamma|^{\frac{\gamma}{\gamma-2}}+O_{x\to+\infty}(e^{\lambda_{-,\gamma}x}),
   \eee
    to estimate for $x\gg 1$:
    \[
    \int^{x}_{-\infty} h_2(y) e^{-2(\gamma-1)y} |\phi_\gamma|^{\frac{\gamma}{\gamma-2}}(y) dy=\frac{|\phi^+_\gamma|^{\frac{\gamma}{\gamma-2}}}{2d_\gamma \left(\gamma-1-\lambda_{-,\gamma}\right)\left[\lambda_{+,\gamma}-2(\gamma-1)\right]} e^{(\lambda_{+,\gamma}-2(\gamma-1))  x}+O_{x\to+\infty}(x).
    \]
    Thus 
    \[
     \left(R^{-}_{\gamma}-\mathcal{P}^{x_*}_{\phi'_{\gamma}}\right)\left(|\phi_\gamma|^{\frac{\gamma}{\gamma-2}}\right)=\frac{-|\phi^+_\gamma|^{\frac{\gamma}{\gamma-2}}}{2 \left(\gamma-1-\lambda_{-,\gamma}\right)\left[\lambda_{+,\gamma}-2(\gamma-1)\right]} 
     +O_{x\to+\infty}(x e^{\lambda_{-,\gamma}x}).
    \]
    Taking a derivative yields
    \bee
    &&\partial_x\left(R^{-}_{\gamma}-\mathcal{P}^{x_*}_{\phi'_{\gamma}}\right)\left(|\phi_\gamma|^{\frac{\gamma}{\gamma-2}}\right)\\
    &=&-\phi_\gamma''(x)\left[\int^{x}_{-\infty} h_2(y) e^{-2(\gamma-1)y} |\phi_\gamma|^{\frac{\gamma}{\gamma-2}}(y) dy\right]-\phi_\gamma'(x) h_2(x) e^{-2(\gamma-1)x} |\phi_\gamma|^{\frac{\gamma}{\gamma-2}}(x)\\
    &=&\frac{-\lambda_{-,\gamma}|\phi^+_\gamma|^{\frac{\gamma}{\gamma-2}}}{2 \left(\gamma-1-\lambda_{-,\gamma}\right)\left[\lambda_{+,\gamma}-2(\gamma-1)\right]}-\frac{|\phi^+_\gamma|^{\frac{\gamma}{\gamma-2}}}{2 \left(\gamma-1-\lambda_{-,\gamma}\right)} 
     +O_{x\to+\infty}(x e^{\lambda_{-,\gamma}x}).
   \eee
    The leading order term cancels out and \eqref{finaneonoexp} is proved.\\

\noi{\bf step 4} Choice of $c(v)$.\\

 \noi\und{Fixed point formulation}. We solve \eqref{lienarinoing} by showing that the map 
\begin{align*}
T^{-}_{\gamma} v&\equiv R^{-}_{\gamma}\left[-a^2e^{2x}(\partial^2_x+\partial_x)\left(\phi_{\gamma}+v\right)-(c+\eps)|\phi_\gamma|^{\frac{\gamma}{\gamma-2}}\right]-(c+\eps)R^{-}_{\gamma}\left[\frac{\gamma \int_0^{1}\left(\phi_\gamma+s v\right)|\phi_\gamma+s v|^{\frac{4-\gamma}{\gamma-2}}ds }{\gamma-2}v\right]
\\&+R^{-}_{\gamma}\left[\frac{\int_0^1 \left[4(\gamma+2)|\phi_\gamma+sv|^{\frac{6-\gamma}{\gamma-2}}-2\gamma A_{\gamma} |\phi_\gamma+sv|^{\frac{4-\gamma}{\gamma-2}}\right](1-s)ds}{(\gamma-2)^2}  v^2\right].
\end{align*}
has a fixed point in $B(0,C_{\gamma} a^2)$ in the topology $E^{0,x_*}$ for $a$ sufficiently small and a well chosen $c(v)$.\\

\noi\und{Choice of $c(v)$}. We rewrite the previous equations as
\begin{align*}
T^{-}_{\gamma} v&=\left(R^{-}_{\gamma}-\mathcal{P}^{x_*}_{\phi'_{\gamma}}\right)\left[-a^2e^{2x}(\partial^2_x+\partial_x)\left(\phi_{\gamma}+v\right)-(c+\epsilon)|\phi_\gamma|^{\frac{\gamma}{\gamma-2}}\right]+Er(x)\\
&-(c+\epsilon)\left(R^{-}_{\gamma}-\mathcal{P}^{x_*}_{\phi'_{\gamma}}\right)\left[\frac{\gamma \int_0^{1}\left(\phi_\gamma+s v\right)|\phi_\gamma+s v|^{\frac{4-\gamma}{\gamma-2}}ds }{\gamma-2}v\right]-\epsilon \mathcal{P}^{x_*}_{\phi'_{\gamma}}\left(|\phi_\gamma|^{\frac{\gamma}{\gamma-2}}\right)
\\&+\left(R^{-}_{\gamma}-\mathcal{P}^{x_*}_{\phi'_{\gamma}}\right)\left[\frac{\int_0^1 \left[4(\gamma+2)|\phi_\gamma+sv|^{\frac{6-\gamma}{\gamma-2}}-2\gamma A_{\gamma} |\phi_\gamma+sv|^{\frac{4-\gamma}{\gamma-2}}\right](1-s)ds}{(\gamma-2)^2}  v^2\right].
\end{align*}
where 
\bee
&&Er(x)\\
&=&\mathcal{P}^{x_*}_{\phi'_{\gamma}}\left[-a^2e^{2x}(\partial^2_x+\partial_x)\left(\phi_{\gamma}+v\right)-c|\phi_\gamma|^{\frac{\gamma}{\gamma-2}}\right]-(c+\epsilon)\mathcal{P}^{x_*}_{\phi'_{\gamma}}\left[\frac{\gamma \int_0^{1}\left(\phi_\gamma+s v\right)|\phi_\gamma+s v|^{\frac{4-\gamma}{\gamma-2}}ds }{\gamma-2}v\right]
\\&+&\mathcal{P}^{x_*}_{\phi'_{\gamma}}\left[\frac{\int_0^1 \left[4(\gamma+2)|\phi_\gamma+sv|^{\frac{6-\gamma}{\gamma-2}}-2\gamma A_{\gamma} |\phi_\gamma+sv|^{\frac{4-\gamma}{\gamma-2}}\right](1-s)ds}{(\gamma-2)^2}  v^2\right].
\eee
We explicitely choose $c(v)$ to ensure $$Er(x)=0, \text{ for }x\geq 0.$$

\noi\und{Main term}.  We split
$$c=c_\gamma(x_*)+c_1(v)$$  with for $x\ge 0$:
\bee
&&\mathcal{P}^{x_*}_{\phi'_{\gamma}}\left[-a^2e^{2x}(\partial^2_x+\partial_x)\phi_{\gamma}-c_\gamma(x_*)|\phi_\gamma|^{\frac{\gamma}{\gamma-2}}\right]=0\\
&\LR&c_\gamma(x_*)=-\frac{\mathcal{P}^{x_*}_{\phi'_{\gamma}}\left[a^2e^{2x}(\partial^2_x+\partial_x)\phi_{\gamma}\right]}{\mathcal{P}^{x_*}_{\phi'_{\gamma}}\left(|\phi_\gamma|^{\frac{\gamma}{\gamma-2}}\right)}.
\eee
We infer by evaluating \eqref{ionoeinoe}, \eqref{finaneonoexp} at $x=x_*$ and using \eqref{esthtwto}:
\bee
c_\gamma(x_*)&=&-
\frac{a^2\left[1+(\gamma-2)\right]\left[\int^{+\infty}_{-\infty}  e^{2(2-\gamma)y} \phi_\gamma'(y)^2dy\right]h_2(x_*)\left[1+O\left(e^{(2+\lambda_{-,\gamma}-\l_{+,\gamma}) x_*}\right)\right]}{\left[\int^{+\infty}_{-\infty} \phi_\gamma'(y) e^{-2(\gamma-1)y}|\phi_\gamma|^{\frac{\gamma}{\gamma-2}}(y)dy\right]h_2(x_*)\left[1+O\left(e^{-\l_{+,\gamma} x_*}\right)\right]}\\
&=& -a^2c_\gamma\left[1+O\left(e^{(2+\lambda_{-,\gamma})x_*}\right)\right]
\eee
with $$c_\gamma=\frac{\left[1+(\gamma-2)\right]\int^{+\infty}_{-\infty}  e^{2(2-\gamma)y} \left[\phi_\gamma'(y)\right]^2dy}{\int^{+\infty}_{-\infty} \phi_\gamma'(y) e^{-2(\gamma-1)y}|\phi_\gamma|^{\frac{\gamma}{\gamma-2}}(y)dy}>0.$$
\noi\und{Bound for $c_1(v)$.} We claim 
\be
\label{eq: est c1}
\left|\bear
|c_1(v)|\leq  C_\gamma\left( a^2 \left\Vert v \right\Vert_{E^{2,x_*}_{-}}+\left\Vert v \right\Vert_{E^{0,x_*}_{-}}^2 \right)\\
|c_1(v)-c_1(w)|\leq C_{\gamma}\left[\left(\left\Vert v \right\Vert_{E^{0,x_*}_{-}}+\left\Vert w \right\Vert_{E^{0,x_*}_{-}}\right)\left\Vert v-w \right\Vert_{E^{0,x_*}_{-}}+a^2 \left\Vert v-w \right\Vert_{E^{2,x_*}_{-}}\right].
\ear\right.
\ee
\noi{\em Proof of \eqref{eq: est c1}}. By definition 
\bee
    &&c_1\int^{x_*}_{-\infty} \phi_\gamma'(y) e^{-2(\gamma-1)y} |\phi_\gamma|^{\frac{\gamma}{\gamma-2}}(y)dy\\
    &=&\int^{x_*}_{-\infty} \phi_\gamma'(y) e^{-2(\gamma-1)y}\left[-a^2e^{2y}(\partial^2_x+\partial_x)v-(c+\epsilon)\frac{\gamma \int_0^{1}\left(\phi_\gamma+s v\right)|\phi_\gamma+s v|^{\frac{4-\gamma}{\gamma-2}}ds }{\gamma-2}v\right]dy\\
    &+&\int^{x_*}_{-\infty} \phi_\gamma'(y) e^{-2(\gamma-1)y}\left[\frac{\int_0^1 \left[4(\gamma+2)|\phi_\gamma+sv|^{\frac{6-\gamma}{\gamma-2}}-2\gamma A_{\gamma} |\phi_\gamma+sv|^{\frac{4-\gamma}{\gamma-2}}\right](1-s)ds}{(\gamma-2)^2}  v^2\right]dy
\eee
yields
\bee
&&c_1\int^{x_*}_{-\infty} \phi_\gamma'(y) e^{-2(\gamma-1)y} \left[|\phi_\gamma|^{\frac{\gamma}{\gamma-2}}(y)+\frac{\gamma \int_0^{1}\left(\phi_\gamma+s v\right)|\phi_\gamma+s v|^{\frac{4-\gamma}{\gamma-2}}ds }{\gamma-2}v\right]dy\\
&=&\int^{x_*}_{-\infty} \phi_\gamma'(y) e^{-2(\gamma-1)y}\left[-a^2e^{2y}(\partial^2_x+\partial_x)v+(c_{\gamma}a^2+\epsilon)\frac{\gamma \int_0^{1}\left(\phi_\gamma+s v\right)|\phi_\gamma+s v|^{\frac{4-\gamma}{\gamma-2}}ds }{\gamma-2}v\right]dy\\
 &+&\int^{x_*}_{-\infty} \phi_\gamma'(y) e^{-2(\gamma-1)y}\left[\frac{\int_0^1 \left[4(\gamma+2)|\phi_\gamma+sv|^{\frac{6-\gamma}{\gamma-2}}-2\gamma A_{\gamma} |\phi_\gamma+sv|^{\frac{4-\gamma}{\gamma-2}}\right](1-s)ds}{(\gamma-2)^2}  v^2\right]dy.
\eee
We write
\bee
&&\int^{x_*}_{-\infty} \phi_\gamma'(y) e^{-2(\gamma-1)y} |\phi_\gamma|^{\frac{\gamma}{\gamma-2}}(y)dy\\
&=&\underbrace{\int^{+\infty}_{-\infty} \phi_\gamma'(y) e^{-2(\gamma-1)y} |\phi_\gamma|^{\frac{\gamma}{\gamma-2}}(y)dy}_{>0}+\int^{+\infty}_{x_*} \phi_\gamma'(y) e^{-2(\gamma-1)y} |\phi_\gamma|^{\frac{\gamma}{\gamma-2}}(y)dy\\
&=&\int^{+\infty}_{-\infty} \phi_\gamma'(y) e^{-2(\gamma-1)y} |\phi_\gamma|^{\frac{\gamma}{\gamma-2}}(y)dy+O\left(e^{\left[\lambda_{-,\gamma}-2(\gamma-1)\right]x_*}\right),
\eee
and estimate for $v\in B(0,1)\subset E^{0,x_*}_{-}$ 
\[
\int^{x_*}_{-\infty} \phi_\gamma'(y) e^{-2(\gamma-1)y} \frac{\gamma \int_0^{1}\left(\phi_\gamma+s v\right)|\phi_\gamma+s v|^{\frac{4-\gamma}{\gamma-2}}ds }{\gamma-2}vdy\leq C_{\gamma}\left\Vert v\right\Vert_{E^{0,x_*}_{-}}\int^{x_*}_{-\infty} \phi_\gamma'(y) e^{-2(\gamma-1)y}\frac{e^{(\gamma+2)y}}{1+e^{(\gamma+2)y}}dy,
\]
from which
\begin{align}\label{eq: c1 denom est}
    &\int^{x_*}_{-\infty} \phi_\gamma'(y) e^{-2(\gamma-1)y} \left[|\phi_\gamma|^{\frac{\gamma}{\gamma-2}}(y)+\frac{\gamma \int_0^{1}\left(\phi_\gamma+s v\right)|\phi_\gamma+s v|^{\frac{4-\gamma}{\gamma-2}}ds }{\gamma-2}v\right]dy\nonumber \\
    &=\int^{+\infty}_{-\infty} \phi_\gamma'(y) e^{-2(\gamma-1)y} |\phi_\gamma|^{\frac{\gamma}{\gamma-2}}(y)dy+O\left[e^{(\lambda_{-,\gamma}-2(\gamma-1))x_*}+\left\Vert v \right\Vert_{ E^{0,x_*}_{-}}\right].
\end{align}
This ensures that $c_1(v)$ is a  well defined smooth function of $v\in B(0,\epsilon_{\gamma})\subset E^{2,x_*}_{-}$ for $\epsilon_{\gamma}$ sufficiently small given by the formula
\begin{align*}
    c_1(v)&=\frac{\int^{x_*}_{-\infty} \phi_\gamma'(y) e^{-2(\gamma-1)y}\left[-a^2e^{2y}(\partial^2_x+\partial_x)v+(c_{\gamma}a^2+\epsilon)\frac{\gamma \int_0^{1}\left(\phi_\gamma+s v\right)|\phi_\gamma+s v|^{\frac{4-\gamma}{\gamma-2}}ds }{\gamma-2}v\right]dy}{\int^{x_*}_{-\infty} \phi_\gamma'(y) e^{-2(\gamma-1)y} \left[|\phi_\gamma|^{\frac{\gamma}{\gamma-2}}(y)+\frac{\gamma \int_0^{1}\left(\phi_\gamma+s v\right)|\phi_\gamma+s v|^{\frac{4-\gamma}{\gamma-2}}ds }{\gamma-2}v\right]dy}\\
    &+\frac{\int^{x_*}_{-\infty} \phi_\gamma'(y) e^{-2(\gamma-1)y}\left[\frac{\int_0^1 [4(\gamma+2)|\phi_\gamma+sv|^{\frac{6-\gamma}{\gamma-2}}-2\gamma A_{\gamma} |\phi_\gamma+sv|^{\frac{4-\gamma}{\gamma-2}}](1-s)ds}{(\gamma-2)^2}  v^2\right]dy}{\int^{x_*}_{-\infty} \phi_\gamma'(y) e^{-2(\gamma-1)y} \left[|\phi_\gamma|^{\frac{\gamma}{\gamma-2}}(y)+\frac{\gamma \int_0^{1}\left(\phi_\gamma+s v\right)|\phi_\gamma+s v|^{\frac{4-\gamma}{\gamma-2}}ds }{\gamma-2}v\right]dy}.
\end{align*}
We estimate
\begin{align*}
\left\vert \int^{x_*}_{-\infty} \phi_\gamma'(y) e^{-2(\gamma-1)y}\left[-a^2e^{2y}(\partial^2_x+\partial_x)v\right]dy\right\vert&\leq C_{\gamma}a^2  \int^{x_*}_{-\infty} \phi_\gamma'(y) e^{-2(\gamma-1)y}e^{2y}\frac{e^{(\gamma+2)y}}{1+e^{(\gamma+2)y}} dy\left\Vert v \right\Vert_{E^{2,x_*}_{-}}\\
&\leq C_{\gamma}a^2 \left\Vert v \right\Vert_{E^{2,x_*}_{-}},
\end{align*}
and 
\[
\left\vert \int^{x_*}_{-\infty}  \phi_\gamma'(y) e^{-2(\gamma-1)y} (c_{\gamma}a^2+\epsilon)\frac{\gamma \int_0^{1}\left(\phi_\gamma+s v\right)|\phi_\gamma+s v|^{\frac{4-\gamma}{\gamma-2}}ds }{\gamma-2}vdy\right\vert\leq C_{\gamma}(a^2+|\epsilon|) \left\Vert v \right\Vert_{E^{0,x_*}_{-}},
\]
as well as
\[
\left\vert \int^{x_*}_{-\infty} \phi_\gamma'(y) e^{-2(\gamma-1)y}\left[\frac{\int_0^1 \left[4(\gamma+2)|\phi_\gamma+sv|^{\frac{6-\gamma}{\gamma-2}}-2\gamma A_{\gamma} |\phi_\gamma+sv|^{\frac{4-\gamma}{\gamma-2}}\right](1-s)ds}{(\gamma-2)^2}  v^2\right]dy\right\vert \leq C_{\gamma} \left\Vert v \right\Vert_{E^{0,x_*}_{-}}^2
\]
which concludes the proof of the first estimate in \eqref{eq: est c1}. The Lipschitz regularity is estimated similarly, this is left to the reader.\\

\noi{\bf step 5} Closing the fixed point. We claim that for $v,w\in B(0,\epsilon_{\gamma})\subset E^{2,x_{*}}_{-}$,
\be
\label{eq: est T-}
\left|\bear
\left\Vert T^{-}_{\gamma}v \right\Vert_{E^{2,x_{*}}_{-}}\leq C_{\gamma}\left(a^2+|\epsilon| e^{\lambda_{+,\gamma}x_*} +a^2 e^{2x_*}\left\Vert v \right\Vert_{E^{2,x_{*}}_{-}}+\left\Vert v \right\Vert_{E^{0,x_{*}}_{-}}^2\right)\\
\left\Vert T^{-}_{\gamma}v- T^{-}_{\gamma}w \right\Vert_{E^{2,x_{*}}_{-}}\leq C_{\gamma}\left[a^2 e^{2x_*}+(\left\Vert v\right\Vert_{E^{2,x_{*}}_{-}}+\left\Vert w\right\Vert_{E^{2,x_{*}}_{-}})\right]\left\Vert v -w\right\Vert_{E^{2,x_{*}}_{-}},
\ear\right.
\ee
Assume \eqref{eq: est T-}, then given $\eps$ in the range \eqref{eq: eps bound}, a standard application of the Banach fixed point Theorem for all $a<a_*(\gamma)$ sufficiently small ensures the existence of a fixed point in a ball $B\left(0,C(\Lambda_*) a^2\right)$ of $E^{2,x_{*}}_-$ yielding the inner solution of Lemma \ref{inenrsotuon}. Moreover by construction, we estimate at $x=x_*$ using \eqref{ionoeinoe}, \eqref{finaneonoexp}, \eqref{esthtwto}, \eqref{eq: est c1}:
\bee
v(x_*)&=&-\epsilon \mathcal{P}^{x_*}_{\phi'_{\gamma}}\left(|\phi_\gamma|^{\frac{\gamma}{\gamma-2}}\right)+\left(R^{-}_{\gamma}-\mathcal{P}^{x_*}_{\phi'_{\gamma}}\right)\left[-a^2e^{2x}(\partial^2_x+\partial_x)\phi_{\gamma}-(c+\eps)|\phi_\gamma|^{\frac{\gamma}{\gamma-2}}\right]+O\left(a^2 a^2 e^{2x_*}\right)\\
&=& -\eps \left[\int^{+\infty}_{-\infty} \phi_\gamma'(y) e^{-2(\gamma-1)y}|\phi_\gamma|^{\frac{\gamma}{\gamma-2}}(y)dy\right]\left[\frac{1+O(e^{\lambda_{-,\gamma}x_*})}{2d_\gamma \left(\gamma-1-\lambda_{-,\gamma}\right)} e^{\lambda_{+,\gamma}  x_*}\right]\\
&-&\left[c_\gamma(x_*)+c_1(v)+\eps\right]\left[\frac{|\phi^+_\gamma|^{\frac{\gamma}{\gamma-2}}}{2 \left(\gamma-1-\lambda_{-,\gamma}\right)\left[\lambda_{+,\gamma}-2(\gamma-1)\right]}   +O\left(\la x\ra e^{\lambda_{-,\gamma} x_*}\right)\right]+O(a^2\eta_*)\\
&=& v_{1,\gamma}a^2-v_{2,\gamma}e^{\l_{+,\gamma}x_*}\eps+O\left(a^2\eta_*\right)
\eee
with $$\left|\bear
v_{1,\gamma}=\frac{c_\gamma(x_*)|\phi^+_\gamma|^{\frac{\gamma}{\gamma-2}}}{2 \left(\gamma-1-\lambda_{-,\gamma}\right)\left[\lambda_{+,\gamma}-2(\gamma-1)\right]}\\
v_{2,\gamma}=\frac{\int^{+\infty}_{-\infty} \phi_\gamma'(y) e^{-2(\gamma-1)y}|\phi_\gamma|^{\frac{\gamma}{\gamma-2}}(y)dy}{2d_\gamma \left(\gamma-1-\lambda_{-,\gamma}\right)}>0.
\ear\right.
$$
similarly for the derivative
\bee
\pa_xv(x_*)&=&-\eps \left[\int^{+\infty}_{-\infty} \phi_\gamma'(y) e^{-2(\gamma-1)y}|\phi_\gamma|^{\frac{\gamma}{\gamma-2}}(y)dy\right]\left[\frac{1+O(e^{\lambda_{-,\gamma}x_*})}{2d_\gamma \left(\gamma-1-\lambda_{-,\gamma}\right)} \l_{+,\gamma}e^{\lambda_{+,\gamma}  x_*}\right]\\
&+& O\left(\left[c_\gamma(x_*)+c_1(v)+\eps\right]\la x_*\ra e^{\lambda_{-,\gamma} x_*}\right)+O\left(a^2\eta_*\right)\\
&=& -\l_{+,\gamma}v_{2,\gamma}e^{\l_{+,\gamma}x_*}\eps+O\left(a^2\eta_*\right)
\eee
which concludes the proof of \eqref{eq: exp v x*}.\\

\noi{\em Proof of \eqref{eq: est T-}}. The choice of $c(v)$ ensures
$$
\left|\bear
T^{-}_{\gamma} v=F_1+F_2+L^{-}_{1}v+L^{-}_{2}(v)v+N_1(v)v^2\\
F_1(v)=\left(R^{-}_{\gamma}-\mathcal{P}^{x_*}_{\phi'_{\gamma}}\right)\left[-a^2e^{2x}(\partial^2_x+\partial_x)\phi_{\gamma}-(c(v)+\epsilon)|\phi_\gamma|^{\frac{\gamma}{\gamma-2}}\right]\\
F_2=\epsilon \mathcal{P}^{x_*}_{\phi'_{\gamma}}\left(|\phi_\gamma|^{\frac{\gamma}{\gamma-2}}\right)\\
L^{-}_{1}v=-a^2\left(R^{-}_{\gamma}-\mathcal{P}^{x_*}_{\phi'_{\gamma}}\right)\left[e^{2x}(\partial^2_x+\partial_x)v\right]\\ L^{-}_{2}(v)v=-\left[c_\gamma(x_*)+c_1(v)+\eps\right]\left(R^{-}_{\gamma}-\mathcal{P}^{x_*}_{\phi'_{\gamma}}\right)\left[\frac{\gamma \int_0^{1}\left(\phi_\gamma+s v\right)|\phi_\gamma+s v|^{\frac{4-\gamma}{\gamma-2}}ds }{\gamma-2} v\right]\\
N_1(v)v^2=\left(R^{-}_{\gamma}-\mathcal{P}^{x_*}_{\phi'_{\gamma}}\right)\left[\frac{\int_0^1 \left[4(\gamma+2)|\phi_\gamma+sv|^{\frac{6-\gamma}{\gamma-2}}-2\gamma A_{\gamma} |\phi_\gamma+sv|^{\frac{4-\gamma}{\gamma-2}}\right](1-s)ds}{(\gamma-2)^2}   v^2\right].
\ear\right.
$$
We estimate each term separately for $v,w\in B(0,\epsilon_{\gamma})\subset E^{2,x_{*}}_{-}$.\\

\noi\und{Forcing and $L^{-}_{1}$.} A simple application of \eqref{cointorissbis} yields:
$$\left|\bear
\left\Vert F_1(v) \right\Vert_{E^{2,x_*}_{-}}\leq C_{\gamma} a^2\left[1+e^{(2+\lambda_{-,\gamma})x_*}\right]\\
\left\Vert F_1(v)-F_1(w) \right\Vert_{E^{2,x_*}_{-}}\leq C_{\gamma}\left[\left(\left\Vert v \right\Vert_{E^{0,x_*}_{-}}+\left\Vert w \right\Vert_{E^{0,x_*}_{-}}\right)\left\Vert v-w \right\Vert_{E^{0,x_*}_{-}}+a^2 \left\Vert v-w \right\Vert_{E^{2,x_*}_{-}}\right]\\
\left\Vert F_2 \right\Vert_{E^{2,x_*}_{-}}\leq C_{\gamma} \epsilon e^{\lambda_{+,\gamma}x_*}\\
\left\Vert L^{-}_{1}v \right\Vert_{E^{2,x_{*}}_{-}}\leq C_{\gamma}a^2 e^{2x_*}\left\Vert v \right\Vert_{E^{2,x_{*}}_{-}}.
\ear\right.
$$
\noi\und{$L^{-}_{2}(v)$ term} Using \eqref{cointorissbis}, \eqref{eq: est c1} yields
\begin{align*}
\left\Vert L^{-}_{2}(v)v \right\Vert_{E^{2,x_{*}}_{-}}\leq C_{\gamma}|c(v)|\left\Vert v \right\Vert_{E^{0,x_{*}}_{-}}&\leq  C_{\gamma}\left(a^2+a^2 e^{2x_*}\left\Vert v \right\Vert_{E^{2,x_{*}}_{-}}+a^2 \left\Vert v \right\Vert_{E^{0,x_{*}}_{-}}+\left\Vert v \right\Vert_{E^{0,x_{*}}_{-}}^2\right)\left\Vert v \right\Vert_{E^{0,x_{*}}_{-}}\\
&\leq  C_{\gamma}\left(a^2 e^{2x_*}+\left\Vert v \right\Vert_{E^{0,x_{*}}_{-}}^2\right)\left\Vert v \right\Vert_{E^{0,x_{*}}_{-}}.
\end{align*}
Then 
\bee
&&\left[L^{-}_{2}(v)-L^{-}_{2}(w)\right]v=-\left[c(v)-c(w)\right]\left(R^{-}_{\gamma}-\mathcal{P}^{x_*}_{\phi'_{\gamma}}\right)\left[\frac{\gamma \int_0^{1}\left(\phi_\gamma+s v\right)|\phi_\gamma+s v|^{\frac{4-\gamma}{\gamma-2}}ds }{\gamma-2} v\right]\\
&+&c(w)\left(R^{-}_{\gamma}-\mathcal{P}^{x_*}_{\phi'_{\gamma}}\right)\left[\frac{\gamma \int_0^{1}\left(\phi_\gamma+s v\right)|\phi_\gamma+s v|^{\frac{4-\gamma}{\gamma-2}}-\left(\phi_\gamma+s w\right)|\phi_\gamma+s v|^{\frac{4-\gamma}{\gamma-2}}ds }{\gamma-2} w\right],
\eee
which using \eqref{eq: est c1} ensures
\begin{align*}
   &\left\Vert [c(v)-c(w)]\left(R^{-}_{\gamma}-\mathcal{P}^{x_*}_{\phi'_{\gamma}}\right)\left[\frac{\gamma \int_0^{1}\left(\phi_\gamma+s v\right)|\phi_\gamma+s v|^{\frac{4-\gamma}{\gamma-2}}ds }{\gamma-2} v\right]\right\Vert_{E^{2,x_*}_{-}} \\
   &\leq C_{\gamma}\left[\left(\left\Vert v \right\Vert_{E^{0,x_*}_{-}}+\left\Vert w \right\Vert_{E^{0,x_*}_{-}}\right)\left\Vert v-w \right\Vert_{E^{0,x_*}_{-}}+a^2 \left\Vert v-w \right\Vert_{E^{2,x_*}_{-}}\right)\left\Vert v \right\Vert_{E^{0,x_{*}}_{-}},
\end{align*}
and the smoothness of the function $x|x|^{\frac{4-\gamma}{\gamma-2}}$ gives
\begin{align*}
    &\left\Vert c(w)\left(R^{-}_{\gamma}-\mathcal{P}^{x_*}_{\phi'_{\gamma}}\right)\left[\frac{\gamma \int_0^{1}\left(\phi_\gamma+s v\right)|\phi_\gamma+s v|^{\frac{4-\gamma}{\gamma-2}}-\left(\phi_\gamma+s w\right)|\phi_\gamma+s w|^{\frac{4-\gamma}{\gamma-2}}ds }{\gamma-2} w\right]\right\Vert_{E^{2,x_*}_{-}}\\
    &\leq C_\gamma\left( a^2 \left\Vert w \right\Vert_{E^{2,x_*}_{-}}+\left\Vert w \right\Vert_{E^{0,x_*}_{-}}^2 \right)\left\Vert v-w \right\Vert_{E^{0,x_{*}}_{-}}.
\end{align*}
The collection of above bounds yields
\begin{align*}
\left\Vert \left[L^{-}_{2}(v)-L^{-}_{2}(w)\right]v \right\Vert_{E^{2,x_{*}}_{-}}&\leq C_{\gamma}\left[\left(\left\Vert v \right\Vert_{E^{0,x_*}_{-}}+\left\Vert w \right\Vert_{E^{0,x_*}_{-}}\right)\left\Vert v-w \right\Vert_{E^{0,x_*}_{-}}+a^2 \left\Vert v-w \right\Vert_{E^{2,x_*}_{-}}\right]\left\Vert v \right\Vert_{E^{0,x_{*}}_{-}}\\
&+C_\gamma\left( a^2 \left\Vert v \right\Vert_{E^{2,x_*}_{-}}+\left\Vert v \right\Vert_{E^{0,x_*}_{-}}^2 \right)\left\Vert v-w \right\Vert_{E^{0,x_{*}}_{-}}.
\end{align*}
\noi\und{$N_1(v)$} From \eqref{cointorissbis}:
\[
\left\Vert N_1(v)v^2 \right\Vert_{E^{2,x_{*}}_{-}}\leq C_{\gamma}\left\Vert v \right\Vert_{E^{0,x_{*}}_{-}}^2,
\]
and the smoothness of the functions $x|x|^{\frac{4-\gamma}{\gamma-2}}$ and $x|x|^{\frac{6-\gamma}{\gamma-2}}$ ensures
\[\left\Vert N_1(v)v^2-N_1(w)w^2 \right\Vert_{E^{2,x_{*}}_{-}}\leq C_{\gamma}(\left\Vert v \right\Vert_{E^{0,x_{*}}_{-}}+\left\Vert w \right\Vert_{E^{0,x_{*}}_{-}})\left\Vert v -w\right\Vert_{E^{0,x_{*}}_{-}}.
\]
which concludes the proof of \eqref{eq: est T-}.
\end{proof}


\subsection{Trapping near $\phi_\gamma^+$}


We conclude the proof of Proposition \ref{thm: exist of sol} by showing that all inner solutions of Lemma \ref{inenrsotuon} remains trapped in a tube near $\phi_{\gamma,+}$, and a special choice of $\eps(a)$ ensures a cancellation which produces the asymptotic \eqref{ieononveoompmev} near $+\infty$. Let us insist that the full stabilization mechanism is created by the $a$ deformation of the equation which pushes for us.

\begin{lemma}[Trapping of the inner solution]
\label{elnannn}
There exists $\eps(\gamma,a)$ in the range \eqref{eq: eps bound} such that the inner solution of Lemma \ref{elnannn} can be continued as a global solution to \eqref{eq: closed ODE} with the asymptotic behaviour \eqref{ieononveoompmev}.
\end{lemma}

\begin{proof}[Proof of Lemma \ref{elnannn}] The key is the structure of the $a$ independent linearized operator near $\phi_{\gamma}^+$ which nicely enough has an explicit scattering structure thanks to the repulsivity of the potential.\\

\noi{\bf step 1} The linearized flow. We solve \eqref{eq: closed ODE} from $x_*$. We look for the solution in the form $$\phi=\phi^+_{\gamma}+w$$ so that $w$ solves 
\bee
   &&\partial^2_x w -2(\gamma -1)\partial_xw+\gamma(\gamma-2) w+a^2e^{2x}(\partial^2_x+\partial_x)w-\frac{\gamma+2}{\gamma-2}\phi^+_{\gamma}|\phi^+_{\gamma}|^{\frac{6-\gamma}{\gamma-2}}w+\frac{A_{\gamma}\gamma}{\gamma-2}\phi^+_{\gamma}|\phi^+_{\gamma}|^{\frac{4-\gamma}{\gamma-2}}w\\
   &=&-(c(a,x_*)+\epsilon)|\phi^+_{\gamma}+w|^{\frac{\gamma}{\gamma-2}}\\
   &+&\frac{ \int_0^1 [4(\gamma+2)|\phi^+_{\gamma}+sw|^{\frac{6-\gamma}{\gamma-2}}-2\gamma \left[A_{\gamma}+c(a,x_*)\right] |\phi^+_{\gamma}+sw|^{\frac{4-\gamma}{\gamma-2}}](1-s)ds}{(\gamma-2)^2}w^2.
\eee
Let$$\left|\bear
X=x-\log\left(\frac{1}{a}\right)\LR a^2e^{2x}=e^{2X}\\
X\in [X_*,+\infty), \ \ X_*=\log(\eta_*)\ll -1
\ear\right.
$$ where we used \eqref{conosnmslalnes},  this is equivalently:
\begin{equation}
\label{eq: l plus mu}
\left|\bear
L^{+}_{\gamma}w=G\\
G=N_1^+(w)+N_2^+(w)w^2\\
L^{+}_{\gamma}w=\partial^2_X w +\left[1-\frac{1+2(\gamma -1)}{1+e^{2X}}\right]\partial_X w-\frac{V_+}{1+e^{2X}}w\\
V_+=-\left[\gamma(\gamma-2)-\frac{\gamma+2}{\gamma-2}\phi_{\gamma}^+|\phi^+_\gamma|^{\frac{6-\gamma}{\gamma-2}}+\frac{A_{\gamma}\gamma}{\gamma-2}\phi^+_\gamma|\phi^+_\gamma|^{\frac{4-\gamma}{\gamma-2}}\right]\\
N_1^+(w)=-\frac{\left[c(a,x_*)+\epsilon\right]|\phi^+_{\gamma}+w|^{\frac{\gamma}{\gamma-2}}}{1+e^{2X}}\\
N_2^+(w)w^2=\frac{ \int_0^1 \left[4(\gamma+2)|\phi^+_{\gamma}+sw|^{\frac{6-\gamma}{\gamma-2}}-2\gamma \left[A_{\gamma}+c(a,x_*)+\epsilon\right]|\phi^+_{\gamma}+sw|^{\frac{4-\gamma}{\gamma-2}}\right](1-s)ds}{(\gamma-2)^2(1+e^{2X})}w^2.
\ear\right.
\end{equation}
The boundary condition at $X_*$ generated by the inner solution satisfies from \eqref{eq: exp v x*}, \eqref{waiofneonei}:
\be
\label{esigonionvnoenione}
\left|\bear
   w(X_*)=v_{1,\gamma}a^2-v_{2,\gamma}e^{\l_{+,\gamma}x_*}\eps+O\left(a^2\eta_*\right)\\
    \partial_X w(X_*)=-\l_{+,\gamma}v_{2,\gamma}e^{\l_{+,\gamma}x_*}\eps+O\left(a^2\eta_*\right).
\ear\right.
\ee

\noi{\bf step 2} Fundamental basis. We claim that the $a$ independent linearized operator $L_{\gamma}^+$ admits a basis of fundamental solutions $f_1(X),f_2(X)$ with 
\be
\label{eq: exp f1}
    f_1(X)=\left|\bear
    f_{1,\gamma}e^{\lambda_{-,\gamma}X}+O_{X\to-\infty}\left[e^{(2+\lambda_{-,\gamma})X}\right], \ \ f_{1,\gamma}>0\\
    \left[1+O_{X\to +\infty}(e^{-2X})\right]e^{-X}
    \ear\right.
\ee
\be
\label{eq: exp f2}
    f_2(X)=\left|\bear
    f_{2,\gamma} e^{\lambda_{+,\gamma}X}+O_{X\to-\infty}\left[e^{(2+\lambda_{+,\gamma})X}\right], \ \ f_{2,\gamma} > 0\\
    1 -f_{3,\gamma}e^{-X}+O(e^{-2X}), \ \ f_{3,\gamma}>0
    \ear\right. 
\ee
and similarly for derivatives.\\

\noi{\em Proof of \eqref{eq: exp f1}, \eqref{eq: exp f2}}. Let
$$
\left|\bear \rho=e^{X} \left(1+e^{-2X}\right)^{\frac{1+2(\gamma -1)}{2}}\\
\frac{\pa_X\rho}{\rho}=1-\frac{1+2(\gamma -1)}{2}\frac{2e^{-2X}}{1+e^{-2X}}=1-\frac{1+2(\gamma -1)}{1+e^{2X}}
\ear\right.
$$
we recall \eqref{fenoinoenineovnoive} and freeze the Wronskian 
\be
\label{vneinoenvionioveninoevoenv}
\frac{\pa_XW}{W}=-\left[1-\frac{1+2(\gamma -1)}{1+e^{2X}}\right]\Leftarrow W=\frac1{\rho}= e^{-X} \left(1+e^{-2X}\right)^{-\frac{1+2(\gamma -1)}{2}}.
\ee

\noi\und{Basis at $\pm \infty$}. An elementary perturbation argument ensures the existence of two basis of fundamental solutions with the asymptotic behaviour \eqref{eq: exp f1}, \eqref{eq: exp f2} as $X\to -\infty$.\\

\noi\und{Conjuguation.} We integrate the homogeneous equation:
\bea
\label{heioonenioe}
 L^{+}_{\gamma}f=0&\LR&
\frac{1}{\rho}\partial_X\left(\rho\partial_X f\right)-\frac{V_+f}{1+e^{2X}}=0\LR\partial_X\left(\rho\partial_X f\right)f=\frac{V_+\rho }{1+e^{2X}}f^2\\
\nonumber&\LR&
\partial_X\left(\rho f \partial_X f \right)=\rho\left(\partial_X f\right)^2+\frac{V_+\rho f^2}{1+e^{2X}}.
\eea
Let $g(\phi)=\gamma(\gamma-2) \phi-|\phi|^{\frac{\gamma+2}{\gamma-2}}+A_{\gamma}|\phi|^{\frac{\gamma}{\gamma-2}}$, then  by definition $\phi^{+}_{\gamma}$ is the unique positive solution of $g(\phi)=0$ and $\lim_{\phi\to+\infty}g(\phi)=-\infty$, and hence $g'(\phi^{+}_{\gamma})<0$ which yields
\be
\label{dinoenonnv}
V_+=-\left[\gamma(\gamma-2) -\frac{\gamma+2}{\gamma-2}\phi^+_\gamma|\phi^+_\gamma|^{\frac{6-\gamma}{\gamma-2}}+\frac{A_{\gamma}\gamma}{\gamma-2}\phi_\gamma^+|\phi^+_\gamma|^{\frac{4-\gamma}{\gamma-2}}\right]=-g'(\phi^{+}_{\gamma})>0.
\ee

\noi\und{$f_1$ solution}. We integrate \eqref{heioonenioe} from $+\infty$:
\be
\label{veinonveoinnevnoev}
 \partial_X\left(\rho\partial_X f_1\right)=\frac{V_+\rho f_1}{1+e^{2X}}\Leftarrow \left|\bear
\pa_Xf_1=-\frac{1}{\rho}\left(1+\int_X^{+\infty}\frac{V_+\rho f_1}{1+e^{2Y}}dY\right)\\
f_1=-\int_X^{+\infty}\pa_X f_1 dY.
\ear\right.
\ee
An elementary fixed point argument yields the existence of a solution
with the behaviour $$f_1(X)=\left[1+O_{X\to +\infty}(e^{-2X})\right]e^{-X}.$$ The repulsivity of the potential \eqref{dinoenonnv} and an elementary bootstrap from $+\infty$ ensure $f_1>0, f'_1<0$ which forces the asymptotic behaviour on the left 
\be
\label{vionievnoinieveniov}
f_1(X)=f_1^{\gamma}\left[1+O_{X\to-\infty}(e^{2X})\right] e^{\lambda_{-,\gamma}X}, \ \ f_1^{\gamma}>0.
\ee

\noi\und{$f_2$} We enforce the normalization \eqref{vneinoenvionioveninoevoenv}:
\bee
&&{\rm Wr}(f_1,f_2)=\frac{1}{\rho}\LR\left(\frac{f_2}{f_1}\right)'=\frac{ e^{-X} \left(1+e^{-2X}\right)^{-\frac{1+2(\gamma -1)}{2}}}{f_1^2}\\
&\Leftarrow& f_2=  f_1\int_{-\infty}^X\frac{e^{-Y} \left(1+e^{-2Y}\right)^{-\frac{1+2(\gamma -1)}{2}}}{f_1^2}dY.
\eee
Recalling \eqref{febionoienionoef}, the behaviour \eqref{eq: exp f2} near $-\infty$ follows from \eqref{vionievnoinieveniov}. We now study the asymptotic behaviour near $+\infty$.  Let $$\left|\bear
H_1=\int_X^{+\infty}\frac{V_+ \rho f_1}{1+e^{2Y}}dY>0\\
H_1=O_{X\to +\infty}(e^{-2X})
\ear\right.
$$ then from \eqref{veinonveoinnevnoev}
$$\pa_Xf_1=-\frac{1+H_1}{\rho}\LR\frac{1}{\rho}=-\pa_Xf_1-\frac{H_1}{\rho}$$ so that
  \bee
 f_2&=& f_1\int_{-\infty}^X\frac{e^{-Y} \left(1+e^{-2Y}\right)^{-\frac{1+2(\gamma -1)}{2}}}{f_1^2}dY= f_1\int_{-\infty}^X\frac{1}{\rho f_1^2}dY\\
 &=&  f_1\int_{-\infty}^X\frac{1}{f_1^2}\left(-\pa_Xf_1-\frac{H_1}{\rho}\right)dY= 1-f_1\int_{-\infty}^X\frac{H_1}{\rho f_1^2}dY
\eee
which yields the expansion \eqref{eq: exp f2} with $f_{3,\gamma}=\int_{-\infty}^{+\infty}\frac{H_1}{\rho f_1^2}dY>0.$\\

\noi{\bf step 3} Resolvent estimate. From \eqref{variontocnstante}, the resolvent of $L^+_{\gamma}$ with data at $X_*$ is 
$$\left|\bear
R^+_{\gamma}(X)=R^+_{hom,\gamma}f(X)+R^+_{inh,\gamma}f(X)\\
R^+_{hom,\gamma}f(X)=K^*_1f_1+K^*_2f_2\\
R^+_{inh,\gamma}f(X)=\int^{X}_{X_*}\left[f_1(Y)f_2(X)-f_1(X)f_2(Y)\right]f(Y)e^{Y} \left(1+e^{-2Y}\right)^{\frac{1+2(\gamma -1)}{2}}dY
\ear\right.
$$
where $(K_1^*, K_2^*)$ solve
\be
\label{enoeonoelkstar}
\left|\bear
    K_1^*f_1(X_*)+K_2^*f_2(X_*)=f(X_*)\\
    K_1^*f'_1(X_*)+K_2^*f'_2(X_*)=f'(X_*)
    \ear\right.\LR
\left|\bear K_1^*=\frac{f(X_*)f'_2(X_*)-f'(X_*)f_2(X_*)}{W(X_*)}\\
K_2^*=\frac{f_1(X_*)f'(X_*)-f_1'(X_*)f(X_*)}{W(X_*)}.
\ear\right.
\ee
We introduce the space $E^{0,X_*}_{+}$ defined by the norm
   \[
   \left\Vert f  \right\Vert_{E^{0,X_*}_{+}}=\left\Vert (1+e^{2X})  f  \right\Vert_{L^\infty([X_*,+\infty))},
   \]
   and the spaces $F^{0,X_*}_{+}$ and $F^{1,X_*}_{+}$ defined by 
  $$\left|\bear
     \left\Vert f  \right\Vert_{F^{0,X_*}_{+}}=\left\Vert   f  \right\Vert_{L^\infty([X_*,+\infty))}\\
   \left\Vert f  \right\Vert_{F^{1,X_*}_{+}}=\left\Vert   f  \right\Vert_{L^\infty([X_*,+\infty))}+\left\Vert (1+e^{X}) \partial_X f  \right\Vert_{L^\infty([X_*,+\infty))}.
   \ear\right.
   $$
   We claim 
   \be
\label{gionioneonenog}
\left|\bear
|K_1^*|\le  \eta_* a^2\\
K_2^*=f_{1,\gamma}e^{-\l_{+,\gamma}X_*}\left[1+O(\eta_*)\right]\left[\l_{-,\gamma}v_{1,\gamma}a^2-\left(\l_{+,\gamma}-\l_{-,\gamma}\right)v_{2,\gamma}\left[1+O(\eta_*)\right]e^{\l_{+,\gamma}x_*}\eps+O\left(a^2\eta_*\right)\right]\\
\ear\right.
\ee
and the resolvent estimates
   \be
   \label{resoestimatebis}
   \left|\bear
   \left\Vert R^+_{inh,\gamma} f\right\Vert_{F^{1,X_*}_{+}}\leq C_{\gamma} \left\Vert f  \right\Vert_{E^{0,X_*}_{+}}\\
   \left\Vert R^+_{hom,\gamma} w\right\Vert_{F^{1,X_*}_{+}}\leq C_\gamma (1+\Lambda_*)a^2.
   \ear\right.
    \ee

 \noi{\em Proof of \eqref{gionioneonenog}}. We compute
$$W(X_*)=e^{-X_*} \left(1+e^{-2X_*}\right)^{-\frac{1+2(\gamma -1)}{2}}=e^{-X_*}e^{\left[1+2(\gamma -1)\right]X_*}\left[1+O(e^{2X_*})\right]=e^{2(\gamma-1)X_*}\left[1+O(\eta_*)\right].$$

From \eqref{enoeonoelkstar}, \eqref{eq: exp f2}, \eqref{esigonionvnoenione}:
\bee
\left|K_1^*\right|&=&\left|\frac{w(X_*)f'_2(X_*)-w'(X_*)f_2(X_*)}{W(X_*)}=\frac{f_2(X_*)}{W(X_*)}\left[\frac{f'_2(X_*)}{f_2(X_*)}w(X_*)-w'(X_*)\right]\right|\\
&=&\Bigg|\frac{f_2(X_*)}{W(X_*)}\Bigg\{\left[v_{1,\gamma}a^2-v_{2,\gamma}e^{\l_{+,\gamma}x_*}\eps+O\left(a^2\eta_*\right)\right]\left[\l_{+,\gamma}+O(\eta_*)\right]\\
&-&\left[-\l_{+,\gamma}v_{2,\gamma}e^{\l_{+,\gamma}x_*}\eps+O\left(a^2\eta_*\right)\right]\Bigg\}\Bigg|\\
&=& \Bigg|\frac{f_2(X_*)}{W(X_*)}\Bigg\{\left[v_{1,\gamma}a^2-v_{2,\gamma}e^{\l_{+,\gamma}x_*}\eps+O\left(a^2\eta_*\right)\right]\left[\l_{+,\gamma}+O(\eta_*)\right]\\
&\lesssim & C_\gamma e^{\left[\l_+-2(\gamma-1)\right]X_*}a^2\leq \eta_* a^2.
\eee
similarly,
\bee
K_2^*&=& \frac{f_1(X_*)w'(X_*)-f_1'(X_*)w(X_*)}{W(X_*)}=\frac{f_1(X_*)}{W(X_*)}\left[w'(X_*)-\frac{f_1'(X_*)}{f_1(X_*)}w(X_*)\right]\\
&=& \frac{f_1(X_*)}{W(X_*)}\Bigg\{-\l_{+,\gamma}v_{2,\gamma}e^{\l_{+,\gamma}x_*}\eps+O\left(a^2\eta_*\right)\\
&-& \l_{-,\gamma}\left[1+O(e^{2X_*})\right]\left[v_{1,\gamma}a^2-v_{2,\gamma}e^{\l_{+,\gamma}x_*}\eps+O\left(a^2\eta_*\right)\right]\Bigg\}\\
&=& \frac{f_{1,\gamma}e^{\l_{-,\gamma}X_*}\left[1+O(e^{2X_*})\right]}{e^{2(\gamma-1)X_*}\left[1+O(e^{2X_*})\right]}\left[-\l_{-,\gamma}v_{1,\gamma}a^2-\left(\l_{+,\gamma}-\l_{-,\gamma}\right)v_{2,\gamma}\left[1+O(\eta_*)\right]e^{\l_{+,\gamma}x_*}\eps+O\left(a^2\eta_*\right)\right]\\
&=& f_{1,\gamma}e^{-\l_{+,\gamma}X_*}\left[1+O(\eta_*)\right]\left[-\l_{-,\gamma}v_{1,\gamma}a^2-\left(\l_{+,\gamma}-\l_{-,\gamma}\right)v_{2,\gamma}\left[1+O(\eta_*)\right]e^{\l_{+,\gamma}x_*}\eps+O\left(a^2\eta_*\right)\right]
\eee
which concludes the proof of \eqref{gionioneonenog}.\\

\noi{\em Proof of \eqref{resoestimatebis}.}   Assume wlog $\left\Vert f  \right\Vert_{E^{0,X_*}_{+}}=1$. We first estimate $R^+_{inh,\gamma} f$ in the sup norm. For $X\le 0$:
 $$
 \left|\bear
 \left|f_2(X)\int^{X}_{X_*}\frac{f_1(y)}{1+ e^{2y}}e^{y} \left(1+e^{-2y}\right)^{\frac{1+2(\gamma -1)}{2}}dy\right|\lesssim e^{\l_{+,\gamma}X}\int^{X}_{X_*}e^{\l_{-,\gamma}y} e^{-2(\gamma-1)y}dy\leq  C_\gamma \\
 \left|f_1(X)\int^{X}_{X_*}\frac{f_2(y)}{1+ e^{2y}}e^{y} \left(1+e^{-2y}\right)^{\frac{1+2(\gamma -1)}{2}}dy\right|\lesssim e^{\l_{-,\gamma}X}\int_{X_*}^Xe^{\l_{+,\gamma}y}e^{-2(\gamma-1)y}dy\leq C_\gamma,
 \ear\right.
 $$
and for $X\ge 0$:
 $$\left|\bear
 \left|f_2(X)\int^{X}_{X_*}\frac{f_1(y)}{1+ e^{2y}}e^{y} \left(1+e^{-2y}\right)^{\frac{1+2(\gamma -1)}{2}}dy\right|\leq C_\gamma\left[1+\int^{0}_{X}\frac{e^{-y}e^y}{1+e^{2y}} \left(1+e^{-2y}\right)^{\frac{1+2(\gamma -1)}{2}}dy\right]\leq C_\gamma\\
  \left|f_1(X)\int^{X}_{X_*}\frac{f_2(y)}{1+ e^{2y}}e^{y} \left(1+e^{-2y}\right)^{\frac{1+2(\gamma -1)}{2}}dy\right|\leq C_\gamma\left(1+e^{-x}\int_0^Xe^{-y}dy\right)\le C_\gamma
 \ear\right.
 $$
 where we recall that $C_\gamma$ depends on $\eta_*$. The decay of derivatives follows by writing  
 \bee
\partial_X R^+_{inh,\gamma}f(X)&=&\frac{f'_2(X)}{W_{\gamma}}\int^{X}_{X_*}f_1(Y)f(Y)e^{Y} \left(1+e^{-2Y}\right)^{\frac{1+2(\gamma -1)}{2}}dY\\
&-&\frac{f'_1(X)}{W_{\gamma}}\int^{X}_{X_*}f_2(Y)f(Y)e^{Y} \left(1+e^{-2Y}\right)^{\frac{1+2(\gamma -1)}{2}}dY
\eee
and using similarly \eqref{eq: exp f1}, \eqref{eq: exp f2}. From \eqref{gionioneonenog}, \eqref{eq: exp f1}, \eqref{eq: exp f2}, $$ \left\Vert R^+_{hom,\gamma} w\right\Vert_{F^{1,X_*}_{+}}\leq C_\gamma(1+\Lambda_*) a^2.$$

\noi{\bf step 5} Closing the fixed point. We solve \eqref{eq: l plus mu} as the fixed point problem in a ball $B(0,C(\Lambda_*) a^2)$ of $F^{0,x_*}_{+}$ for the map
$$
T_{\gamma}^+w=R^+_{hom,\gamma}+R^+_{inh,\gamma}\left[N_1^{+}(w)+N_2^{+}(w)w^2\right],
$$
We estimate in brute force recalling $\eps\ll a^2$:
$$\left|\bear
\left\Vert N^{+}_{1}(w)\right\Vert_{E^{0,x_*}_+}\leq C_{\gamma} (a^2+|\epsilon|) \left(1+\left\Vert w \right\Vert_{F^+_{0,x_*}}\right)\le C_{\gamma} a^2\left(1+\left\Vert w \right\Vert_{F^+_{0,x_*}}\right)\\
\left\Vert N^{+}_{1}(w)-N^{+}_{1}(w')\right\Vert_{E^{0,x_*}_+}\leq C_{\gamma} a^2 \left\Vert w-w' \right\Vert_{F^+_{0,x_*}}
\ear\right.
$$
and
$$\left|\bear
\left\Vert N_2^{+}(w)w^2\right\Vert_{E^{0,x_*}_+}\leq C_{\gamma} \left\Vert w  \right\Vert_{E^{0,x_*}_{+}}^2\\
  \left\Vert \left[N_2^{+}(w)-N_2^{+}(w')\right]w^2\right\Vert_{E^{0,x_*}_+}\leq C_{\gamma} \left\Vert w-w'  \right\Vert_{E^{0,x_*}_{+}} \left\Vert w  \right\Vert_{E^{0,x_*}_{+}}^2.
\ear\right.
$$
We conclude from \eqref{resoestimatebis}:
\[\left\Vert  T^{+}_{\gamma}w \right\Vert_{F^{1,x_*}_{+}}\leq C_{\gamma}\left(a^2+a^2\left\Vert w  \right\Vert_{E^{0,x_*}_{+}}+\left\Vert w  \right\Vert_{E^{0,x_*}_{+}}^2\right).\]
and 
\[\left\Vert  T^{+}_{\gamma}w -T^{+}_{\gamma}w'\right\Vert_{F^{1,x_*}_{+}}\leq C_{\gamma}\left(a^2 +\left\Vert w  \right\Vert_{E^{0,x_*}_{+}}+\left\Vert w'  \right\Vert_{E^{0,x_*}_{+}}\right)\left\Vert w-w'  \right\Vert_{E^{0,x_*}_{+}},\]
the result then follows by standard Banach fixed point in $B\left(0,C(\Lamdba_*)a^2\right)$ of $F^{0,x_*}_{+}$. The smoothness of the solution globally follows directly from standard elliptic regularity theory.\\

\noi{\bf step 6} Asympotic behaviour at $+\infty$ and choice of $\eps$.\\

\noi\und{Expansion at $+\infty$}.  Recall $G=N_1^{+}(w)+N_2^{+}(w)w^2$ so that using \eqref{eq: exp f1}, \eqref{eq: exp f2}:
\bee
w(x)&=&K^*_1f_1+K^*_2f_2+\int^{X}_{X_*}\left[f_1(Y)f_2(X)-f_1(X)f_2(Y)\right]G(Y)\rho dY\\
&=& \left[K^*_1-\int^{X}_{X_*}f_2G\rho dY\right]f_1+\left[K^*_2+\int^{X}_{X_*}f_1G\rho dY\right]f_2\\
&=& \left[K^*_1-\int_{X_*}^{+\infty}f_2G\rho dY+\int_X^{+\infty}f_2G\rho dY\right] \left[1+O_{X\to +\infty}(e^{-2X})\right]e^{-X}\\
&+&\left[K^*_2+\int_{X_*}^{+\infty}f_1G\rho dY-\int_{X}^{+\infty}f_1G\rho dY\right]\left[1 -f_{3,\gamma}e^{-X}+O(e^{-2X})\right]\\
&=& w_\infty+M_\infty e^{-X}+O\left(e^{-2X}\right)
\eee
with
$$w_{\infty}=K^*_2+\int_{X_*}^{+\infty}f_1G\rho dY$$ and 
\bee
M_\infty&=&K^*_1-\int^{+\infty}_{X_*}f_2G\rho dY-f_{3,\gamma}\left(K^*_2+\int^{+\infty}_{X_*}f_1G\rho dY\right)\\
& =& K^*_1-f_{3,\gamma} K^*_2-\int_{X_*}^{+\infty}(f_2+f_{3,\gamma}f_1)G\rho dY.
\eee

\noi\und{Cancellation of $M_\infty$.}
We have by construction of $w$
\[
\left\vert \int_{X_*}^{+\infty}(f_2+f_{3,\gamma}f_1)G\rho dY\right\vert\leq C_{\gamma}\left(a^2+|\epsilon|\right)\le C_\gamma a^2
\]
which together with \eqref{gionioneonenog} yields 
\bea
\label{vneoniveonvnenvo}
\non &&\left|M_\infty-f_{3,\gamma}f_{1,\gamma}e^{-\l_{+,\gamma}X_*}\left(\l_{+,\gamma}-\l_{-,\gamma}\right)v_{2,\gamma}\left[1+O(\eta_*)\right]e^{\l_{+,\gamma}x_*}\eps\right|\le C_\gamma a^2\\
\non &\LR&\left|M_\infty-\frac{f_{3,\gamma}f_{1,\gamma}\left(\l_{+,\gamma}+|\l_{-,\gamma}|\right)v_{2,\gamma}\left[1+O(\eta_*)\right]}{a^{\l_{+,\gamma}}}\eps\right|\le C_\gamma a^2\\
&\LR&\left|M_\infty-\frac{f_{3,\gamma}f_{1,\gamma}}{\eta_*^{\l_{+,\gamma}}}\left(\l_{+,\gamma}+|\l_{-,\gamma}|\right)v_{2,\gamma}\left[1+O(\eta_*)\right]e^{\l_{+,\gamma}x_*}\eps\right|\le C_\gamma a^2.
\eea
The continuity of the map $\eps\mapsto M_\infty$ is a simple application of the Banach fixed point theorem with parameters, and hence deforming $\eps$ in the range \eqref{eq: eps bound}, \eqref{vneoniveonvnenvo} ensures the existence of at least one\footnote{uniqueness of $\eps(a)$ follows again from Banach fixed point, and it can be computed as a function of $a$ independent of $x_*$ as it should be.}  $\eps(a)$ with $M_\infty=0$.\\

\noi\und{Proof of \eqref{ieononveoompmev}}. We conclude that the constructed solution satisfies $$w=w_\infty+O(e^{-2X}).$$  We now rewrite the equation \eqref{eq: l plus mu} as
\bee
&&L^{+}_{\gamma}w=\left[\partial^2_X +\left[1-\frac{1+2(\gamma -1)}{1+e^{2X}}\right]\partial_X -\frac{V_+}{1+e^{2X}}\right]w=G\\
&\LR&\left|\bear
(\partial^2_X +\pa_X)w=\tilde{G} \\
\tilde{G}=G+\frac{1+2(\gamma -1)}{1+e^{2X}}\pa_Xw+\frac{V_+}{1+e^{2X}}w
\ear\right.
\eee
and hence the a priori bound $|\tilde{G}|\lesssim e^{-2X}$ ensures that $w$ satisfies the integral equation:
\be
\label{vneionoevnoinevnoevn}
\pa_X\left(e^X\pa_Xw\right)=e^X\tilde{G}\Rightarrow \pa_X w=-e^X\int_{X}^{+\infty}\tilde{G}dY\Rightarrow w=w_\infty-\int_{X}^{+\infty}\pa_X wdX'.
\ee
Uniqueness of the solution to \eqref{vneionoevnoinevnoevn} in $[X_0,+\infty)$ with the a priori bound $|w-w_\infty|\lesssim e^{-2X}$ for $X_0\gg1 $ large enough is a simple sequence of the structure of $G$ and the smoothness of the non linearity. We now build a solution to \eqref{vneionoevnoinevnoevn} near $+\infty$ in the form $$w_{\rm new}=w_\infty+\sum_{j=1}^k d_je^{-2jX}+w_k$$ which for an explicit sequence $$(4j^2-2j)d_j=2j(2j-1)d_j=F\left[(d_\ell)_{1\le \ell \le j-1}\right]$$ and using an elementary fixed point in $[X_0,+\infty)$ yields a solution with $w_k=O(e^{-2(k+1)X})$. Hence $w=w_{\rm new}$ yields \eqref{ieononveoompmev}. This concludes the proof of Lemma \ref{elnannn} and Proposition \ref{thm: exist of sol}.
 
 \end{proof}

\end{document}